\documentclass[11pt,reqno]{amsart}
\usepackage[T1]{fontenc}
\usepackage{lmodern}
\usepackage{amsmath,amsfonts,amssymb,amsthm,mathtools,booktabs}
\usepackage{microtype}\usepackage{enumitem}
\usepackage{mathrsfs}
\usepackage[colorlinks=true,urlcolor=blue,citecolor=red,linkcolor=blue]{hyperref}
\newtheorem{theorem}{Theorem}[section]
\newtheorem{lemma}[theorem]{Lemma}
\newtheorem{proposition}[theorem]{Proposition}
\newtheorem{corollary}[theorem]{Corollary}
\theoremstyle{definition}

\numberwithin{equation}{section}
\newcommand{\R}{\mathbb R}
\newcommand{\Gal}{\operatorname{Gal}}
\newcommand{\eps}{\varepsilon}
\newcommand{\Lam}{\Lambda'}
\newcommand{\Q}{\mathbb Q}
\newcommand{\Z}{\mathbb Z}
\newcommand{\N}{\mathbb N}
\newcommand{\F}{\mathbb F}
\newcommand{\C}{\mathbb C}
\newcommand{\OO}{\mathcal O}
\newcommand{\HTP}{\operatorname{HTP}}
\newcommand{\Nm}{\operatorname{N}}
\newcommand{\Tr}{\operatorname{Tr}}
\newcommand{\disc}{\operatorname{disc}}
\newcommand{\lcm}{\operatorname{lcm}}

\newcommand{\Sq}{\operatorname{Sq}}
\newcommand{\Supp}{\operatorname{Supp}}
\newcommand{\im}{\operatorname{Im}}

\def\l{\left}
\def\r{\right}
\def\bg{\bigg}
\def\({\bg(}
\def\){\bg)}
\def\t{\text}
\def\f{\frac}
\def\mo{{\rm{mod}\ }}
\def\pmod#1{\ (\mo\ #1)}

\def\lcm{\rm{lcm}}

\def\sm{\setminus}

\DeclareMathOperator{\Sel}{Sel}
\newcommand{\Sha}{\operatorname{Sha}}

\begin{document}
\hbox{Preprint}
\medskip

\title[38 unknowns over the rationals]{38 unknowns for Hilbert's tenth problem over the rationals}
\author[Zhi-Wei Sun]{Zhi-Wei Sun}

\address{School of Mathematics, Nanjing
University, Nanjing 210093, People's Republic of China}
\email{zwsun@nju.edu.cn}

\keywords{Hilbert's tenth problem, rational solutions, bounded variables, elliptic curves, polynomial patterns in primes.
\newline \indent
Supported by the National Natural Science Foundation of China (grant no. 12371004).}

\begin{abstract}
Recently Hilbert's tenth problem over the field $\Q$ was solved negatively by OpenAI. We establish that $\HTP_{38}(\Q)$ is undecidable.
\end{abstract}

\maketitle

\section{Introduction}

Hilbert's tenth problem (HTP in short), posed in 1900, asks for an algorithm to decide for any polynomial equation
$$P(x_1,\ldots,x_n)=0$$
with integer coefficients whether it has integer solutions. This was finally solved negatively by
 Matiyasevich \cite{Mat} in 1970, based on the important work of Davis, Putnam and Robinson \cite{DPR} in 1961.
 
 For a ring $R$ containing $\Z$ and a fixed positive integer $n$, let $\HTP_n(R)$ be the problem of deciding whether an input polynomial in $\Z[x_1,\ldots,x_n]$ has a zero in $R^n$. 
Sun \cite[Corollary~1.2(i)]{Sun11} proved that $\HTP_{11}(\Z)$ is undecidable. 
More recently,  Sun \cite[Corollary~1.1]{Sun10} established the undecidability of $\HTP_9(\Z)$.  

Let $\mathbb Q$ be the field of rational numbers. In 1949, Robinson \cite{Rob} gave a 
first-order definition of $\Z$ in $\Q$ via the theory of ternary quadratic forms.
This was simplified by Poonen  \cite{Poonen} in 2009.
In 2016, Koenigsmann \cite{Koen} found that $\Q\sm \Z$ is diophantine over $\Q$, i.e.,
there is a polynomial $f(x_0,x_1,\ldots,x_n)$ with integer coefficients such that  $t\in\Q$ is not an integer if and only if $f(t,x_1,\ldots,x_n)=0$ for some $x_1,\ldots,x_n\in\Q$.
In 2024 Daans \cite{Daans} proved that we may take $n=10$; recently Sun \cite{Sun7}
showed further that we may even take $n=7$. 

Quite recently, OpenAI released a preprint \cite{OAI} in which HTP over $\Q$ is proved to be undecidable. The proof uses a five-contact height estimate and an existential predicate excluding odd poles.  We will state precisely the arithmetic inputs used there in Section~\ref{sec:inputs}.
At the end of \cite{OAI}, it is mentioned that {\it the construction supplies no bound on the number
of variables or coefficient heights, and it gives no single existential definition of $\Z$ in $\Q$}.

In this paper, we go further by establishing the following result.

\begin{theorem}\label{thm:main} $\HTP_{38}(\Q)$ is undecidable.
\end{theorem}

We use the single Lucas term bridge of \cite[Section~6, Proposition~6.1]{Sun10}, which strengthens the earlier binomial-coefficient construction of \cite{Sun11}. The source of undecidability is a system of two polynomial equations in eight integer indices and a rational parameter $w_*$, with nonzero guards. The first equation depends on seven integer indices, and the second on five integer indices and $w_*$. A parity condition in the first equation forces the size of a particular witness $x$ to be at most logarithmic in the maximum absolute value of the indices. In the affirmative case, $w_*$ may be fixed to be an integer greater than one and $x$ can subsequently be chosen arbitrarily large and divisible by any prescribed positive integer. We prove these properties in Section~\ref{sec:split-source}.

We use a fixed rank-one elliptic curve $E(F)$ over $F=\Q(\sqrt2)$. For a suitable even $m$ and a nontorsion point $P$, its cyclic image is $[m]E(F)=\Z P$. Each integer index is represented by an elliptic point using two rational abscissa components; the remaining ordinate is a finite algebraic choice. An aggregate trace of the eight root abscissas controls their indices by the height theorem of Von Buhren \cite{VB}. Three shifts of this trace give a positive rational square scale $A$, and one common slope $\eta$ supplies the arguments $t_0=\eta$ and $t_1=A\eta$.

Two disjoint five-contact configurations, each obtained from the same height theorem, are used for the two source equations. One additional contact balances the descent matrix. In the two-scale construction there are 21 prime positions, since the zero contact is shared. Seven local points and a common anchor suffice; we assign their indices by the Chinese Remainder Theorem at the disjoint sets of primes. Group-dependent rational factors suppress inactive tests. In addition, the scalar source, computed from $w_*$, imposes a height bound on $w_*$ in terms of the root index $x$. When the first source equation is true, the logarithmic bound on $|x|$ makes the numerator of the second discrepancy sufficiently small for the height argument.

The two ordinary sources and the scalar source define a genus-one covering. We give the local descent equations, an explicit 27-vertex certificate and its realization by the actual prime factors, including the expansion for the variable scalar support. The pointwise $2$-converse \cite{OAIPointwise} gives a rational representative of the specified class; odd multiples move it into the required local neighborhoods and real arc. The covering also enforces the positivity of a product of two adjacent contact factors once the integer source equations are true. We prescribe additional quadratic characters and use two existing local points to test a binary norm at the relevant finite primes. Hilbert reciprocity then enforces positivity without a separate global norm witness. The reciprocal trace of products combines the active comparisons, and the square-denominator lemma of Bhargava and Burungale \cite{BB} supplies the two terminal witnesses. We use their joint version of the preprint, and include the height argument for the required five contacts using their method and classical Shimura curve equations.

Finally, successive height arguments bound the integer indices of every false certificate independently of a finite cutoff. We exclude the entire cutoff box and run the resulting rational queries in parallel with the enumeration of genuine integer witnesses. The number of rational unknowns is
\begin{equation}\label{eq:introcount}
 2(8+7+1)+1+1+2+1+1=38.
\end{equation}
The three terms inside the parentheses count root points, shared local points and the anchor. The remaining terms count the slope, covering coordinate, two terminal parameters, compiler variable, and the rational source parameter $w_*$. 

\section{The split source with a rational witness}\label{sec:split-source}
Write $\Sq(\Z)=\{n^2:\,n\in\Z\}$. For a rational prime $p$, $v_p$ denotes its additive valuation, normalized by $v_p(p)=1$, with $v_p(0)=+\infty$. Fix a recursively enumerable set $\mathcal A\subseteq\N$. We use the source polynomials of \cite[Proposition~2.1 and (2.15)--(2.16)]{Sun10}.  In particular, put
\begin{gather*}
 d_a=3(2a+1),\qquad b=(1+d_af)^2,\\
 X=X(a,f,g),\quad Y=Y(a,f,g),\quad \mathscr C=2^{\alpha_1}b^{\alpha_2},
\end{gather*}
where $X,Y$ are integer polynomials and $\alpha_1,\alpha_2$ are fixed positive integers. Here $a\in\N$ is the input and $f,g$ are integer indices. For $f\ne0$ one has $b\ge4$. The proof of the cited proposition (pp.~7--8 of \cite{Sun10}) gives $Y>0$ and $4\mid X$ for all integer $f,g$, even outside the coding range. In its notation, $4\mid B'$, $B'\mid N'$ and $B'\mid T'+1$; thus $B'\mid R'=(S'+T'+1)N'+T'+1$ and $X=(N'-1)R'$. Here $R'$ is the source quantity just defined (written $R$ in the subsequent display on p.~8). We will use this unconditional parity in the growth argument.

For five integer indices $f,g,h,k,l$, use the polynomials of \cite[(6.7)--(6.13)]{Sun10}. To specify the interface completely, we give their formulas, using $M_0,T_0,D_{\rm size}$ for the size certificate and writing $D_{\rm size}$ in place of the $D_0$ of that manuscript:
\begin{gather*}
 A_0=2bYk,\ A=A_0^2,\
 M_0=2X+b+20,\ T_0=2M_0+2,\\
 D_{\rm size}=(T_0^2-4)M_0^2 A_0^2+4,\
 B=X+1,\ C=B+(A-2)h,\ L=bYl.
\end{gather*}
In particular, $A$ is a nonnegative even integer, $D_{\rm size}$ is divisible by $4$, and $B,C$ are odd, even before any square or positivity test is imposed. To define the positivity polynomial, put
\begin{align*}
 R_g=(g+1)(2\mathscr C-g),\ \kappa_0&=(2\mathscr C+1)^2,\
 S=2A-5,\ J_0=bACS.
\end{align*}
In the following expressions $t$ is a formal indeterminate:
\begin{gather*}
 s(t)=(A-4)t-1,\ H(t)=SL-2AC(2-A+t),\\
 \mathcal R(t)=4\kappa_0^2s(t)H(t)^2J_0^2
 -\big[\kappa_0(s(t)H(t)^2+J_0^2)-s(t)S^2R_g\big]^2.
\end{gather*}
Let $r_0'+r_1't$ be the remainder of $\mathcal R(t)$ upon division by $t^2-At+1$. Put
\[
 O_0=-f^2k^2((r_0')^2+Ar_0'r_1'+(r_1')^2).
\]
All the expressions are integer polynomials in $a,f,g,h,k,l$, since the divisor in the remainder computation is monic. In particular no $t$ is quantified in the source. The letters $A,B,C,S$ in this section have the meanings of the integer source, not those of the rational scale, elliptic curve or prime sets in later sections.

We record the precise use of the new bridge. Define the sequence $u_n=u_n(A)$ ($n=0,1,2,\ldots)$
by  $u_0=0$, $u_1=1$ and $u_{j+1}=Au_j-u_{j-1}$ for $j\in\Z$.
By Proposition~6.1 of \cite{Sun10}, we have
\begin{equation}\label{eq:source-bridge}
 a\in\mathcal A\iff\exists f,g,h,k,l\in\Z\,
 [D_{\rm size}\in\Sq(\Z)\land O_0>0 \land C=u_B(A)].
\end{equation}
Moreover, by the first two conditions we have
$$f,k\ne0,\ A\ge6,\ 2\mid A,\ 2\nmid B,\ 1<B<\f A2-1\ \t{and}\ A-2\mid C-B.$$ 
For completeness, the positive construction of the cited proposition first fixes $f,g$ with the coding divisibility, then takes $k>0$ arbitrarily large by Lemma~6.3, sets $C=u_B(A)$ via $h$, and takes the integral quotient $l$ of the extracted coefficient by $bY$. Taking $k\ge2$ ensures $O_0\ge2$, because $O_0$ is a positive integer divisible by $f^2k^2$. Thus the positive direction of \eqref{eq:source-bridge} remains valid with $O_0\ge2$.

For clarity, the reasons the new bridge applies are the following. If $O_0>0$, the two real conjugates of $\mathcal R(t)$ in $\Q[t]/(t^2-At+1)$ have opposite signs. At the smaller root $\lambda^{-1}$, where $\lambda=(A+\sqrt{A^2-4})/2$, the expression $s(\lambda^{-1})=\lambda^{-1}(\lambda^{-1}-4)$ is negative. Hence the other conjugate is positive. Factoring that conjugate gives, for one $\varepsilon\in\{\pm1\}$,
the facts
\[
 0\le g<2\mathscr C,\ \ 
 |L-\alpha_\varepsilon|^2<\frac{R_g}{\kappa_0}\le\frac14
 \ \ \t{and}\ \ 
 \alpha_\varepsilon=AC\left(\frac{2}{\lambda-2}
       +\frac{\varepsilon b}{\sqrt{\lambda(\lambda-4)}}\right),
\]
as in \cite[Lemma~6.4, (6.14)--(6.15)]{Sun10}. This argument uses only $f,k\ne0$ from $O_0>0$, $Y>0$ and $A\ge64$; it precedes the use of the coding range or a Lucas index. Once the coding range and $f,k\ne0$ have been recovered, the square condition on $D_{\rm size}$ gives $A>1000(b+2)4^X$ by \cite[Lemma~6.3, (6.9)]{Sun10}. If $C=u_B(A)=u_{X+1}(A)$, the estimate \cite[Lemma~6.1, (6.4)]{Sun10}, together with $\lambda>A/2$, bounds the error by $24(b+2)4^X/A<1/16$ for the integer
\[
 P_X^{(\varepsilon)}(A)=d_X^{(\varepsilon)}+A\sum_{j=1}^X d_{X-j}^{(\varepsilon)}u_j(A)\ \ \t{with}\ \ 
 d_i^{(\varepsilon)}=2^{i+1}+\varepsilon b\binom{2i}{i}.
\]
Consequently, $L=P_X^{(\varepsilon)}(A)$ and $bY\mid L$. This is the setting of \cite[Lemma~6.2, (6.5)--(6.6)]{Sun10}. Since $bY\mid A$, the constant term congruence
\[
 P_X^{(\varepsilon)}(A)\equiv 2^{X+1}+\varepsilon b\binom{2X}{X}\pmod{bY}
\]
first gives $b\mid 2^{X+1}$. Then $b,Y$ are powers of two and $bY\le X(X+1)<2^{X+1}$ by the cited source-size estimates, so $bY\mid2^{X+1}$. Thus $Y\mid\binom{2X}{X}$, and the reverse coding implication of Proposition~2.1 gives $a\in\mathcal A$. This explains why the earlier second Lucas candidate and the variable $w$ are no longer needed. 

For integer witnesses $\xi,y$, define the modified Lucas certificate as in \cite[(3.1)--(3.6)]{Sun10}, writing $\omega$ for the unnumbered quantity $\Omega=16E$ there:
\begin{gather*}
 D=(A^2-4)C^2+4,\ E=C^2D\xi,\\ F=4(A^2-4)E^2+1,\ G^*=2(A+2)(A-2)^2E^2-1,\\
 \omega=16E,\ H^*=C+(B-C)F+\omega Fy,\ I^*=((G^*)^2-1)(H^*)^2+1.
\end{gather*}
Set
\begin{equation}\label{eq:xi-y}
 \xi=2BA(A^2-4)Cx\ \ \t{and}\ \ y=2y_0+1,
\end{equation}
where $x$ and $y_0$ are new integer indices. On the guarded domain $fkx\ne0$, we have $A\ge4$ and $2\nmid BC$, so $\xi\ne0$ and $2\mid \xi$. With one more integer index $m_1$ define
\begin{equation}\label{eq:split-f1}
 F_{1,a}=m_1^2-DFI^*.
\end{equation}
This equation is independent of $l$ and of any positivity witness. By \cite[Theorem~3.1, (3.10)--(3.11)]{Sun10}, if $F_{1,a}=0$ and the square and positivity hypotheses of \eqref{eq:source-bridge} hold, then $C=u_B(A)$, since the range and congruence hypotheses and $\xi\ne0$ have been verified above.

\begin{lemma}\label{lem:rational-positivity}
Let $M,O_*\in\Z$ with $4\mid M$, and let $\varepsilon_+\in\Z^+$. If $w_*\in\Q\setminus\{0\}$ satisfies
\begin{equation}\label{eq:split-f2}
 (w_*-\varepsilon_+ O_*(M+1))^2-4\varepsilon_+^2O_*^2 M=0,
\end{equation}
then $M\in\Sq(\Z)$ and $w_*\in\Z$, with the sign of $w_*$ equal to the sign of $O_*$. Conversely, if $M=r_M^2$ and $O_*>0$, then $w_*=\varepsilon_+O_*(r_M+1)^2$ satisfies \eqref{eq:split-f2} for either integer square root $r_M$ of $M$.
\end{lemma}
\begin{proof}
If $O_*=0$, the equation gives $w_*=0$, which is excluded. Otherwise $r_M=(w_*-\varepsilon_+O_*(M+1))/(2\varepsilon_+O_*)$ has square $M$ and hence is an integer; it is even since $4\mid M$. Substitution gives $w_*=\varepsilon_+O_*(r_M+1)^2$. The last square is a positive integer, as $r_M$ is even. This proves both assertions.
\end{proof}

In \eqref{eq:split-f2} put $M=D_{\rm size}$ and $O_*=O_0$, and denote the left-hand side by $F_{2,a}$. We will take the positive integer $\varepsilon_+=1$ in the certificate; the elementary lemma allows any fixed positive integer. The polynomial depends on $f,g,h,k,l$ and $w_*$, but not on $x,y_0,m_1$. Its degree in the indices is uniformly bounded.

\begin{lemma}\label{lem:split-source}
Let $a\in\N$ and fix $\varepsilon_+\in\Z^+$. Then $a\in\mathcal A$ is equivalent to the existence of eight integers
\[
 \mathbf n=(f,g,h,k,l,x,y_0,m_1)
\]
and one integer $w_*>1$ satisfying $F_{1,a}=F_{2,a}=0$ and $fkx\ne0$. The integer-index lists of the two equations are
\begin{equation}\label{eq:split-slots}
 \mathcal I_1=(f,g,h,k,x,y_0,m_1)\ \ \t{and}\ \ 
 \mathcal I_2=(f,g,h,k,l).
\end{equation}
In the affirmative case, one may choose the five coding indices and $w_*$ first, with $w_*=\varepsilon_+O_0(r_M+1)^2$ for the positive square root $r_M$ of $D_{\rm size}$, and then realize $x>0$ arbitrarily large and divisible by any prescribed positive integer, together with $y_0$ and $m_1$.
\end{lemma}
\begin{proof}
For the reverse implication, Lemma~\ref{lem:rational-positivity} gives $D_{\rm size}\in\Sq(\Z)$ and $O_0>0$. Hence all the range and congruence conditions of the modified Lucas certificate hold. The guard gives $\xi\ne0$, and $F_{1,a}=0$ gives $DFI^*\in\Sq(\Z)$. Theorem~3.1 of \cite{Sun10} yields $C=u_B(A)$, and \eqref{eq:source-bridge} applies.

For the affirmative implication take the five coding indices from the positive construction preceding \eqref{eq:source-bridge}, so that $D_{\rm size}\ge4$ is a square, $O_0>0$ and $C=u_B(A)$. For the positive square root $r_M\ge2$, the indicated choice of $w_*$ is an integer at least nine. Fix any positive integer $N$ prescribing divisibility of $x$. Theorem~4.1 of \cite{Sun10} supplies a positive $\xi$ divisible by $2BA(A^2-4)CN$, with the second Lucas witness $y$ available in any residue class modulo $\xi$. The multiplier is positive in the affirmative construction. Choose the class $1\pmod\xi$. As $\xi$ is even, the resulting witness is odd, so $y=2y_0+1$ for an integer $y_0$. The square condition $DFI^*\in\Sq(\Z)$ supplies $m_1$. This proves the assertion on the freedom of $x=\xi/(2BA(A^2-4)C)$, since $N$ can be replaced by arbitrarily large multiples.
\end{proof}

\begin{lemma}\label{lem:sourcegrowth}
Let $a\in\N$, and let $$\mathbf n=(n_1,\ldots,n_8)=(f,g,h,k,l,x,y_0,m_1)$$ be an integer tuple
satisfying $fkx\ne0$ and $F_{1,a}=0$. Then 
\begin{equation}\label{eq:sourcegrowth}
 |x|\ll_a 1+\log B_0\ \ \t{with}\ B_0=\max(1,|n_1|,\ldots,|n_8|).
\end{equation}
\end{lemma}
\begin{proof}
Clearly, $A\ge4$, $2\nmid BC$, $\xi\ne0$, and $|\xi|\ge2|B|$. In particular, $D,F>0$ and $G^*>1$ (since $E$ is a nonzero integer), so $I^*>0$. Since $D\mid E$, we have $F\equiv I^*\equiv1\pmod D$.
As $2G^*=(A-2)F-A$ and $H^*\equiv C\pmod F$, we have $4I^*\equiv D\pmod F$. Now it is clear that
$D,F,I^*$ are pairwise coprime. The square condition on their product forces each to be an integer square. As $4I^*=((2G^*)^2-4)(H^*)^2+4$, the Lucas square criterion \cite[Lemma 2.1]{Sun10} gives $H^*=u_j(2G^*)$ for some $j\in\Z$. The expression $H^*$ is odd because $B,F$ are odd and $\omega$ is even. Since $2G^*$ is even, $j$ is odd.

By \cite[Lemma~3.1, (3.8)--(3.9)]{Sun10} (whose hypotheses $2\mid A$ and $E\ne0$ hold here),
we have
\[
 F\equiv1\pmod{\omega\xi}\ \ \t{and}\ \ 2G^*\equiv-2\pmod{\omega\xi}.
\]
The recurrence gives $u_j(-2)=j$ for odd $j$, including negative indices. Comparing the two expressions for $H^*$ yields
\[
 j\equiv B+\omega y\pmod{\omega\xi}.
\]
Since $\xi$ is nonzero and even, while $y$ is odd. So $\xi\nmid y$ and hence $j\ne B$. As $\omega\mid(j-B)$ and $|\omega|=16C^2D|\xi|\ge16|\xi|$, it follows that $|j|\ge |\omega|-|B|\ge(31/2)|\xi|$. The Lucas recurrence with $2G^*\ge3$ gives $|H^*|\ge2^{|j|-1}$. But $H^*$ is a fixed polynomial in the integer indices for each $a$, so $|H^*|\ll_a B_0^{d_1}$ for a fixed degree bound $d_1$. Taking logarithms and using $|x|\le|\xi|$ proves \eqref{eq:sourcegrowth}. 
\end{proof}

From now on, fix a nonrecursive recursively enumerable set $\mathcal A\subseteq\N$. The input is $a\in\N$. Fix an integer $d_\ast\ge1$ bounding the degrees in the indices of $F_{1,a}$ and $F_{2,a}$ (in the latter case treating $w_*$ and $\varepsilon_+$ as coefficients), uniformly before specializing $a$. Let $e=2$, which bounds the degree in $w_*$ of $F_{2,a}$. Put
\begin{equation}\label{eq:precision}
 k_\ast=24d_\ast.
\end{equation}
The precision and all subsequent fixed arithmetic constants are independent of $a$.

\section{Arithmetic data and the cyclic carrier}\label{sec:inputs}

We use the following consequence of the pointwise $2$-converse \cite[Theorem 1.1]{OAIPointwise}: if an elliptic curve $C/\Q$ has $C(\Q)[2]\ne0$ and $\operatorname{corank}_{\Z_2}\Sel_{2^\infty}(C/\Q)\in\{0,1\}$, then $\Sha(C/\Q)$ is finite. Here $\Sel_{2^\infty}$ is the full $2$-primary Selmer group and $\Sha$ denotes the Tate--Shafarevich group. We invoke the cited theorem, rather than attempting to reprove its source-level modularity or integral-family arguments. Its only application to the covering curves is made in Corollary~\ref{cor:rational-cover}, where its hypotheses are verified. The fixed carrier rank calculation below is independent of this application.

All heights below are absolute \emph{logarithmic} Weil heights. In particular, for coprime integers $a$ and $b>0$, we have $h(a/b)=\log\max(|a|,b)$, and in particular $h(0)=0$. 
For any prime $p$, we write $\F_p=\Z/p\Z$, set
$$\Z_{(p)}=\l\{\f ab:a,b\in\Z\ \t{and}\ p\nmid b\r\},$$
 and let $\Q_p$ be the field of $p$-adic numbers. 
   For $b\in\Q^\times=\Q\sm\{0\}$, we define its {\it support} 
$\Supp(b)$ as the set of primes $p$ with the $p$-adic valuation $v_p(b)$ of $b$ nonzero. 
As usual, we regard $v_p(0)$ as $+\infty$ for any prime $p$. We also set
  $$\Sq(\Q)=\{x^2:x\in\Q\}\ \ \t{and}\ \ \Q^{\times2}=\Sq(\Q)\setminus\{0\}.$$

The letter $p$ always denotes a rational prime. For a number field $L$, write $\OO_L$ for the ring of algebraic integers in $L$. We use $\mathcal P$ for a nonzero prime ideal of $\OO_L$. The notation $\mathcal P\mid p$ means that $\mathcal P$ divides the ideal $p\OO_L$ (equivalently, $\mathcal P\cap\Z=p\Z$). For $b\in L^\times=L\sm\{0\}$, let $\operatorname{ord}_{\mathcal P}(b)$ be the exponent of $\mathcal P$ in the principal fractional ideal $(b)$. Call $e_{\mathcal P}=\operatorname{ord}_{\mathcal P}(p)$ the ramification index, and define the normalized \emph{additive} valuation by
\[\qquad v_{\mathcal P}(0)=+\infty,
\ \ \t{and}\
 v_{\mathcal P}(b)=\frac{\operatorname{ord}_{\mathcal P}(b)}{e_{\mathcal P}} \ \t{for}\ b\in L^\times. 
\]
Thus $v_{\mathcal P}(L^\times)=e_{\mathcal P}^{-1}\Z$, and its values are integers when $\mathcal P$ is unramified. Write $L_{\mathcal P}$ for the completion and $\OO_{\mathcal P}$ for its valuation ring. 
For $b\in\Q^\times$ and $\mathcal P\mid p$, we have $v_{\mathcal P}(b)=v_p(b)$ and $|b|_p=p^{-v_p(b)}$. 

 For a number field $L$, $\Tr_{L/\Q}$ and $\Nm_{L/\Q}$ denote trace and norm respectively,  with respect to the field extension $L/\Q$. The origin of an elliptic curve is $O$, and $[n]Q$ denotes multiplication of a point by the integer $n$. The symbols $O(\cdot),o(\cdot),\ll,\asymp$ have their usual asymptotic meanings; subscripts specify allowed dependence of constants.

\subsection{A cyclic elliptic subgroup}
Fix
\begin{equation}\label{eq:curve}
 F=\Q(\sqrt2),\qquad d_0=75-53\sqrt2,\qquad
 E:y^2=x^3-d_0^2x.
\end{equation}
Let $\sigma$ be the nontrivial automorphism of $F$; a superscript $\sigma$ on a curve, point or expression means that its coefficients or coordinates are conjugated by this automorphism. 
For an ideal $\mathcal P\subset\OO_F$, write $\mathcal P^\sigma=\sigma(\mathcal P)$ for its conjugate ideal.
We have $d_0d_0^\sigma=7$. The rank calculation in \cite[Section 5]{OAI} gives a fixed even integer $m>0$ and a non-torsion point $P$ such that
\begin{equation}\label{eq:cyclic}
 mE(F)=\Z P.
\end{equation}
We recall the calculation to make the cyclic group in the construction explicit.

It is well-known that the ring $\OO_F=\Z[\sqrt2]$ is a principal ideal domain. Its unit group modulo squares has order four, and the signs of $-1$ and $1+\sqrt2$ show that every totally positive unit is a square. For the two-isogenous curve $E':y^2=x^3+4d_0^2x$, the usual descent maps $\alpha,\alpha'$ to $F^\times/F^{\times2}$ satisfy
\[
 |\im\alpha|\cdot|\im\alpha'|=4\cdot2^{\operatorname{rank}E(F)};
\]
see \cite[Chapter X, Section 4]{Silv}. The image for $E$ is supported only at $(d_0)$ and in the unit classes, and hence has size at most eight. Its torsion classes are $1,-1,d_0,-d_0$. With $w=-5+4\sqrt2=d_0(1+\sqrt2)^3$, one has $w^3-w=-8d_0$ and
$ (-d_0w,2\sqrt2\,d_0^2)\in E(F).$
Its $x$-coordinate has opposite signs at the two embeddings, so its class is outside those four classes. Thus $|\im\alpha|=8$.

For this descent paragraph, use the integer-valued ideal valuations $\operatorname{ord}_{\mathcal P}$. For $E'$, every nonzero $x$ is totally positive, and it has even $\operatorname{ord}_{\mathcal P}(x)$ away from $2d_0$. At $\mathcal P=(\sqrt2)$, one has $\operatorname{ord}_{\mathcal P}(4d_0^2)=4$. If $a=\operatorname{ord}_{\mathcal P}(x)$ were odd, then $3a$ and $4+a$ would be distinct odd integers, so $x^3+4d_0^2x=y^2$ would have odd valuation, a contradiction. At $\mathcal P=(d_0)$ the same argument leaves only $\operatorname{ord}_{\mathcal P}(x)=1$ to check. Writing $x=d_0u$ in that case would require $u^2+4=0$ in $\F_7$, which is impossible. Thus $\operatorname{ord}_{\mathcal P}(x)$ is even at every nonzero prime ideal. Principality and positivity imply that $x$ is a square. The exceptional points also have trivial image, so $|\im\alpha'|=1$. The rank is one. Taking $m$ divisible by the exponent of the torsion subgroup proves \eqref{eq:cyclic}; it is even because $E$ has nontrivial rational two-torsion.

Every point in \eqref{eq:cyclic} has an affine preimage under $[m]$. For the origin one can use $(0,0)$, and a preimage of a nonzero point cannot be the origin. This observation will save coordinates in Section~\ref{sec:encoding}.

\subsection{Aggregate trace height}\label{sec:productheight}
Put
\begin{equation}\label{eq:normalizedcurve}
 a=3-\sqrt2,\qquad \lambda=3-2\sqrt2,
 \qquad \widetilde E:y^2=x^3-a^2x.
\end{equation}
Both real conjugates of $a$ exceed $1$, and $a\lambda^2=75-53\sqrt2=d_0$. Since $\lambda(3+2\sqrt2)=1$, $\lambda$ is a unit of $\OO_F$. Thus
\begin{equation}\label{eq:curveisomorphism}
 \tau:\widetilde E\longrightarrow E,\qquad
 (x,y)\longmapsto(\lambda^2x,\lambda^3y)
\end{equation}
is an $F$-isomorphism. Put $\widetilde P=\tau^{-1}(P)$.  Since $P\in[m]E(F)$ and $m$ is even, $\widetilde P$ belongs to $2\widetilde E(F)$.  This divisibility will be used only to make the aggregate trace positive at the real places. The following lemma is motivated by \cite[A38]{Glazer}.

\begin{lemma}\label{lem:traceheight}
For every fixed $r\ge1$, every non-torsion point $Q\in\widetilde E(F)$,
and all nonzero integers $n_1,\ldots,n_r$, put
\[
 T_r=\frac12\Tr_{F/\Q}\left(\sum_{i=1}^{r}x([n_i]Q)\right)\in\Q.
\]
There are constants $M_r>0$ and $M_{r,0}\ge0$, depending only on
$Q$ and $r$, such that
\begin{equation}\label{eq:traceheight}
 \sum_{i=1}^{r}n_i^2\widehat h_{\widetilde E}(Q)
 \le M_r h(T_r)+M_{r,0}.
\end{equation}
The slope $M_r$ may be chosen from the fixed geometry; an effective
value of $M_{r,0}$ is not asserted.  This ineffectivity is harmless for
the reduction, which later enumerates all finite guard bounds rather than
requiring an effective value of the resulting index bound.
\end{lemma}
\begin{proof}
We verify the geometric hypotheses of Von Buhren's family height theorem
\cite[Theorem~1.1]{VB}, following the aggregate trace-height idea in
\cite[A38]{Glazer}.
Write $a'=3+\sqrt2$. Over an algebraic extension identify both
$\widetilde E$ and $\widetilde E^\sigma$ with
$E_0:Y^2=X^3-X$ by the pointed isomorphisms
\[
 (x,y)\longmapsto\left(x/a,y/(a\sqrt a)\right)\ \ \t{and}\ \ 
 (x',y')\longmapsto\left(x'/a',y'/(a'\sqrt{a'})\right).
\]
Let $Q_1,Q_2\in E_0(\overline\Q)$ be the images of $Q,Q^\sigma$.
They are independent over $\operatorname{End}(E_0)=\Z[i]$ modulo torsion.
Indeed, $a,a'$ are totally positive in $F$, have rational norm $7$,
and neither is a square in $F$; their product $7$ is also nonsquare
in $F$. A totally positive element of $F$ that is a square in
$F(i)$ must already be a square in $F$, because a square
$(u+iv)^2\in F$ requires $uv=0$, and the case $u=0$ is negative
at each real embedding. Hence $[a],[a']$ are independent in
$F(i)^\times/F(i)^{\times2}$. The automorphism of
$F(i,\sqrt a,\sqrt{a'})/F(i)$ fixing $\sqrt a$ and negating
$\sqrt{a'}$ fixes $Q_1$, negates $Q_2$, and fixes $\Z[i]$.
An endomorphism relation $\alpha Q_1+\beta Q_2$ equal to torsion
would therefore imply $2\beta Q_2$ torsion and then $\beta=0$;
non-torsion of $Q_1$ forces $\alpha=0$.

On $A=E_0^{2r}$ define the rational function
\[
 \mathscr F((U_i),(V_i))=
 \frac a2\sum_i X(U_i)+\frac{a'}2\sum_i X(V_i).
\]
Let $q=([n_i]Q_1)_i\times([n_i]Q_2)_i$.
We show that no positive-dimensional abelian translate through $q$
is contained in the fiber of $\mathscr F$ at $\mathscr F(q)$.
Any positive-dimensional abelian subvariety of $A$ has a
nonzero homomorphism from $E_0$, with coordinate endomorphisms
$\alpha_i,\beta_i\in\Z[i]$. Along that elliptic direction,
$\mathscr F$ is a sum of two blocks of abscissas
$X([n_i]Q_1+\alpha_i U)$ and $X([n_i]Q_2+\beta_iU)$.
At a common pole of summands $i,j$ in the first block,
$[n_i]Q_1+\alpha_i U=[n_j]Q_1+\alpha_j U=O$,
so $n_i\alpha_j=n_j\alpha_i$ by non-torsion of $Q_1$.
Writing the tangent multipliers of $\alpha_i$ as $\lambda_i$,
the corresponding leading Laurent coefficients at this pole
are proportional to $\lambda_i^{-2}$, with
$\lambda_j=(n_j/n_i)\lambda_i$. Their sum is
\[
 \lambda_i^{-2}\sum_{j\ \mathrm{at\ the\ pole}}
                      (n_i/n_j)^2\ne0.
\]
Thus a pole in a nonconstant block survives. The same holds in
the second block. A pole cannot occur in both blocks, for then
nonzero endomorphisms $\alpha,\beta$ would give an equation
$n\beta Q_1-m\alpha Q_2=O$, contrary to their independence.
At least one block is nonconstant, so the whole restriction is
nonconstant. This excludes every positive-dimensional abelian
translate through $q$ in the fiber.

Let $V\subseteq A\times\mathbb P^1$ be the closure of the graph
of $\mathscr F$.  It is integral because it is the closure of the graph
of a rational map from the integral variety $A$.  Let $\Gamma$ be the
subgroup of $A(\overline\Q)$ generated by the $2r$ coordinate points
\[
 (0,\ldots,Q_1,\ldots,0)\ \ \t{and}\ \ 
 (0,\ldots,Q_2,\ldots,0).
\]
Thus $\Gamma$ has finite rank and $q\in\Gamma$.

We have checked the exceptional locus required by
\cite[Theorem~1.1]{VB}.  Indeed, if a positive-dimensional abelian
translate through $q$ were contained in the fiber of $V\to\mathbb P^1$
through $(q,\mathscr F(q))$, its abelian part would contain an elliptic
subvariety, because $A=E_0^{2r}$.  Restriction to a nonzero homomorphism
$E_0\to A$ would then make $\mathscr F$ constant, contrary to the pole
argument above.  Hence $q$ lies outside the union $Z_{V_{\mathscr F(q)}}$
of positive-dimensional abelian translates in that fiber.  Moreover
$q\in\Gamma\subset C(\Gamma,\varepsilon)$ in the notation of Von
Buhren's theorem, by taking the error term there to be zero.

Apply that theorem with $P=\mathbb P^1$ and its standard ample line
bundle.  Its geometric constant is fixed by $A$ and $V$, while its
additive constant may depend on the fixed finite-rank group $\Gamma$.
We obtain
\[
 \widehat h_A(q)\le C_r h(\mathscr F(q))+C_{r,0}.
\]
Use a product symmetric ample line bundle on $A$. The left-hand
side is a fixed positive multiple of the sum of canonical heights
in the $2r$ coordinates. Galois conjugation and pointed
isomorphisms preserve canonical heights, so this sum equals
$2\sum_i n_i^2\widehat h_{\widetilde E}(Q)$.
Also $\mathscr F(q)=T_r$.  The standard height on $\mathbb P^1$
is $h(T_r)$ at the affine point $(T_r:1)$; any normalization change
is absorbed into $C_{r,0}$. Rescaling constants proves
\eqref{eq:traceheight}.
\end{proof}

For eight integer indices $\mathbf n=(n_1,\ldots,n_8)$ of nonempty support $I=\{i:n_i\ne0\}$, use the normalized abscissas and put
\begin{equation}\label{eq:tracecoords}
 R_I=\frac12\Tr_{F/\Q}\left(\sum_{i\in I}x([n_i]\widetilde P)\right)\ \ \t{and}\ \ 
 B_0=\max(1,|n_1|,\ldots,|n_8|).
\end{equation}
Each nonzero point $[n_i]\widetilde P$ is a double and has positive abscissa greater than one at both real embeddings. In fact, the duplication formula gives a nonzero square abscissa, and the real curve equation then forces $x>3\pm\sqrt2$ at the respective embeddings. Thus $R_I>0$. By Lemma~\ref{lem:traceheight}, taking the maximum of the constants over the finitely many nonempty supports, we see that there are fixed constants with
\begin{equation}\label{eq:aggregateheight}
 B_0^2\ll 1+h(R_I).
\end{equation}
The trace is computed from the root points on $E$ by $x(\tau^{-1}P_i)=\lambda^{-2}x(P_i)$ on the nonidentity branches. The fixed point $P$ and the integer $m$ may be built into the reduction. Neither the constants in the height estimate nor a Mordell--Weil computation depending on $a$ is needed by the decision algorithm.

\subsection{Formal comparisons}
Put $z=-x/y$ on nonzero cyclic points and $z(O)=0$. Such points are not two-torsion, so both $x$ and $y$ are nonzero. Let
\[
 \ell(Z)=Z+\sum_{t\ge2}c_tZ^t\ \ \t{and}\ \ 
 \ell_K(Z)=Z+\sum_{2\le t\le K}c_tZ^t
\]
be the formal logarithm and its truncation. The coefficients are effectively computable from \eqref{eq:curve}.

\begin{lemma}\label{lem:formal}
There is a fixed finite set $S_E$ of primes such that the following holds. 
Suppose $p\notin S_E$, $\mathcal P\mid p$, and $nP,n'P$ reduce to $O$ modulo $\mathcal P$, where $n,n'\in\Z$ and $n'\ne0$. Then
\begin{equation}\label{eq:formalval}
 v_{\mathcal P}\left(\frac{z(nP)}{z(n'P)}\right)=v_p\l(\f n{n'}\r).
\end{equation}
If $p>K+1$ and $v_p(n/n')\ge0$, then $\ell_K(z(n'P))\ne0$ and
\begin{equation}\label{eq:formalapprox}
 v_{\mathcal P}\left(\frac{\ell_K(z(nP))}{\ell_K(z(n'P))}-\frac n{n'}\right)\ge K.
\end{equation}
For $p>K+1$, we have
\begin{equation}\label{eq:truncval}
 v_{\mathcal P}\left(\frac{\ell_K(z(nP))}{\ell_K(z(n'P))}\right)=v_p\l(\f n{n'}\r).
\end{equation}
The same statements hold for $E^\sigma,P^\sigma$, using the conjugate coefficients $\ell_K^\sigma$, and for the normalized curve $\widetilde E,\widetilde P$ after transporting by the isomorphism \eqref{eq:curveisomorphism}.
\end{lemma}
\begin{proof}
Exclude 2, ramified primes, and bad reduction primes. The invariant differential has integral coefficients, so $v_{\mathcal P}(c_t)\ge-v_p(t)$. For $v_{\mathcal P}(Z)\ge1$ the logarithm converges, is a homomorphism on the formal kernel, and has valuation $v_{\mathcal P}(Z)$. The identity $[n'](nP)=[n](n'P)$ gives \eqref{eq:formalval}. If $p>K+1$, then for $t>K$ one has $t-1-v_p(t)\ge K$, and consequently
\[
 v_{\mathcal P}(\ell(Z)-\ell_K(Z))\ge v_{\mathcal P}(Z)+K.
\]
Apply this to both parameters, using \eqref{eq:formalval} and the integrality of $n/n'$, and divide by $\ell_K(z(n'P))$, whose valuation is $v_{\mathcal P}(z(n'P))$. This proves \eqref{eq:formalapprox}. For every formal parameter, $\ell_K(Z)$ has valuation $v_{\mathcal P}(Z)$, which also proves \eqref{eq:truncval}. These arguments prove the conjugate statements; compare \cite[Proposition 2.2]{OAI} and \cite[Chapter IV]{Silv}.
\end{proof}

\section{The two five-contact configurations}\label{sec:heightcontacts}

For the height input, the five contacts of the height theorem are
\[
 \mathscr C=\{\infty,\,0,\,1,\,2401,\,-343/32\}.
\]
For a fixed finite contact set $\mathscr D$ in $\mathbb P^1(\Q)$ and a fixed exceptional prime set $S'$, let $M_{\mathscr D,S'}(a)$ be the product of the primes $p\notin S'$ at which the normalized depth at one of the contacts is positive and odd. For a homogeneous coordinate pair $(u:v)$, this depth is $v_p(L(u,v))-\min(v_p(u),v_p(v))$, where $L$ is a primitive integral linear form for the contact. For the infinite contact take $L(u,v)=v$. We give the required height bound using the method of \cite[Section 3]{BB}, with a different quaternionic discriminant. The five-contact configuration used by Burungale is recovered from the classical explicit quotient equations, so that the only cited version of the Bhargava--Burungale preprint is the current joint version.

\begin{lemma}\label{lem:fivecontact}
There is a fixed finite set $S_{\rm geom}$ of primes such that, for $\mathscr D=\mathscr C$ or any fixed rational projective transform of $\mathscr C$, and any fixed finite set $S'$ containing $S_{\rm geom}$ and the coefficient and determinant primes of that transform, one has
\begin{equation}\label{eq:height-input23}
 h(a)\le C_{\mathscr D,S'}M_{\mathscr D,S'}(a)^{24}\quad (a\notin\mathscr D)
\end{equation}
with a constant independent of $a\in\mathbb P^1(\Q)$.
\end{lemma}
\begin{proof}
Let $B$ be the indefinite rational quaternion algebra of discriminant $210$, let $\mathcal O$ be a maximal order, and let $V$ be the corresponding level-one Shimura curve over $\Q$. Its Atkin--Lehner group is $W=\{w_m:m\mid210\}\cong(\Z/2\Z)^4$. The standard genus formula gives $g(V)=5$: both elliptic-point terms vanish, since $5$ splits in $\Q(\sqrt{-1})$ and $7$ splits in $\Q(\sqrt{-3})$ (see \cite{Rot02,Voight}). By \cite[Theorem 4.2, Table 2]{GR06} and \cite[Table III]{Rot02}, the five quotients $V/\langle w_m\rangle$ for $m=30,42,70,105,210$ have genus one. Each such involution has eight fixed points by Riemann--Hurwitz. A point stabilizer for $W$ is cyclic, by its faithful tangent action, so a nontrivial stabilizer has order two. Thus each of the five involutions gives one distinct branch value of $V\to V/W$. If the total number of branch values is $r$ and the quotient genus is $g_*$, Riemann--Hurwitz gives $8=16(2g_*-2)+8r$. Since $r\ge5$, necessarily $g_*=0$ and $r=5$. Each branch value $P_m$ is rational, since it is uniquely specified by its rational inertia involution. In particular $V/W\cong\mathbb P^1_\Q$.

We compute its five labelled branch values. Over $\overline\Q$, the group $W/\langle w_m\rangle$ acts faithfully on the genus-one curve $V/\langle w_m\rangle$. After choosing an origin, its translation subgroup is the full group of two-torsion translations: it has order four because the linear image has order at most two, and an exponent-two translation subgroup has order at most four. Quotienting first by these translations gives a genus-one curve with the same $j$-invariant, by multiplication by two. The remaining double cover of $V/W$ is branched precisely at the four $P_n$ with $n\ne m$. Indeed, at $P_m$ the first genus-one quotient is unramified over $V/W$; at the other four values the induced inertia has fixed points on the genus-one curve and is therefore not a nontrivial translation.

Write $J(\lambda)=256(1-\lambda+\lambda^2)^3/(\lambda^2(1-\lambda)^2)$ for the Legendre $j$-invariant. The quotient models in \cite[Table 2]{GR06} and \cite[Table III]{Rot02} give the following data (only the $j$-invariants of the genus-one curves are used):
\[
\begin{array}{c|rrrrr}
 m&30&42&70&105&210\\ \hline
 \lambda&-2400&32/375&225&-1024/5&-3/32
\end{array}
\]
Here $j(V/\langle w_m\rangle)=J(\lambda)$ in each column. For example, for a binary quartic with coefficients $(a_4,a_3,a_2,a_1,a_0)$ these equalities follow by substituting into
\[
\begin{gathered}
 I=12a_4a_0-3a_3a_1+a_2^2,\\
 J_0=72a_4a_2a_0+9a_3a_2a_1-27a_4a_1^2-27a_3^2a_0-2a_2^3,
 \ \
 j=\frac{6912I^3}{4I^3-J_0^2}.
\end{gathered}
\]
For $m=210$, the model is $y^2+xy=x^3+x^2-23x+33$.
Normalize a rational coordinate $z$ by $z(P_{42})=\infty$, $z(P_{105})=0$, $z(P_{30})=1$. Write $A_*=z(P_{70})$ and $B_*=z(P_{210})$. Omitting one branch value at a time gives
\[
\begin{gathered}
 J(A_*)=J(-3/32),\ \  J(B_*)=J(225),\\
 J(B_*/A_*)=J(-2400),\ 
 J((B_*-1)/(A_*-1))=J(-1024/5).
\end{gathered}
\]
The fibre of $J$ at each of these nonspecial values is the six-element set
$\{\lambda,1/\lambda,1-\lambda,1/(1-\lambda),\lambda/(\lambda-1),(\lambda-1)/\lambda\}$.
Checking the $36$ pairs from the first two equalities in the last two gives uniquely
$A_*=-3/32$, $B_*=225$. The fifth equality is also satisfied, since $A_*(B_*-1)/(B_*(A_*-1))=32/375$. The coordinate $a=2401z/(z-225)$ therefore has branch set exactly $\mathscr C$.

We next use the quaternionic height argument of \cite[Section 3]{BB}. Add fine level $V_1(11)$ (the level is prime to $210$), with its universal marked, principally polarized abelian surface, as in \cite[Section 1, Theorem 2.1, Lemma 2.2 and Corollary 2.3]{Buz}. The map to $V$ is finite \'etale, since $V$ has no elliptic points. As in the proof of \cite[Proposition 3.1]{BB}, pass further to the finite \'etale cover parametrizing symmetric rigidified representatives of the polarization and compatible theta data of level four. Keep all components, so that the resulting smooth projective curve $Y$ and its universal marked family are defined over $\Q$. Its map $\pi:Y\to V/W$ is finite, unramified away from $\mathscr C$, with ramification index two over each contact.

After inverting a fixed finite set $S_0$ of rational primes, containing $2$, the curves and abelian family have smooth proper models, the branch sections are disjoint, and the finite map has the same tame ramification along these sections. One may ensure the latter by spreading out the finite \'etale maps and the quotient map by $W$, including its inertia loci, and inverting $|W|$. If $a\in\mathbb P^1(\Q)\setminus\mathscr C$, choose a closed point $y$ of the fibre of $\pi$ over $a$ with residue field $L$ of degree at most $\deg\pi$. Outside $S_0$ and away from the contact sections the fibre is \'etale over the valuation ring. At a contact, after unramified base change and strict henselization, the local map has the form $u\mapsto u^2$ in a transverse parameter, up to a unit with a square root there. Its pullback along $a$ is unramified if the normalized contact depth is even, since a square root of an element of even valuation is contained in an unramified extension at an odd prime. It follows that $L/\Q$ is unramified outside $S_0$ and the primes dividing $M_{\mathscr C,S_0}(a)$.

The fibre $A_y$ of the universal abelian family extends to an abelian scheme over $U=\operatorname{Spec}O_L[1/\prod_{p\in S_0}p]$ by properness of the model of $Y$. Let $N_U$ be the product of norms of omitted finite places, and let $\operatorname{rad}$ denote the positive radical. We have
\[
 \operatorname{rad}(N_U|\operatorname{disc}(L)|)
 \mid M_{\mathscr C,S_0}(a)\prod_{p\in S_0}p.
\]
Since $\pi(y)$ is rational, for every $\tau\in\Gal(\overline\Q/\Q)$ the conjugate marked surface is related to $A_y$ by an Atkin--Lehner action. This gives an $\mathcal O$-linear isogeny $\tau^*A_y\to A_y$: quotient by the kernel attached to a two-sided ideal of $\mathcal O$ and compose with the marked isomorphism, exactly as in \cite[Lemma 3.2]{BB}. Fix a real quadratic subfield of $B$ (for instance $\Q(\sqrt{210})$, which embeds since none of the ramified primes of $B$ splits in it). Its action is defined over $L$ and the isogenies commute with it. Thus $A_y$ is of $\mathrm{GL}_2$-type with $G_\Q$-isogenies in the sense of \cite[(7.1)]{vKK23}. Theorem 7.1(i) of that paper, with the refinement immediately following it to exponent $\max\{24,5\dim A_y\}=24$, gives, for the stable Faltings height,
\[
 h_F(A_y)\ll \operatorname{rad}(N_U|\operatorname{disc}(L)|)^{24}
 \ll M_{\mathscr C,S_0}(a)^{24}.
\]
The implied constants are uniform because the field degree is bounded.

Finally the theta-null height comparison in \cite[proof of Proposition 3.1]{BB} applies to this family: the theta-null morphism $Y\to\mathbb P^{15}$ is nonconstant on every geometric component, since a theta-null point determines finitely many polarized surfaces, each with countably many $\mathcal O$-actions and finitely many level and theta choices. Hence its hyperplane pullback is ample. The same comparison of ample heights and of theta and stable Faltings heights gives
\[
 h(a)\ll 1+\max\{0,h_F(A_y)\}.
\]
This proves \eqref{eq:height-input23} for $\mathscr C$. A fixed projective change preserves normalized contact depths outside its coefficient and determinant primes, by applying both it and its inverse to primitive homogeneous coordinates, and changes the height by $O(1)$. Enlarging the exceptional set absorbs only a fixed product of primes in the constant. This proves the lemma.
\end{proof}

The set $\mathcal T$ below is the inverse image of $\mathscr C$ under $a=2+1/t$; we include the fixed coordinate-change exceptions in $S_{\rm geom}$.

Put
\begin{equation}\label{eq:Tset}
 \mathcal T=(0,-1/2,-1,1/2399,-32/407)=(\theta_0,\ldots,\theta_4).
\end{equation}
Choose a sufficiently small positive rational square $a_*<1/20$ so that the twenty-five numbers $a_*\theta_j-\theta_i$ are distinct (it suffices that $a_*$ times the largest root difference be smaller than the smallest nonzero root difference), and so that for $-7/10<b_*<-3/5$ all $a_*\theta_j+b_*$ have absolute value between $32/407$ and $1$. We next fix $b_*$ in that interval and $\theta_{10}\in(-9/10,-4/5)$ by the construction below. Set
\[
\begin{gathered}
 \theta_{5+j}=a_*\theta_j+b_*\quad(0\le j\le4),\\
 \mathcal T_L=\mathcal T,\ \ \mathcal T_N=a_*\mathcal T+b_*,\ \ 
 \mathcal T_*=\mathcal T_L\cup\mathcal T_N\cup\{\theta_{10}\}.
\end{gathered}
\]
The indices $0,\ldots,10$ enumerate the eleven roots, and $\mathcal T_*^\times=\mathcal T_*\setminus\{0\}$. All roots are distinct; the maximum positive ratio of nonzero roots is less than $16$ and the maximum absolute ratio is at most $2399$. The height estimate of Lemma~\ref{lem:fivecontact} applies to both $\mathcal T$ and $\mathcal T_N$.

For each pair $(i,j)$ with $0\le i,j\le4$, choose a distinct fixed odd prime $\ell_{ij}>11$ outside the supports of the original coefficients and differences, $a_*$ and the nonzero differences between the twenty-five numbers just displayed, also excluding $37$ and $2399$. By weak approximation choose $b_*$ in its prescribed real interval with
\[
 v_{\ell_{ij}}(b_*+a_*\theta_j-\theta_i)=1.
\]
At $\ell_{ij}$ every other difference between the two copies is a unit; the within-copy differences are also units. Next choose ten distinct additional primes $m_i>11$, outside the supports of all coefficients and differences of the first ten roots and all previously selected primes, also excluding $37$ and $2399$. Choose $\theta_{10}$ in its indicated interval, with $v_{m_i}(\theta_{10}-\theta_i)=1$ and with residue different from all first ten roots at every $\ell_{ij}$. This is again finite rational weak approximation. At $m_i$ every other difference is a unit. All these choices are made once and for all. At each of these 35 private coefficient primes its assigned new within-scale difference has valuation one and all the other within-scale differences have valuation zero. For the primes $\ell_{ij}$ this follows from the distinctness modulo $\ell_{ij}$ of the 25 offsets $a_*\theta_j-\theta_i$, the unit within-copy differences, and the subsequent avoidance condition on $\theta_{10}$. For the primes $m_i$ the first ten roots are pairwise distinct modulo $m_i$. These choices require only finitely many rational congruences, together with open real intervals, so weak approximation applies. The 35 primes are included in the fixed preparation core. Their independently prescribed common norm characters will realize the desired internal blocks in Section~\ref{sec:selmer}.

For later use, we define
\[
\begin{aligned}
 F_{11}(V)&=\prod_{\theta\in\mathcal T_*}(V-\theta),\qquad F_{10}(V)=F_{11}(V)/V,\\
 g_G(V)&=\frac{\prod_{\theta\in\mathcal T_*\setminus\mathcal T_G}(V-\theta)}{(V-3)^{|\mathcal T_*\setminus\mathcal T_G|}}\quad(G=L,N),\\
 g_+(V)&=\frac{\prod_{\theta\in\mathcal T_*\setminus\{\theta_7,\theta_{10}\}}(V-\theta)}{(V-3)^9}.
\end{aligned}
\]

\section{Scalar preparation and the square scale}\label{sec:fixed}

Choose a fixed finite preparation core $S_{\rm core}$ containing $S_{\rm geom}$, the exceptional primes in Lemma~\ref{lem:formal}, the prime divisors of $64$, all primes at most $\max\{10^{16},k_\ast+1\}$, the coefficient and nonzero difference primes of the finite scalar contacts and of the extra contact $1/2$, the supports of the contact coefficients, their nonzero differences and the differences with $3$, the private coefficient primes, the coefficient and determinant primes of the fixed affine change of contact set, the supports of differences of distinct nonzero-root ratios, and the supports of all nonzero $25\theta-\theta'$ with $\theta,\theta'\in\mathcal T_*$. Include the prime divisors of the numerators and denominators of the fixed nonzero roots and their nonzero ratios. Also include the prime divisors of $r_*(2402)r_*(2404)$, where
\[
 r_*(Y)=Y(Y-1)(Y-2401)(32Y+343)(2Y-1).
\]
All these choices are independent of the input $a$; this set can be very large, but is finite. Put
\begin{equation}\label{eq:Omega-kappa}
 \Omega=\prod_{q\in S_{\rm core}}q\ \ \t{and}\ \ 
 \kappa=\prod_{q\in S_{\rm core},\ q\ne2}q.
\end{equation}
For each of the 16 profiles $f\in\F_2^4$ of the additive Legendre characters of $-1,3,2,5$, choose a fixed prime $q_f\notin S_{\rm core}$ of that profile, also avoiding the prime divisors of $1+\Omega$ and $9+\Omega$. Such primes exist by Dirichlet's theorem, since each profile is a unit class condition modulo $120$. Choose also eight distinct primes of profile zero, denoted by $q_{g,j}$ ($1\le g\le4$
and $1\le j\le2$), outside the core and the sixteen profile primes, avoiding $1+\Omega$ and $9+\Omega$. They exist by Dirichlet's theorem. Put $\mathcal R_{\rm res}=\{q_{g,j}\}$ and define
\begin{equation}\label{eq:Sht38}
 S_{\rm ht}=S_{\rm core}\cup\{q_f:f\in\F_2^4\}\cup\mathcal R_{\rm res}\ \ \t{and}\ \
 \mathcal M=\left\{\prod_{q\in S_{\rm ht}}q^{\varepsilon_q}:\varepsilon_q\in\{0,1\}\right\}.
\end{equation}
Fix a positive integer $L_0$ divisible by $8$ and by every prime in $S_{\rm ht}$. For $w_*\in\Q$ put
\begin{equation}\label{eq:Qstar}
 Q_*=1+L_0w_*.
\end{equation}
Guard $w_*\ne0,1$ and $Q_*\ne0$. For integral $w_*$ satisfying these guards, $Q_*$ is a nonzero integer, is a square unit in every $\Q_q$ for $q\in S_{\rm ht}$, and $Q_*>0$ is equivalent to $w_*>0$ on the guarded domain. Indeed, $Q_*\equiv1\pmod{q}$ at odd $q\mid L_0$ and $Q_*\equiv1\pmod8$.

Let $N_q\ge\max(1+v_q(2),1+2v_q(2))$ be fixed depths for Lemma~\ref{lem:coreprep} (with $b_q=1$), and fix a positive integer $\Gamma_0$ supported on the core with $2v_q(\Gamma_0)\ge N_q$ there. Define
\begin{equation}\label{eq:scalar-input38}
 \beta=2Q_*^2+2402\ \ \t{and}\ \ r_0=r_*(\beta).
\end{equation}
All factors of $r_0$ are positive over the reals. For branches $\mu,\nu_*\in\mathcal M$, set
\begin{equation}\label{eq:rdef}
 r=\mu\Gamma_0^2 Q_*r_0\ \ \t{and}\ \ \rho=\nu_*\kappa\mu\Gamma_0^2 r_0 Q_*(Q_*-1).
\end{equation}
The multipliers are finite choices, not quantified parameters. For an affirmative integer $w_*>1$, $Q_*>1$ and $r_0$ are positive integers. Lemma~\ref{lem:finite-normalization} chooses $\mu$ so that $r$ has odd $2$-valuation, even valuations at the other core primes, and the prescribed profile-count parities for its exterior odd support $\mathcal B$, including the eight reserve primes $\mathcal R_{\rm res}$. The square $\Gamma_0^2$ gives the required core depths. Put
\begin{equation}\label{eq:oddw}
 \mathcal P_Q=\{p:v_p(Q_*)\text{ is odd}\}.
\end{equation}
In the affirmative construction $\mathcal P_Q\cap S_{\rm ht}=\varnothing$. For any $p\notin S_{\rm core}$ dividing $Q_*(Q_*-1)$, the value of $\beta$ modulo $p$ is $2402$ or $2404$, so $r_0$ is a unit there by the core exclusions. In particular $\mathcal P_Q\subseteq\mathcal B$. Since $Q_*>0$ and $Q_*\equiv1\pmod4$, $\prod_{p\in\mathcal P_Q}(-1/p)=1$. At a prime of $\mathcal B\setminus S_{\rm ht}$, $Q_*-1$ is a unit: otherwise both $Q_*$ and $r_0$ would be units. As $\rho=\nu_*\kappa r(Q_*-1)$, it has positive odd valuation there. At each core or other fixed prime belonging to $\mathcal B$, the valuation of $\kappa r(Q_*-1)$ is positive, so a toggle by $\nu_*$ gives
\begin{equation}\label{eq:rho-odd38}
 v_p(\rho)>0\text{ odd}\qquad(p\in S_{\rm core}\cup\mathcal B).
\end{equation}
This does not change $r$. For every branch, including nonintegral $w_*$, the height estimates are
\begin{equation}\label{eq:rheight38}
 h(r)+h(\rho)\ll 1+h(\beta)\ \ \t{and}\ \ h(w_*)\le \tfrac12h(\beta)+O(1).
\end{equation}
Indeed, $Q_*^2=(\beta-2402)/2$ and $w_*=(Q_*-1)/L_0$. The constants are uniform over the finite branches. The height-exception set remains fixed; the variable support of $Q_*$ is not added to it.

\begin{proposition}\label{prop:unit-scale}
Let $\rho\in\Q^\times$, put $L=1+\rho^2$, and define
\begin{equation}\label{eq:Arho}
 f_\rho(T)=5L\frac{T^2+\rho}{T^2+L\rho}\ \ \t{and}\ \  A_\rho(T)=f_\rho(T)^2.
\end{equation}
At every prime $p$ with $v_p(\rho)>0$ odd, the scale is a unit for every rational $T$ if $p\ne5$, and has valuation two at $p=5$.  More precisely, if $m=v_p(\rho)>0$ is odd, then
\begin{equation}\label{eq:Acong}
 f_\rho(T)/5\in1+p^{2m}\Z_p\ \ \t{and}\ \ A_\rho(T)/25\in1+p^{2m}\Z_p.
\end{equation}
If $\rho>0$, then $A_\rho(T)>25$ for every nonzero real $T$.  Moreover,
\begin{equation}\label{eq:Aheight}
 4h(T)-10h(\rho)-2\log30\le h(A_\rho(T))
 \le4h(T)+6h(\rho)+2\log20,
\end{equation}
whenever the displayed rational functions are defined.
\end{proposition}
\begin{proof}
At such a prime $L$ is a unit. The valuations of $T^2$ and $L\rho$ have opposite parity (unless $T=0$, when the identity is immediate), so no cancellation occurs and
\begin{equation}\label{eq:scaleidentity}
 f_\rho(T)/5-1=\frac{\rho^2T^2}{T^2+L\rho}.
\end{equation}
The valuation of the right-hand side is at least $2m$, which proves \eqref{eq:Acong}.  The real assertion follows from $(T^2+\rho)/(T^2+L\rho)>1/L$.

For the height estimate write $\rho=a/b$ and $T=p/q$ in lowest terms, with $b,q>0$, and put $W=a^2+b^2$.  Then
\[
 f_\rho(T)=\frac{5W(bp^2+aq^2)}{b^3p^2+aWq^2}.
\]
The two quadratic forms $U=bp^2+aq^2$ and $V=b^3p^2+aWq^2$ satisfy
\[
 aWU-aV=a^3bp^2\ \ \t{and}\ \  bV-b^3U=a^3bq^2.
\]
Hence $\gcd(U,V)$ divides $|a^3b|$, even if one of the two forms is zero. After multiplication of $U$ by $5W$ the gcd divides $5W|a^3b|$. If $H_\rho=\max(|a|,b)$ and $H_T=\max(|p|,q)$, the two adjugate identities give $\max(|U|,|V|)\ge |a^3b|H_T^2/(3H_\rho^3)$, because $\max\{|a|(W+1),b(1+b^2)\}\le3H_\rho^3$. The numerator and denominator of the unreduced $f_\rho(T)$ are each bounded in absolute value by $20H_\rho^3H_T^2$. Dividing by the gcd proves
\[
 H(f_\rho(T))\ge\frac{H(T)^2}{30H(\rho)^5}\ \ 
 \t{and}\ \  H(f_\rho(T))\le20H(\rho)^3H(T)^2.
\]
Taking logarithms and using $h(A_\rho)=2h(f_\rho)$ gives \eqref{eq:Aheight}.
\end{proof}

\subsection{Explicit input-dependent shifts with controlled height}\label{subsec:shifts}

For the shifted-scale existence argument suppose $\rho>0$.  For the finite contact-ratio set
\begin{equation}\label{eq:Cratios}
 \mathcal C=\{\theta'/\theta:\theta,\theta'\in\mathcal T_*^\times\},
\end{equation}
one has $\max\mathcal C_{>0}<16$.  In addition, $|d|\le2399$ for $d\in\mathcal C$. For either square root $s$ of $d$, the zeros satisfy $T^2=L_\rho\rho(s-5)/(5L_\rho-s)$. Neither numerator nor denominator vanishes: positive $d$ has $s^2<16$, while for negative $d$ the square root is imaginary. The four roots for each $d$ are distinct and the sets for different $d$ are disjoint. For positive $d$, $|T|^2<9\rho$ follows from $|s|<4$ and $L_\rho/(5L_\rho-4)\le1$. For negative $d$, $|T|^2<10\rho$ follows from $|s-5|=\sqrt{|d|+25}<50$ and $|5L_\rho-s|\ge5L_\rho$. All branch points thus lie in $|T|<4\sqrt\rho$. The poles are even and infinity is unramified.  Consequently the three shifts
\begin{equation}\label{eq:three-shifts}
 b_j(\rho)=5j(1+\rho),\qquad j=1,2,3,
\end{equation}
separate the branch sets.

\begin{lemma}\label{lem:three-shift}
Fix $\rho>0$ and a finite constant squareclass space $V\subset\Q^\times/\Q^{\times2}$.  All but finitely many rational $T$ have a shift $j\in\{1,2,3\}$ for which the classes
\begin{equation}\label{eq:independentclasses}
 [A_\rho(T+b_j(\rho))-d],\qquad d\in\mathcal C,
\end{equation}
are independent modulo $V$.  For either sign of $\rho\ne0$, uniformly in the three branches on their domains,
\begin{equation}\label{eq:shiftheight}
 h(A_\rho(T+b_j(\rho)))\ge4h(T)-14h(\rho)-O(1).
\end{equation}
\end{lemma}
\begin{proof}
A failure of independence for a fixed shift gives, for some nonempty $I\subseteq\mathcal C$ and some representative $\delta$ of the finite space $V$, a quadratic equation
\[
 Y^2=\delta\prod_{d\in I}(A_\rho(T+b_j)-d).
\]
For a triple of failures, the three quadratic extensions of $\Q(T)$ are independent because the three finite branch sets are disjoint.  The resulting degree-eight fiber product has at least twelve branch points, all inertia groups have order two, and infinity is unramified.  Riemann--Hurwitz gives genus at least $17$.  Faltings' theorem \cite{Faltings} therefore leaves only finitely many rational $T$ for each finite set of choices.  Formula \eqref{eq:shiftheight} follows from \eqref{eq:Aheight} and $h(b_j)\le h(\rho)+\log30$.
\end{proof}

At every $p\in\mathcal B$ in the affirmative construction, $A\equiv25\pmod p$, by \eqref{eq:rho-odd38} and Proposition~\ref{prop:unit-scale}. For the root trace $R_I$ choose one of the three shifts and set
\begin{equation}\label{eq:Ashift}
 A=A_\rho(R_I+b_j(\rho)),\qquad 1\le j\le3.
\end{equation}
Use only branches on which the rational functions are defined and $A\ne0$. In the affirmative case $\rho>0$, $R_I>0$, so $A>25$. For soundness the height estimate
\begin{equation}\label{eq:scaleheight}
 4h(R_I)\le h(A)+14h(\rho)+O(1)
\end{equation}
holds for either sign of $\rho$. The independence condition of Lemma~\ref{lem:three-shift} is needed only for completeness, after the choice of $w_*$ and the multipliers has fixed $r,\rho$.

\section{Shared local witnesses and the certificate}\label{sec:carrier}

For the input $a$, assign the eight integer indices to the ordered slots
\[
 (f,g,h,k,l,x,y_0,m_1)
\]
of Lemma~\ref{lem:split-source}. Choose eight root points $P_i=[m]Q_i=[n_i]P$ on $E$, where each $Q_i\in E(F)$ is affine. Impose the three nonidentity guards corresponding to $f,k,x$, as well as $w_*\ne0,1$ and $Q_*\ne0$ for the free rational parameter $w_*$. Split the identity cases into the finitely many support branches $I$ in \eqref{eq:tracecoords}, and compute the trace and $A$ by \eqref{eq:Ashift}. Every admissible integer tuple has an affine representation, and conversely the index of each root point is a unique integer. The indices themselves are not quantified rational variables.

Use one anchor $H=[h_0]P$ and seven local points $U_j=[z_j]P$ ($1\le j\le7$). Here $h_0,z_j$ are integer indices in the proof. Let $G=L,N$ denote the two groups. Set $s_L=7$, $s_N=5$, and let $i(G,j)$ be the root slot of the $j$th entry of the corresponding list in \eqref{eq:split-slots}, for $1\le j\le s_G$. Write $f_{L,a}=F_{1,a}$ and $f_{N,a}=F_{2,a}$. The latter also uses the free parameter $w_*$ unchanged in its evaluation. Set
\begin{equation}\label{eq:Vi}
 V_0=H-[64]P,\qquad V_{G,j}=U_j-[64]P_{i(G,j)}.
\end{equation}
Require $H,V_0\ne O$. On $\Z P$, let $z_E=-x_E/y_E$ and $\rho_E=1/x_E$ for nonidentity points, setting them to zero at $O$ (here $x_E,y_E$ denote elliptic coordinates). On $E^\sigma$ use the conjugated coordinates. These functions have no other poles on the indicated cyclic group. Use the truncation $\ell_{k_\ast}$ defined before Lemma~\ref{lem:formal}. Require the denominators below to be nonzero, and define
\begin{equation}\label{eq:Li}
 L_{G,j}=\frac{\ell_{k_\ast}(z_E(U_j))}{\ell_{k_\ast}(z_E(H))}\ \ \t{and}\ \ 
 L'_{G,j}=\frac{\ell_{k_\ast}^\sigma(z_E(V_{G,j}^\sigma))}
 {\ell_{k_\ast}^\sigma(z_E(V_0^\sigma))}.
\end{equation}
The first quotient is common to both groups for the shared slots. The second depends on the active root slot. For each $G$ form the list $\mathcal L_G$ of functions $D$ with a positive integer precision tag $k_D$:
\begin{center}
\begin{tabular}{lcc}
\toprule
Function & Number & Tag\\\midrule
$\rho_E(H),\rho_E(V_0^\sigma)$ &2&1\\
$\rho_E(U_j),\rho_E(V_{G,j}^\sigma)$ ($1\le j\le s_G$)&$2s_G$&1\\
$L_{G,j}-L'_{G,j}$ ($1\le j\le s_G$)&$s_G$&$k_\ast$\\
Source polynomial evaluated at the $L_{G,j}$&1&$k_\ast$\\\bottomrule
\end{tabular}
\end{center}
The last test is $F_{1,a}(L_{L,1},\ldots,L_{L,7})$ or $F_{2,a}(L_{N,1},\ldots,L_{N,5},w_*)$, with arguments in the order of \eqref{eq:split-slots}. There are respectively twenty-four and eighteen entries. Define $\delta_G$ by evaluating the corresponding polynomial on the actual indices, using the same $w_*$ in group $N$. Write $w_*=c/d$ in lowest terms, with $d>0$. For the proof only, set
\[
 \widehat\delta_L=\delta_L\ \ \t{and}\ \ \widehat\delta_N=d^e\delta_N\in\Z.
\]
These integers and the denominator $d$ are not variables or expressions required to be computed by the certificate.

For the scalar list $\mathcal L_s$ use the two rational components in the basis $1,\sqrt2$ of each of
\begin{equation}\label{eq:scalar-tests38}
 \rho_E(H),\quad \rho_E(P_x),\quad z_E(P_x)/z_E(H),
\end{equation}
where $P_x$ is the root point of the index $x$. The tags are $1,1,48$, respectively, for both components. The denominator is nonzero by $H\ne O$ and the nontorsion cyclic carrier. These six entries give the scalar height tests. One further scalar test of tag one, defined below, uses only the existing points $U_6$ and $U_7$.

\begin{lemma}\label{lem:local-comparison38}
Let $p\notin S_{\rm core}$ and let $\mathcal P$ be a place of $F$ over $p$. If every raw test of $\mathcal L_G$ has $v_{\mathcal P}(D)\ge k_D$, then $v_p(\widehat\delta_G)\ge k_\ast$. If the six scalar height tests have their tagged depths at one place, then $p^{48}\mid x$.
\end{lemma}
\begin{proof}
For the ordinary list the tag-one tests put the corresponding points in the formal kernels, since positive valuation of $1/x_E$ characterizes reduction to the identity at good reduction primes, with the identity convention. The two exact index quotients are $z_j/h_0$ and $(z_j-64n_{i(G,j)})/(h_0-64)$. At least one denominator is a $p$-adic unit. Lemma~\ref{lem:formal} gives the valuation of its truncation, and the difference test gives integrality of the other truncation. The same lemma gives integrality of the other exact quotient and applies its approximation assertion to both. Hence
\[
 64 L_{G,j}\equiv h_0 L_{G,j}-(h_0-64)L'_{G,j}
 \equiv64 n_{i(G,j)}\pmod{\mathcal P^{k_\ast}}.
\]
As $p\nmid64$, $L_{G,j}$ agrees with the true integer index to precision $k_\ast$. In group $L$, substitution into the integer polynomial proves the assertion. In group $N$, the polynomial $d^e F_{2,a}(\mathbf Z,c/d)$ has integer coefficients in the five variables $\mathbf Z$, since $e$ bounds the degree in $w_*$. Write $F_{2,a}(\mathbf Z,w_*)=\sum_{i=0}^{e} f_i(\mathbf Z)w_*^i$ with $f_i\in\Z[\mathbf Z]$. The difference of the two evaluations of $\sum_i c^i d^{e-i}f_i(\mathbf Z)$, at the local quotients and at the true integer indices, belongs to $p^{k_\ast}\OO_{\mathcal P}$, since both tuples are integral and congruent to that precision. Multiplying the final raw test by $d^e$ cannot decrease its valuation. Adding these two inequalities proves the assertion even when $v_p(w_*)<0$. No condition of integrality of $w_*$ is imposed in this direction. Since $p$ is unramified, the resulting inequality on the rational discrepancy is the asserted one.

For the scalar list, the componentwise inequalities on the first two elements imply that the points lie in the formal kernel. The componentwise inequality on the quotient implies $v_{\mathcal P}(z_E(P_x)/z_E(H))\ge48$; both basis elements are integral at $p$. Lemma~\ref{lem:formal} gives $v_p(x)-v_p(h_0)\ge48$, and $h_0$ is a nonzero integer. This proves $p^{48}\mid x$.
\end{proof}

\subsection{Two ordinary sources and the covering}
For one free rational $\eta$, set
\begin{equation}\label{eq:t0t1}
 t_0=\eta,\qquad t_1=A\eta,
\end{equation}
and
\begin{equation}\label{eq:R0R1}
 R_0=\frac{F_{11}(\eta)}{(\eta-3)^{12}},\qquad
 R_1'=A^{-10}\frac{F_{10}(A\eta)}{(A\eta-3)^{12}}.
\end{equation}
For every ordinary test $D_i$ of tag $k_i$ in either group put
\begin{equation}\label{eq:normalized-tests}
 \widetilde D_{j,i}=\frac{D_i}{(t_j-3)^{2k_i+1}}\qquad(j=0,1).
\end{equation}
Require $\eta R_0R_1'\ne0$ and $t_0,t_1\notin\mathcal T_*\cup\{3\}$. These are finite rational inequations. Both normalized slopes of both lists are included at each ordinary source, but with the group gate $g_G(t_i)$ for group $G$ at source $i$, as specified below.

The three-source covering is
\begin{equation}\label{eq:threecover}
 -5R_1'z_c=u_1^2,\qquad -R_0(z_c-1)=u_0^2,\qquad r(z_c+3)=u_s^2.
\end{equation}
Use $u_1$ as the one structural coordinate, compute $z_c=u_1^2/(-5R_1')$, and leave $u_0,u_s$ as finite roots.  The local merger of Section~\ref{sec:merger}, with
\begin{equation}\label{eq:sources}
 (a_s,a_0,a_1)=(r,-R_0,-5R_1'),
\end{equation}
produces a single rational argument $B_\Phi$ for the pole test. At each ordinary source retain the $84$ occurrences obtained by taking both normalized slopes of the twenty-four and eighteen entry lists in \eqref{eq:normalized-tests}. Their gates are evaluated at the slope of the source, not the slope of normalization. Add the two norm tests specified next.

\subsection{Local norm tests using the shared points}\label{subsec:positivity38}
Define the following computed rational functions:
\begin{equation}\label{eq:qplus}
 q_+=(\eta-\theta_{10})(\eta-\theta_7),\qquad d_+=\frac{q_+}{(\eta-3)^2},
\end{equation}
\begin{equation}\label{eq:localnorm38}
 S_j=\frac12\Tr_{F/\Q}\left(\frac{z_E(U_j)}{z_E(H)}\right)\ (j=6,7),
 \qquad D_+=\frac{S_6^2+d_+S_7^2-Q_*}{Q_*}.
\end{equation}
The guards ensure that $d_+,Q_*$ and the anchor denominator are nonzero. The scalar source has, in addition to its six height tests, the rational test $D_+$ of tag one and gate one. At source $0$ add the single test $D_+/(\eta-3)^3$ of tag one and gate $g_+(\eta)$. There is no added test at source $1$. In particular the total number of factors in \eqref{eq:WThetaB} is $85+84+7=176$. At the auxiliary contact $\theta_{10}$ neither ordinary comparison group is active. At $\theta_7$ only group $N$ is active among those groups. Thus both slots $6,7$ are available for the norm test at each of these two contacts of source $0$.

For any odd prime $p$, if $v_p(D_+)\ge1$ then $Q_*$ is represented over $\Q_p$ by the form $U^2+d_+V^2$. Indeed, $S_6^2+d_+S_7^2=Q_*(1+D_+)$, and $1+D_+\in1+p\Z_p$ has a square root $h\in\Z_p^\times$ by Hensel's lemma. Hence $Q_*=(S_6/h)^2+d_+(S_7/h)^2$. Both $S_6$ and $S_7$ lie in $\Q$, since they are half-traces; this proves the representation over $\Q_p$ even when $p$ is inert in $F$.

\begin{lemma}\label{lem:qplus}
If the cover \eqref{eq:threecover} with nonzero roots is soluble over the reals and $A\ge25$, then $q_+>0$.
\end{lemma}
\begin{proof}
The first three roots in increasing order are $\theta_2=-1$, $\theta_{10}\in(-9/10,-4/5)$ and $\theta_7=b_*-a_*$. Indeed, $-7/10<b_*<-3/5$ and $0<a_*<1/20$ imply $-3/4<\theta_7<-3/5$, while the remaining roots of the first copy lie in $[-1/2,1/2399]$ or at $0$, and the remaining four roots of the second copy are larger than $\theta_7$. If $q_+<0$, then $\theta_{10}<\eta<\theta_7<-3/5$. Exactly nine factors of $F_{11}(\eta)$ are negative, so $R_0<0$. Also $A\eta<-15$, below all the roots, so all ten factors of $F_{10}(A\eta)$ are negative and $R_1'>0$. The first two equations of \eqref{eq:threecover}, with their roots nonzero, would then require simultaneously $z_c<0$ and $z_c>1$. This is impossible. Since $q_+\ne0$, the lemma follows.
\end{proof}

\subsection{The square-denominator test}\label{subsec:BB}

For nonzero $B_\Phi$ put
\begin{equation}\label{eq:bdef}
 b=112B_\Phi.
\end{equation}
The terminal equation is
\begin{equation}\label{eq:BB}
 z_p^2=(1-b)(s_p^4-b^2)(t_p^4-b^2).
\end{equation}
Here $s_p,t_p$ are the two free rational parameters, while $z_p$ is a finite algebraic choice and the zero root is allowed.  Bhargava and Burungale \cite[Lemma 2.2]{BB} prove that if
\begin{equation}\label{eq:BBconditions}
 0<b<1\ \ \t{and}\ \ v_2(b)\ge3,
\end{equation}
the positive reduced denominator of $b$ is a square, and the positive squarefree representative of $[b]$ is $7$ modulo $8$, then \eqref{eq:BB} has a rational solution.  We verify these conditions in the affirmative case of the source problem.

The soundness implication is elementary.  If $v_p(b)$ is negative and odd, then $1-b$ has odd valuation.  In each quartic difference the two terms have unequal valuations, one a multiple of four (or $+\infty$ for a zero parameter) and the other congruent to two modulo four; hence each quartic difference has finite even valuation, contradicting \eqref{eq:BB}.  Thus the terminal equation excludes every exterior negative odd valuation of $B_\Phi$.

\section{Simultaneous avoidance in quadratic prime patterns}\label{sec:primes}

For sufficiently large primes $p$ satisfying
\begin{equation}\label{eq:primecongruences}
 p\equiv1\pmod8\ \t{and}\  p\equiv3\pmod7,
\end{equation}
the two prime ideals of $\OO_F$ above $p$ give opposite quadratic twists of $y^2=x^3-x$: the product of their twisting parameters is 7, a nonsquare modulo $p$. Their reduction-group orders form the unordered pair
\begin{equation}\label{eq:orders}
 \{d_1(p),d_2(p)\}=\{p+1-2a,p+1+2a\},
\end{equation}
where we write $p=a^2+b^2$ with $a$ odd and $b$ even.
These are the complex-multiplication formulas used in \cite[Section 6.1]{OAI}. Both orders are divisible by four; their difference is $4a$, so exactly one has 2-adic valuation two. For large $p$, the orders are positive and less than $2p$, and divisibility by four rules out equality with $p$. Hence $p\nmid d_1(p)d_2(p)$.

For a Gaussian prime generator $\pi=a+bi$, write
\begin{equation}\label{an:eq:orders}
 d^+(p)=(a+1)^2+b^2\ \ \t{and}\ \  d^-(p)=(a-1)^2+b^2.
\end{equation}
Conjugation of $\pi$ does not change these integers. The congruence $\pi\equiv9\pmod{16}$ gives $v_2(d^+)=2$ and $v_2(d^-)=6$, so it fixes a common orientation: $d_1=d^+$ and $d_2=d^-$. The next theorem supplies all primes needed for one shared witness group.

\begin{theorem}\label{an:thm:patterns}
Let $c_1,\ldots,c_k$ be distinct integers, $k\ge2$. Let $q\ge1$, let $M$ be odd and squarefree with $7\nmid M$, and prescribe positive lower bounds $Y,H$. There are positive integers $u,w$, with $q\mid w$ and $u/w^2>H$, such that all
\[
 p_j=u+c_jw^2\quad(1\le j\le k)
\]
are primes exceeding $Y$. Each has a Gaussian generator $\pi_j$ satisfying
\begin{equation}\label{an:eq:ray}
 \pi_j\equiv9\pmod{16},\qquad
 \pi_j\equiv1+3i\pmod7,\qquad
 \pi_j\equiv i\pmod M.
\end{equation}
The self-gcds equal four. For $j\ne h$, any common factor of either order at $p_j$ and either order at $p_h$ is a power of two. In particular,
\begin{equation}\label{an:eq:allcross}
 \gcd(d^+(p_j),d^-(p_h))=4\qquad(1\le j,h\le k).
\end{equation}
All orders are coprime to $M$.
\end{theorem}

The theorem concerns any fixed finite $k$; its constants need not be uniform in $k$, the coefficients, or $M$. Repeated positions can be deleted before applying it. Prescribed square-denominator dilations are included by replacing $c_j$ by $q^2c_j$ and then replacing the parameter by $qw$.

The proof uses the uniform quadric congruence estimate of Browning--Loughran \cite[Theorem 4.1]{BL}, the polynomial transference estimates of Tao--Ziegler \cite{TZpoly}, the unconditional \emph{affine-linear} Artin-symbol theorem of Frei--Tinkov\'a \cite[Theorem 1.4]{FT}, and the geometric-sieve inputs of Browning--Heath-Brown \cite[Lemmas 2.1 and 4.1]{BHB}. These published statements have distinct roles. In particular, the Artin theorem in \cite{FT} does not itself count quadratic polynomial patterns. Proposition~\ref{an:prop:Artinpoly} below supplies that extra step by adapting the transference proof in \cite{TZpoly}; Appendix~\ref{an:app:large} proves the additional nonhomogeneous large-prime estimate. The prime-pattern conclusion therefore depends on both of these \emph{internal} arguments, including their fixed-field and order-of-limits verifications.
In the present application the two-scale construction produces twenty-one distinct coefficients. The theorem is stated for every fixed finite $k$, so no uniformity in $k$ is required.

\subsection{Local counts and primitive quadric congruences}\label{an:sec:counts}
Fix $\kappa\in\Z\setminus\{0\}$ and write
\begin{equation}\label{an:eq:quadric}
 Q_\kappa=a^2+b^2-c^2-d^2-\kappa w^2.
\end{equation}
Set $G_\eps=(a+\eps)^2+b^2$, $H_\eta=(c+\eta)^2+d^2$, where $\eps,\eta\in\{1,-1\}$. For an odd $\ell\nmid\kappa$, the nonzero quadric residue set has $\ell^4-1$ elements. The collision set
\[
 \Omega_\ell=\{Q_\kappa=G_\eps=H_\eta=0\}\subset\F_\ell^5
\]
contains no zero vector and has at most $4\ell^2$ elements. If $\ell\equiv3\pmod4$, it has exactly one element. Indeed, in the split case the four branches $a+\eps=\pm ib$, $c+\eta=\pm id$ each leave a single equation, linear in $d$, in $b,d,w$. In the inert case both shifted norm-zero equations force $(a,b,c,d)=(-\eps,0,-\eta,0)$, and then $w=0$.

Let $L_j=a^2+b^2+\gamma_jw^2$, where the distinct integers $\gamma_1,\ldots,\gamma_k$ include $0,-\kappa$. Outside a fixed finite prime set $S$, containing the divisors of $2\kappa\prod_{i<j}(\gamma_i-\gamma_j)$, the fraction $g(p)$ of nonzero quadric residues on which $\prod_jL_j=0$ is
\begin{equation}\label{an:eq:localg}
 g(p)=\frac{k+2\chi(p)}p+O_k(p^{-2})\ \t{with}\  \chi(p)=\left(\frac{-1}p\right)=(-1)^{(p-1)/2}.
\end{equation}
Here and below $S$ may be enlarged by finitely many small primes. For an explicit verification, put $\nu_0=p+(p-1)\chi(p)$. The numerator of $g(p)$ is
\begin{equation}\label{an:eq:exactlocal}
 (p-1)(p-\chi(p))\bigl(2\nu_0+(k-2)(p-\chi(p))\bigr)+\nu_0^2-1,
\end{equation}
and the denominator is $p^4-1$. For an intermediate $L_j$, its zero count is $(p-1)(p-\chi)^2+\nu_0^2$; for either norm position it is $\nu_0((p-1)(p-\chi)+\nu_0)$. All pairwise intersections coincide with $w=0$ and both norms zero, a set of size $\nu_0^2$. Inclusion--exclusion and removal of the origin give \eqref{an:eq:exactlocal}.

Let $\psi$ be a nonnegative smooth function of compact support away from the singular point of the real cone. The real density is
\[
 \sigma_\infty(\psi)=\int_{Q_\kappa(\mathbf x)=0}
       \frac{\psi(\mathbf x)}{\|\nabla Q_\kappa(\mathbf x)\|}\,dS(\mathbf x),
\]
where $dS$ is Euclidean surface measure. For a prime $p$, the full cone density is
\[
 \sigma_p=\lim_{e\to\infty}p^{-4e}
       \#\{\mathbf x\bmod p^e:Q_\kappa(\mathbf x)=0\}.
\]
These are the normalizations in \cite[Section 4]{BL}. For squarefree $m$ supported outside $S$, choose arbitrary subsets $\Omega_p$ of the nonzero quadric residues. Browning--Loughran's estimate, followed by primitive-vector inversion, gives
\begin{equation}\label{an:eq:primitive}
 \sum_{\substack{\mathbf x\in\Z^5_{\rm prim},\ Q_\kappa(\mathbf x)=0\\
                  \mathbf x\bmod p\in\Omega_p\ (p\mid m)}}
 \psi(\mathbf x/B)
 =C_\kappa\sigma_\infty(\psi)B^3
     \prod_{p\mid m}\frac{|\Omega_p|}{p^4-1}
 +O_{\kappa,\psi,\varepsilon}(B^{5/2+\varepsilon}m^{5/2+\varepsilon}).
\end{equation}
The constant $C_\kappa>0$ is independent of $m,\psi$ and the residue subsets. Here $\Z^5_{\rm prim}$ denotes vectors with coordinate gcd one.

We detail the inversion because the collision conditions are not homogeneous. For a common scaling divisor $h$, the congruence subsets become $h^{-1}\Omega_p$. If $p\mid h$ there is no contribution, since $0\notin\Omega_p$; otherwise their cardinalities are unchanged. Apply \cite[Theorem 4.1]{BL} at scale $B/h$ and sum with $\mu(h)$, where $\mu$ denotes the M\"obius function in this sieve argument. The errors are summable since $\sum h^{-5/2-\varepsilon}<\infty$. The main sum contributes $\prod_{p\nmid m}(1-p^{-3})$, with a truncation error $O(B)$. At a good prime the full cone density is
\[
 \sigma_p=\frac{1-p^{-4}}{1-p^{-3}}.
\]
This identity follows by separating primitive residues from vectors divisible by $p$. Thus the primitive local factors are $1-p^{-4}$ and the prescribed-residue ratio is exactly $|\Omega_p|/(p^4-1)$. Finitely many bad-prime factors are absorbed into $C_\kappa$. The uniformity in the subsets in the published theorem justifies applying it to every $h^{-1}\Omega_p$. The argument requires $(m,2\disc Q_\kappa)=1$ and nonzero vectors in every local residue set, exactly as imposed by the fixed enlargement of $S$. No residue condition at a prime dividing the inversion variable $h$ survives, since such a prime would force reduction to the excluded zero vector.

\subsection{The intermediate prime range}\label{an:sec:medium}
We prove an upper bound retaining all $k$ prime conditions. Work in a region with both norm coordinates between fixed positive multiples of $B^2$, with every $L_j$ between such multiples, and with $|w|\le\tau B$, where $0<\tau\le\tau_0$. Choose a smooth majorant $\psi_\tau$ with this support enlarged by fixed factors and with
\begin{equation}\label{an:eq:realmass}
 \sigma_\infty(\psi_\tau)\ll_\kappa\tau.
\end{equation}
The implied constant can be uniform in $\tau$. To see this, use polar coordinates in $(a,b)$ and $(c,d)$ on the support. Integrating the delta constraint $r_1^2-r_2^2-\kappa w^2=0$ cancels the factor $r_2$ in the second polar measure. The remaining radial and angular integrals are bounded, and the $w$-interval has length $O(\tau)$. Bounds for derivatives of $\psi_\tau$ may depend on $\tau$; they will remain in error constants.

\begin{lemma}\label{an:lem:G}
For $g$ in \eqref{an:eq:localg}, after a fixed enlargement of $S$ we have $0<g(p)<1$. Put $h(p)=g(p)/(1-g(p))$, extend $h$ multiplicatively on squarefree integers, and omit a further arbitrary prime $\ell\notin S$. Then
\[
 G_\ell(z):=\sum_{\substack{d<z\text{ squarefree}\\p\mid d\Rightarrow p\notin S\cup\{\ell\}}}h(d)
 \gg (\log z)^k
\]
uniformly in $\ell$.
\end{lemma}
\begin{proof}
The classical prime estimates
\[
 \sum_{p\le y}\frac1p=\log\log y+O(1),\qquad
 \sum_{p\le y}\frac{\chi(p)}p=O(1),\qquad
 \sum_{p\le y}\frac{\log p}{p}\ll\log y
\]
follow from Mertens' theorem and the prime number theorem in the fixed progressions modulo four; see \cite{Dav}. Together with \eqref{an:eq:localg} they give
\[
 \prod_{\substack{p\le y\\p\notin S\cup\{\ell\}}}(1+h(p))\asymp(\log y)^k
\]
uniformly in the omitted prime; its single factor is bounded. Take $y=z^a$ for a sufficiently small fixed $a>0$. Normalize the terms in this Euler product to a probability distribution on squarefree products $d$ of these primes. Its expected $\log d$ is $\sum g(p)\log p\ll_k\log y$. With $a$ small enough it is at most $\tfrac12\log z$. At least half the mass therefore has $d<z$, up to an immaterial adjustment of $a$. These terms occur in $G_\ell(z)$.
\end{proof}

For fixed $\ell\notin S$, impose the collision residues $\Omega_\ell$ and sieve $\prod_jL_j$ by primes outside $S\cup\{\ell\}$. Every vector with both norm coordinates prime is primitive, so \eqref{an:eq:primitive} applies to all vectors being counted. For squarefree $d$ coprime to $\ell$, the count with $d\mid\prod_j L_j$ has the form
\begin{equation}\label{an:eq:sievesequence}
 A_{\ell,d}=X_\ell g(d)+r_{\ell,d},\qquad
 X_\ell\ll\frac{\tau B^3}{\ell^2},\qquad
 r_{\ell,d}\ll_{\tau,\varepsilon} B^{5/2+\varepsilon}(\ell d)^{5/2+\varepsilon}.
\end{equation}
All fixed polynomial data may enter these constants, but $\ell,d$ do not.

For clarity, the required Selberg weights can be written explicitly:
\[
 \lambda_d=\frac{\mu(d)}{G_\ell(z)}
   \prod_{p\mid d}(1-g(p))^{-1}
   \sum_{\substack{r<z/d\\(r,d)=1}}h(r),
\]
with all integers in the sums squarefree and supported on the sieving primes; other weights are zero. They satisfy $\lambda_1=1$, $|\lambda_d|\ll_\varepsilon d^\varepsilon$, and
\begin{equation}\label{an:eq:selbergmain}
 \sum_{d,e<z}\lambda_d\lambda_e g([d,e])=G_\ell(z)^{-1}.
\end{equation}
Here $[d,e]=\lcm(d,e)$. The identity follows by writing the quadratic form as
\[
 \sum_r\frac1{h(r)}
       \(\sum_{r\mid d}\lambda_dg(d)\)^2
\]
and applying finite divisor inversion. The bound for $\lambda_d$ follows from $g(p)\ll_k1/p$ and the fact that its inner sum is at most $G_\ell(z)$.

On prime vectors, every $L_j\asymp B^2$ exceeds $z$, so the square of $\sum_{d\mid\prod L_j}\lambda_d$ is one. It bounds their indicator from above on the whole sequence. Applying \eqref{an:eq:selbergmain} and Lemma~\ref{an:lem:G} yields
\begin{equation}\label{an:eq:oneell}
\begin{aligned}
 \#\{\text{prime vectors with collision at }\ell\}
 &\ll \frac{\tau B^3}{\ell^2(\log z)^k}\\
 &\quad +O_{\tau,\varepsilon}\bigl(B^{5/2+\varepsilon}\ell^{5/2+\varepsilon}z^{7+4\varepsilon}\bigr).
\end{aligned}
\end{equation}
For the error, use $[d,e]\le de$ and sum the absolute values in \eqref{an:eq:sievesequence}; no cancellation of remainders is assumed.

Set $z=B^{1/100}$ and $\delta=1/100$. Sum \eqref{an:eq:oneell} over $L<\ell\le B^\delta$, with $L$ beyond $S$. Since $\sum_{\ell>L}\ell^{-2}\ll1/(L\log L)$, the main term is
\begin{equation}\label{an:eq:mediummain}
 \ll\frac{\tau B^3}{(\log B)^kL\log L}.
\end{equation}
With, for example, $\varepsilon=1/1000$, the summed error has exponent at most
\[
 \frac52+\varepsilon+\frac{7+4\varepsilon}{100}
       +\frac{7/2+\varepsilon}{100}<\frac{27}{10}.
\]
Thus its contribution, even after multiplication by $(\log B)^k$, is $O_\tau(B^{14/5})$. The constants in the main term are independent of $\tau,L$.

The self-gcd case uses an ordinary lattice sieve in $(a,b,w)$, because an odd prime dividing both orders divides $a$ and $b^2+1$. The collision set modulo $\ell$ has at most $2\ell$ elements. For the prime forms $a^2+b^2+\gamma_jw^2$, with exactly one $\gamma_j=0$, their union of zero loci has density
\[
 g_0(p)=\frac{k+\chi(p)}p+O_k(p^{-2}).
\]
Indeed its exact cardinality is $(k-1)p^2+(p-k+1)\nu_0$. The same Euler-product proof gives sieve denominator $\gg(\log z)^k$. For arbitrary residue subsets modulo $m$, smooth lattice counting has error $O_\tau(B^2m+Bm^2+m^3)$, obtained by summing the boundary error of the residue-class boxes. Taking $m=\ell[d,e]\le B^{3/100}$, and summing over $\ell,d,e$, gives an error smaller than $B^{27/10}$. This proves the counterpart of \eqref{an:eq:mediummain} for self-gcds as well.

\begin{proposition}\label{an:prop:fulltail}
Fix distinct integers $c_1,\ldots,c_k$. For $N=B^2$, consider
\[
 N/2<u\le N,\qquad 1\le w\le\tau B,
 \qquad p_j=u+c_jw^2,
\]
with $\tau_0$ small enough that all $p_j\asymp N$. Restrict, if desired, to any fixed ray-class conditions under which the norms have chosen Gaussian generators. Give each prime pattern weight $\prod_j\log p_j$. There are constants $L_0,C$ independent of $\tau,L$ such that the total weight of patterns with a self- or cross-order obstruction prime greater than $L\ge L_0$ is at most
\begin{equation}\label{an:eq:weightedtail}
 \frac{C\tau B^3}{L\log L}
       +O_\tau(B^{3-1/400}).
\end{equation}
An obstruction means an odd common divisor of $d_j^+,d_j^-$, or of an order at one position and an order at another.
\end{proposition}
\begin{proof}
For a fixed cross pair, choose its two Gaussian representations and set $\kappa=c_j-c_h$. Their coordinates lie on \eqref{an:eq:quadric}; every other prime form has coefficient $\gamma_i=c_i-c_j$. A pattern maps injectively to such vectors after a deterministic choice of generators: $u$ is recovered from the first norm and $w$. There are finitely many pairs and sign choices. Apply the medium-prime estimate to each, multiplying by $O((\log B)^k)$. Apply Corollary~\ref{an:cor:weightedlarge} for primes greater than $B^{1/100}$. Its bound is $O(B^{3-1/200})$, which is absorbed by the displayed error. The self-gcd large-prime bound is included in Appendix~\ref{an:app:large}. Ray restrictions may be dropped in upper bounds. If $L>B^{1/100}$, the large-prime estimate alone suffices.
\end{proof}

\section{Polynomial patterns with fixed Artin conditions}\label{an:sec:transfer}
\subsection{The weights and the counting statement}
Let $\mathcal K/\Q$ be a fixed finite Galois extension. Write $\mathcal K^{\rm ab}$ for its maximal subextension abelian over $\Q$, and let $\mathcal C_1,\ldots,\mathcal C_k$ be conjugacy-invariant subsets of $\Gal(\mathcal K/\Q)$. For an unramified rational prime $p$, $\operatorname{Frob}_p$ denotes its Frobenius conjugacy class. Define, on all integers,
\[
 \Lam_j(n)=\begin{cases}\log p&\t{if}\ n\ \text{is an unramified prime}\ p\ \t{with}\ \operatorname{Frob}_p\subseteq\mathcal C_j,\\0&\text{otherwise.}\end{cases}
\]
Let $D$ be divisible by the conductor of $\mathcal K^{\rm ab}$, namely the least $d$ such that $\mathcal K^{\rm ab}\subseteq\Q(\zeta_d)$, where $\zeta_d$ is a primitive $d$th root of unity. For $(b,D)=1$, let $\sigma_b(\zeta_D)=\zeta_D^b$, and put
\begin{equation}\label{an:eq:conditionalmu}
 \mu_j(b)=\frac{\#\{g\in\mathcal C_j:
    g|_{\mathcal K^{\rm ab}}=\sigma_b|_{\mathcal K^{\rm ab}}\}}
    {[\mathcal K:\mathcal K^{\rm ab}]}.
\end{equation}
Set $\mu_j(b)=0$ on nonunit classes. Chebotarev applied to $\mathcal K\Q(\zeta_D)$ identifies $\mu_j(b)\in[0,1]$ with the conditional density among primes congruent to $b$ modulo $D$. In the notation of \cite[(1.10)--(1.12)]{FT}, it is the sum of $\varphi(D)\eta_{\mathcal K,C}(b,D)$ over conjugacy classes $C\subseteq\mathcal C_j$. Here $\varphi$ is Euler's totient.

For a finite nonempty set or multiset $X$, $\mathbb E_{x\in X}$ denotes its normalized average, counting multiplicities. Write $[J]=\{1,\ldots,\lfloor J\rfloor\}$ and $1_{\mathcal A}$ for the indicator of a condition $\mathcal A$. Constants may depend on the fixed field, subsets and coefficients. These data never vary with the counting parameter.

\begin{proposition}\label{an:prop:Artinpoly}
Fix $k\ge2$ and distinct integers $c_1,\ldots,c_k$. There is a positive function $\omega(N)\to0$ such that, for every integer-valued $T=T(N)$ satisfying
$ \omega(N)\sqrt N\le T$ and $ T=o(\sqrt N)$,
one has
\begin{equation}\label{an:eq:artinasymptotic}
 \sum_{1\le u\le N}\sum_{1\le w\le T}\prod_{j=1}^k\Lam_j(u+c_jw^2)
 =\mathfrak S\,NT+o(NT),
\end{equation}
where
\begin{equation}\label{an:eq:Artinconstant}
 \mathfrak S=
 \left(\frac D{\varphi(D)}\right)^k
 \mathbb E_{u,w\bmod D}\prod_j\mu_j(u+c_jw^2)
 \prod_{p\nmid D}\beta_p,
\end{equation}
and
\[
 \beta_p=\mathbb E_{u,w\bmod p}
    \prod_{j=1}^k\frac p{p-1}1_{p\nmid u+c_jw^2}.
\]
The Euler product converges absolutely. After enlarging $\omega$ if necessary, the same assertion holds on $N/2<u\le N$, with main term $\mathfrak S NT/2$.
\end{proposition}

We adapt the transference argument of \cite{TZpoly}, replacing the affine-linear prime-correlation input by \cite[Theorem 1.4]{FT}. For each invocation, the Galois field, all conjugacy conditions, and the entire affine-linear complexity type are fixed before the counting parameter tends to infinity. Theorem~1.4 of \cite{FT} has an error term uniform in the convex region and bounded-size affine-linear systems, but its theorem statement does not permit the field itself to grow with the counting parameter. We therefore first establish the needed polynomial-pattern assertion for every \emph{fixed} field, and introduce the increasing finite ray cutoff only afterwards. The distinctions between the two prime majorants, the box norms and the required scales are made explicit below.

\subsection{Difference-multiset norms and global uniformity}
For finite multisets $Q_1,\ldots,Q_a$ in $\Z/n\Z$ and a real function $g$ on that group, define
\begin{equation}\label{an:eq:boxnormdef}
 \|g\|_{\Box^a_{Q_1,\ldots,Q_a}}^{2^a}
 =\mathbb E_{x\in\Z/n\Z}
   \mathbb E_{\substack{h_i,h'_i\in Q_i\\1\le i\le a}}
   \prod_{\omega\in\{0,1\}^a}
       g\left(x+\sum_i\omega_i(h_i-h'_i)\right).
\end{equation}
Write $\|g\|_{U^a_Q}$ when every $Q_i=Q$. These are the box and progression Gowers norms of \cite[Section 1.8]{TZpoly}; they satisfy the triangle inequality. The difference multiset $Q_i-Q_i$ counts differences with their ordered-pair multiplicities. A one-sided increment average from $Q_i$ would not be the norm in \eqref{an:eq:boxnormdef}.

Fix $0<\rho<1$ and, initially, a large fixed cutoff $v$. Put
\begin{equation}\label{an:eq:analyticW}
 W=D\prod_{p\le v,\ p\nmid D}p,\qquad A=W^{1/\rho},\qquad R=n^\rho.
\end{equation}
For $b\in[W]$ coprime to $W$, define on $\Z/n\Z$, using representatives in $[n]$, we have
\begin{equation}\label{an:eq:modifiedweight}
 f_{j,b}(x)=\frac{\varphi(W)}W\Lam_j(Wx+b)1_{R\le x\le n}\ \ \t{and}\ \ g_{j,b}=f_{j,b}-\mu_j(b).
\end{equation}
Arguments of $\mu_j$ are reduced modulo $D$. Since $D\mid W$, one has $\mathcal K\cap\Q(\zeta_W)=\mathcal K^{\rm ab}$.

\begin{lemma}\label{an:lem:Artinuniform}
Fix $a\ge2$ and positive constants $C_1<C_2<C_3$. Uniformly for
\[
 1\le q\le A^{C_1},\qquad
 nA^{-C_3}\le J\le nA^{-C_2},\qquad (b,W)=1,
\]
one has
\begin{equation}\label{an:eq:globaluniform}
 \|g_{j,b}\|_{U^a_{q[J]}}^{2^a}
 \ll \frac1{\log v}+o_v(1).
\end{equation}
Here $o_v(1)\to0$ as $n\to\infty$ with $v$ fixed; the implied constant is independent of $v,n,b,q,J$. A sufficiently slow diagonal choice $v\to\infty$ therefore makes these norms $o(1)$.
\end{lemma}
\begin{proof}
Expand the norm with independent $\boldsymbol h,\boldsymbol h'\in[J]^a$. For a nonempty vertex set $I$, consider the correlation of $f_{j,b}$ at those vertices. To handle cyclic reduction, partition the box according to the integers $m_\omega$ for which
\[
 y_\omega=x+q\omega\cdot(\boldsymbol h-\boldsymbol h')-m_\omega n
 \quad\text{lies in }(0,n].
\]
For large $v$, $aqJ<n$, so $m_\omega\in\{-1,0,1\}$. Add $y_\omega\ge R$ for $\omega\in I$. Each cell is convex, and its prime arguments are
\begin{equation}\label{an:eq:wrappedforms}
 L_\omega=W\bigl(x+q\omega\cdot(\boldsymbol h-\boldsymbol h')-m_\omega n\bigr)+b.
\end{equation}
Their linear parts are pairwise nonproportional over $\Q$: their common $x$ coefficient forces a proportionality constant to be one, and then the step coefficients force equality of the vertices. Thus the system has finite complexity.

Apply \cite[Theorem 1.4 and (1.15)]{FT} on each cell, expanding $\mathcal C_j$ into its finitely many conjugacy classes. The coefficients and the constant terms divided by $n$ are bounded in terms of fixed $v$. Uniformity over bounded-size systems and convex regions gives error $o_v(n^{2a+1})$, even as $m_\omega n$ and $J$ vary. Since $J\asymp_v n$, this is $o_v(nJ^{2a})$.

All forms are $b$ modulo $W$. After multiplication by $(\varphi(W)/W)^{|I|}$, the Artin factor is $\mu_j(b)^{|I|}$ and the factors at primes dividing $W$ cancel. For $p>v$ with $p\nmid q$, the linear parts remain pairwise nonproportional over $\F_p$. Each zero hyperplane has density $1/p$, and each pairwise intersection has density $1/p^2$. Inclusion--exclusion gives a local factor $1+O_a(p^{-2})$, uniformly in the constant terms. At $p\mid q$ it is $1+O_a(p^{-1})$. These factors are positive for large $v$. Since $\log q\le C_1\log A=O(v)$, there are $O(v/\log v)$ prime divisors of $q$ exceeding $v$. The product of all factors outside $W$ is consequently $1+O(1/\log v)$, uniformly in the cells. The bound $\sum_{p\le v}\log p\ll v$ used here is the classical Chebyshev estimate; see \cite{Dav}.

Before the cutoff, the cells fill the original box up to boundaries. For fixed step variables, translation modulo $n$ preserves Lebesgue measure in $x$. The cutoffs remove at most $|I|R/n$ of the normalized volume. Endpoint changes give $o_v(1)$. Thus the volumes sum to $1+o_v(1)$ after normalization, and summing the local-factor errors does not enlarge their order.

The weights in \cite{FT} include proper prime powers. For fixed steps and one vertex, the nonzero $x$ coefficient $W$ and the range $O_v(n)$ allow at most $O_v(n^{1/2}\log n)$ such values. Bound the remaining factors by $O_v(\log n)$ and sum over steps. The normalized error is $O_v(n^{-1/2}(\log n)^{2^a+1})=o_v(1)$. We obtain
\[
 \mathbb E_x\mathbb E_{\boldsymbol h,\boldsymbol h'\in[J]^a}
 \prod_{\omega\in I}
 f_{j,b}\bigl(x+q\omega\cdot(\boldsymbol h-\boldsymbol h')\bigr)
 =\mu_j(b)^{|I|}+O\l(\f1{\log v}\r)+o_v(1).
\]
The empty correlation is one. The centered main terms cancel, since
\[
 \sum_{I\subseteq\{0,1\}^a}
  (-\mu_j(b))^{2^a-|I|}\mu_j(b)^{|I|}=0.
\]
This proves \eqref{an:eq:globaluniform}. For fixed $v$ there are finitely many residues, $q$ values and vertex systems; uniformity in the convex region already covers the varying $J$. A slow diagonal choice proves the last assertion.
\end{proof}

\subsection{The majorant, local scales, and the \texorpdfstring{$W$}{W}-trick}

\begin{proof}[Proof of Proposition~\ref{an:prop:Artinpoly}]
Let $\Lambda'_{\rm pr}(n)$ equal $\log n$ on positive primes and zero otherwise, and put
\[
 \Lambda'_{b,W;\kappa}(x)=\frac{\varphi(W)}W
       \Lambda'_{\rm pr}(Wx+b)1_{n^\kappa\le x\le n}
       \qquad(0<\kappa<1).
\]
Then $0\le f_{j,b}\le\Lambda'_{b,W;\rho}$. Fix the polynomial complexity first. Let $D',C$ be bounds for the local norm order and coefficient exponents in the generalized von Neumann theorem below. Choose an integer $B>C+2$ large enough for all the finitely many coefficient bounds, and a global norm order $a\ge2$ sufficiently large for the corresponding applications of \cite[Theorem 2.5 and Corollary 3.9]{TZpoly}. Next choose $0<\rho<1$ sufficiently small for all applications of \cite[Sections 3--4]{TZpoly}, and put $\rho_* =\rho/B$. All these constants are fixed before $v,n,N$ vary.

\emph{The ordinary majorants.}
For $\kappa\in\{\rho,\rho_*\}$, put $R_\kappa=n^\kappa$ and use the majorant of \cite[(3.5)]{TZpoly}:
\[
 \nu_{b,\kappa}(x)=\frac{\varphi(W)}W\log R_\kappa
 \left(\sum_{d\mid Wx+b}\mu(d)
             \chi_0\left(\frac{\log d}{\log R_\kappa}\right)\right)^2.
\]
Here $\mu$ is the M\"obius function, and $\chi_0$ is smooth, even, supported on $[-1,1]$, positive at zero, with $\int_0^1|\chi'_0(t)|^2\,dt=1$. It is distinct from the quadratic character $\chi(p)$ in Section~\ref{an:sec:counts}. If $x$ belongs to the support of $\Lambda'_{b,W;\kappa}$, then $Wx+b$ is prime and exceeds $R_\kappa$. Only the divisor $d=1$ contributes to the displayed sum. For the slow choice $W=n^{o(1)}$, we therefore have, uniformly in $b,v,n$,
\[
 \Lambda'_{b,W;\kappa}\ll_\kappa\nu_{b,\kappa}
 \ \ \t{and}\ \  f_{j,b}\le\Lambda'_{b,W;\rho}
       \le\Lambda'_{b,W;\rho_*}.
\]
Use $\kappa=\rho$ for the generalized von Neumann argument and $\kappa=\rho_*$ for the dense-model argument. These two invocations need not use the same majorant. Neither majorant imposes an Artin condition.

\emph{The fixed conductor factor and polynomial forms.}
The statements in \cite[Sections 3--4]{TZpoly} use the squarefree primorial
$W_0=\prod_{p\le v}p$. Once $v$ exceeds all prime divisors of $D$, our modulus is
\[
 W=E_DW_0\ \ \t{with}\ \ E_D=\prod_{p\mid D}p^{v_p(D)-1}\ \t{and}\  \frac{\varphi(W)}W=\frac{\varphi(W_0)}{W_0}.
\]
The integer $E_D$ is fixed. We give the divisor expansion, local factors,
and uniform error needed to carry the polynomial-forms argument of
\cite[Appendix A]{TZpoly} over from $W_0$ to $W$.

Fix a polynomial-form system of one of the fixed complexity types used in that
argument, with $s$ factors and parameter vector $\boldsymbol y$; its
coefficients may vary within the prescribed bounds. In the common-base-variable
form relevant here, its arguments are
\[
 \begin{aligned}
 F_j(\boldsymbol y)&=W\psi_j(\boldsymbol y)+b_j,\ \ 
 \psi_j(x,\boldsymbol h)&=x+P_j(\boldsymbol h),
       \ \ (b_j,W)=1\quad(1\le j\le s).
 \end{aligned}
\]
The degrees are bounded, and each $P_i-P_j$ is nonconstant, as required
in the polynomial-forms estimates being invoked. All its positive-degree
coefficients have size at most $A^{C_s}$ for a fixed $C_s$.
The parameter boxes have smallest side $H\ge n^{\eta}$ for some fixed
$\eta>0$; the same argument applies to the bounded-complexity
cells used for cyclic representatives, with their number absorbed into
$A^{C_s}$ if necessary. The bounds $s,C_s,\eta$ are fixed
before $v,n$ vary. The exponent $\kappa\in\{\rho,\rho_*\}$ can be, and was,
chosen sufficiently small relative to these fixed complexities.

Put $R=n^\kappa$ and
$\lambda_R(d)=\mu(d)\chi_0(\log d/\log R)$. Expanding the squares gives
\begin{equation}\label{an:eq:conductor-divisor-expansion}
 \prod_{j=1}^s\nu_{b_j,\kappa}(\psi_j(\boldsymbol y))
 =\left(\frac{\varphi(W)}W\log R\right)^s
 \sum_{\boldsymbol d,\boldsymbol d'}
 \prod_{j=1}^s\lambda_R(d_j)\lambda_R(d'_j)
 1_{[d_j,d'_j]\mid F_j(\boldsymbol y)}.
\end{equation}
Only squarefree $d_j,d'_j\le R$ occur. If $p\mid W$, then
$F_j(\boldsymbol y)\equiv b_j\not\equiv0\pmod p$; hence every nonzero term
has $(d_jd'_j,W)=1$. Write $m_j=[d_j,d'_j]$ and
$m=[m_1,\ldots,m_s]\le R^{2s}$. It is squarefree and coprime to $W$.
For $p\mid m$, put $I_p=\{j:p\mid m_j\}$ and
\[
 \rho_p(I)=p^{-r}\#\{\boldsymbol y\bmod p:
          F_j(\boldsymbol y)\equiv0\pmod p\ (j\in I)\}\ \t{with}\ 
  r=\dim\boldsymbol y,
\]
where we view $\rho_p(\varnothing)$ as $1$. The Chinese Remainder Theorem gives the
\emph{exact} density $\prod_{p\mid m}\rho_p(I_p)$ modulo $m$.
On any of the parameter boxes or wrap cells $\Omega$, the standard
residue-class lattice estimate used in \cite[Appendix A]{TZpoly} gives,
relative to its containing box $\mathcal B$,
\[
 \mathbb E_{\boldsymbol y\in\mathcal B}1_{\Omega}(\boldsymbol y)
     \prod_j1_{m_j\mid F_j(\boldsymbol y)}
 =\frac{\operatorname{vol}(\Omega)}{\operatorname{vol}(\mathcal B)}
       \prod_{p\mid m}\rho_p(I_p)
   +O_s\left(\frac{A^{C_s}m}{H}\right).
\]
An increase of $C_s$ absorbs the fixed complexity of the cells.
Since there are at most $R^{2s}$ divisor tuples, the total error in
\eqref{an:eq:conductor-divisor-expansion} is
\begin{equation}\label{an:eq:conductor-crt-error}
 O_s\left(\frac{A^{C_s}(\log R)^sR^{4s}}H\right)=o(1).
\end{equation}
Here $A=n^{o(1)}$, and $\rho$ was chosen small enough that
$4s\kappa<\eta/2$ for all the finitely many systems. In particular,
this error is uniform in the unit residues $b_j$, the cells, and their
affine translation constants.

We record the exceptional Euler factors explicitly. For every $i<j$,
choose a nonzero positive-degree coefficient of $P_i-P_j$, and let $J_0$
be the product of these finitely many integers. Then
$J_0\ne0$ and $\log|J_0|\ll_s\log A$. If $p\nmid W$, the coefficient of
$x$ in $F_j$ is a unit modulo $p$, whence
\[
 \rho_p(\{j\})=\frac1p.
\]
For $p>v$ not dividing $J_0$, the polynomial $P_i-P_j$ is nonconstant
modulo $p$. Counting $x$ after $\boldsymbol h$ and using the elementary
polynomial zero bound gives, uniformly in the constant terms,
\[
 \rho_p(\{i,j\})\ll_s p^{-2}\qquad(i\ne j).
\]
The same bound holds for every subset with at least two elements.
At a prime $p\mid J_0$, the weaker bound $\rho_p(I)\le1/p$ remains
valid for every nonempty $I$. These statements also cover cyclic wrap
cells, since replacing $P_j$ by $P_j-r_jn$ changes only a constant term.

For clarity, after the Fourier--Mellin expansion of the cutoff functions,
write $z_j,z'_j$ for the resulting small complex exponents, and define
\[
 \mathcal L_p(\boldsymbol z,\boldsymbol z')=
 \sum_{\boldsymbol\varepsilon,\boldsymbol\varepsilon'\in\{0,1\}^s}
 (-1)^{\sum_j(\varepsilon_j+\varepsilon'_j)}
 p^{-\sum_j(\varepsilon_jz_j+\varepsilon'_jz'_j)}
 \rho_p(\{j:\varepsilon_j+\varepsilon'_j>0\}).
\]
At $p\mid W$ the divisor support forces $\mathcal L_p=1$. At a good
prime $p>v$, the singleton and intersection estimates give
\[
 \mathcal L_p=1+\frac1p\sum_{j=1}^s
 (p^{-z_j-z'_j}-p^{-z_j}-p^{-z'_j})+O_s(p^{-2}).
\]
Consequently, after extracting the usual zeta factors, the remaining
local factor is
\begin{equation}\label{an:eq:conductor-local-factors}
 \mathcal L_p\prod_{j=1}^s
 \frac{1-p^{-1-z_j-z'_j}}
 {(1-p^{-1-z_j})(1-p^{-1-z'_j})}
 =\begin{cases}
 1+O_s(p^{-2})&\t{if}\ p>v\ \t{and}\ p\nmid J_0,\\
 1+O_s(p^{-1})&\t{if}\ p>v\ \t{and}\ p\mid J_0.
 \end{cases}
\end{equation}
These estimates are uniform on the standard Fourier contours, where
$\Re z_j,\Re z'_j\ge0$; the finitely many small primes have already
been placed in $W$. Moreover,
\[
 \sum_{\substack{p>v\\p\mid J_0}}\frac1p
 \le\frac{\log|J_0|}{v\log v}
 \ll_s\frac{\log A}{v\log v}
 \ll_{s,D,\rho}\frac1{\log v}\ \ \t{and}\ \ \sum_{p>v}p^{-2}\ll v^{-1}.
\]
Thus the total correction from good and exceptional primes outside $W$
is $1+O_{s,D,\rho}(1/\log v)$. For $p\mid W$, the factor on the left
of \eqref{an:eq:conductor-local-factors}, evaluated at the origin, is
$(1-p^{-1})^{-s}$. Their product is exactly
$(W/\varphi(W))^s$, canceling the prefactor
$(\varphi(W)/W)^s$ in \eqref{an:eq:conductor-divisor-expansion}.
On any bounded Fourier-frequency range, the factor at $p\mid W$
differs from its value at the origin by
$O_s((\log p)/ (p\log R))$; hence the product has relative error
$O_s((\log W)/\log R)=o_v(1)$ for fixed $v$. The rapidly decaying
Fourier tails are handled as in \cite[Appendix A]{TZpoly}.
This cancellation depends only on the prime support of $W$; its
prime-power factors contribute no additional local term.

The remaining zeta-factor evaluation is the same Fourier--Mellin
calculation as in \cite[Appendix A]{TZpoly}; its normalization is
$\int_0^1|\chi'_0(t)|^2\,dt=1$. Combining it with
\eqref{an:eq:conductor-crt-error} and
\eqref{an:eq:conductor-local-factors} gives the required
polynomial-forms estimates with error
$O_{s,D,\rho}(1/\log v)+o_v(1)$, uniformly in the permitted
residues and polynomial systems. The $o_v(1)$ is taken with $v$
fixed; a sufficiently slow choice $v\to\infty$ then removes both
errors. Replacing $W_0$ by $E_DW_0$ changes neither the normalizing
constant nor the degree and coefficient bounds. We apply the argument
separately to $\kappa=\rho$ and $\kappa=\rho_*$. Thus the ordinary
majorants needed for the dense-model and generalized von Neumann
arguments of \cite[Corollary 3.9 and Theorem 4.1]{TZpoly} meet the
same polynomial-forms hypotheses. Lemma~\ref{an:lem:Artinuniform}
has already supplied the Artin correlations directly for this $W$.

\emph{A precise scale choice.}
Set $n=\lfloor N/W\rfloor$ and $M_0=\lfloor T/W\rfloor$. Choose the eventual slow threshold so that
\begin{equation}\label{an:eq:outerthreshold}
 T/\sqrt N\ge4\sqrt W/A\ \ \t{and}\ \ T\le\sqrt N
\end{equation}
for sufficiently large $N$. Its lower bound tends to zero. The integer scales then satisfy
\begin{equation}\label{an:eq:innerscales}
 A^{-1}\sqrt n\le M_0\le\sqrt n\ \ \t{and}\ \
 A^{-3}\sqrt n\le M=\lfloor A^{-1}M_0\rfloor\le\sqrt n.
\end{equation}
The exponent $-3$ absorbs floors. Choose $v$ slowly enough that every fixed power of $A$ is smaller than every fixed positive power of $n$.

The generalized von Neumann theorem \cite[Theorem 4.1]{TZpoly}, with $d=d_0=2$ and one step variable, bounds a correlation containing a centered factor by a positive fixed power of an averaged local box norm of that factor, plus $o(1)$. Its local norm is of the form
\begin{equation}\label{an:eq:localnormtarget}
 \mathbb E_{h\in[M]}
    \|g_{j,b}\|_{\Box^{D'}_{P_1(h)[-M,M],\ldots,P_{D'}(h)[-M,M]}},
\end{equation}
where $\Box^{D'}$ denotes the $D'$-fold box norm defined in
\eqref{an:eq:boxnormdef}. For an integer $a$, the notation
$a[-M,M]$ means the multiset
$\{am:m\in\Z\ \t{and}\ -M\le m\le M\}$ in $\Z/n\Z$, with multiplicities retained;
thus $P_i(h)[-M,M]$ is the increment multiset used in that box norm.
Also, $D'$ is bounded in terms of the original complexity, each $P_i$ is an integral polynomial of degree exactly one, and its coefficients are $O(A^C)$ for fixed $C$, as in (4.1) of that source. All multisets are reduced modulo $n$; endpoint rounding in $M$ is harmless.

With the fixed $B$ and $\rho_*$ chosen above, put
\[
 \mathcal A=A^B=W^{1/\rho_*}.
\]
Use the arbitrary-function part of \cite[Theorem 3.1 and Corollary 3.9]{TZpoly} with coefficient parameter $\mathcal A$ and majorant $\nu_{b,\rho_*}$. It applies because $f_{j,b}\le\Lambda'_{b,W;\rho_*}$: the original cutoff $n^\rho$ is larger. Moreover,
\[
 (\log n)^{1/\rho_*}\le M\le(\mathcal A n)^{1/2}
\]
for sufficiently slow $v$, and all coefficients in \eqref{an:eq:localnormtarget} are at most $\mathcal A$ for large $v$.

Corollary 3.9 supplies a bounded function $f'_{j,b}$ such that $f_{j,b}-f'_{j,b}$ is $o(1)$ both in the averaged local norms in \eqref{an:eq:localnormtarget} and in
\begin{equation}\label{an:eq:progressionscale}
 U^a_{q[J]}\ \t{with}\  q\le\mathcal A^a\ \t{and}\ 
 J=\lfloor\mathcal A^{-2a}M^2\rfloor.
\end{equation}
The same function $f'_{j,b}$ controls both kinds of norms. It may depend on the selected local polynomial system; a single model for every possible system is not needed. The integer $a$ was fixed sufficiently large for the largest local complexity before choosing $\rho$. By \eqref{an:eq:innerscales}, we have
\[
 nA^{-2Ba-7}\le J\le nA^{-2Ba}\ \ \t{and}\ \  q\le A^{Ba}.
\]
Lemma~\ref{an:lem:Artinuniform}, with $(C_1,C_2,C_3)=(Ba,2Ba,2Ba+7)$, makes $f_{j,b}-\mu_j(b)$ small in exactly these progression norms.

By the triangle inequality the same is true of $f'_{j,b}-\mu_j(b)$. Divide by a fixed bound and apply \cite[Theorem 2.5]{TZpoly} with $(d,d_0,r)=(2,2,1)$, coefficient bound $\mathcal A$ and step length $M$. Its condition $n>\mathcal A^{-1}M^2$ follows from $M\le\sqrt n$. The remaining lower bound on $M$ holds for large $n$ at every fixed $v$. For each desired local error $\eps>0$, the smallness threshold $\delta$ in that theorem is independent of $\mathcal A$; the required lower bound on $\mathcal A$ may depend on $\eps,\delta$. Thus first take $v$ large enough to meet that bound and make the $O(1/\log v)$ contribution in Lemma~\ref{an:lem:Artinuniform} sufficiently small, and then take $n$ large enough to control its $o_v(1)$ term and the size conditions. The dense-model errors are handled by the same slow diagonal choice. This yields averaged local norms of $f'_{j,b}-\mu_j(b)$ tending to zero. Adding the dense-model error proves that \eqref{an:eq:localnormtarget} is $o(1)$. No quantitative rate for the Artin theorem is required, and the special second part of \cite[Theorem 3.1]{TZpoly} has not been used outside its stated range.

\emph{Polynomial counting.}
Call residues $b,c\in[W]$ admissible when $(b+c_jc^2,W)=1$ for every $j$. Choose $b_j\in[W]$ congruent to $b+c_jc^2$ and write
\begin{equation}\label{an:eq:innerpolys}
 W x+b+c_j(Wm+c)^2=W(x+R_j(m))+b_j.
\end{equation}
Explicitly,
\[
 R_j(m)=c_jWm^2+2c_jcm+\frac{b+c_jc^2-b_j}{W}.
\]
Its coefficients are integers of size $O_{\boldsymbol c}(W)$, hence at most $A$ for large $v$. The leading coefficient of $R_j-R_l$ is $(c_j-c_l)W\ne0$. These verify the coefficient and degree hypotheses of \cite[Theorem 4.1]{TZpoly}. Expand the product of $f_{j,b_j}$ into constant models and centered terms. Its constant term is $\prod_j\mu_j(b_j)$; all other terms are $o(1)$ by that theorem and \eqref{an:eq:localnormtarget}. The factors are bounded by a fixed multiple of $\Lambda'_{b_j,W;\rho}+1$, as required. A term with only one centered factor is its mean, which tends to zero by the linear Artin theorem. Thus, uniformly over admissible residues,
\[
 \mathbb E_{x\in\Z/n\Z}\mathbb E_{m\in[M_0]}
       \prod_j f_{j,b_j}(x+R_j(m))
   =\prod_j\mu_j(b_j)+o(1).
\]

\emph{Removing the $W$-trick and controlling the endpoint strips.}
Use the residue decomposition of \cite[Section 5, proof of Theorem 1.4]{TZpoly}. The shifts in the outer integer variable have size $O_{\boldsymbol c}(T^2)=o(N)$. We spell out why the endpoint error is $o(NT)$ without imposing a quantitative rate on $T^2/N$. Fix a real $\delta>0$ for the moment. For all sufficiently large $N$ with $T^2/N<\delta$, a cyclic wrap or a change caused by extending or truncating the outer interval can occur only when the base variable belongs to one of a fixed number of strips of relative width $O_{\boldsymbol c}(\delta)$ at the ends of its interval. Each factor is nonnegative and bounded by a fixed multiple (depending only on the already fixed sieve parameters) of its ordinary-prime majorant $\nu_{b,\rho}$. The polynomial-forms estimate proved above, applied to these \emph{fixed-width convex strips} and the full step interval, bounds the normalized total weight by $O_{\boldsymbol c}(\delta)+o_{v,\delta}(1)$. This use is legitimate because for fixed $\delta$ the strips have outer side length $\asymp_\delta n$, while the other sides retain the lower power-of-$n$ bounds required by the polynomial-forms estimate. The cutoff $x<n^\rho$ is treated in the same way since $n^{\rho-1}\to0$. First let $n\to\infty$ for fixed $v$ and $\delta$, then perform the previously specified slow $v$-diagonal, and finally let $\delta\downarrow0$. Positivity makes this fixed-width domination sufficient for an arbitrary $T=o(\sqrt N)$, even when $T^2/N$ decays more slowly than every negative power of $\log N$.

On inadmissible residues a prime value must divide $W$, hence equals one of its finitely many prime divisors; their weighted contribution is negligible by direct counting at fixed $v$. Floors in residue lengths contribute $o(NT)$ under the same slow choice of $v$.

Summing uniform errors over residues causes no loss, because the ordinary admissible mass satisfies
\begin{equation}\label{an:eq:admissiblemass}
 \left(\frac W{\varphi(W)}\right)^k
  \frac{\#\{(b,c)\bmod W:\ (b,c)\text{ admissible}\}}{W^2}
   =\prod_{p\mid W}\beta_p=O_{\boldsymbol c}(1).
\end{equation}
This also holds for prime powers in $W$, since the unit conditions depend only on the underlying prime. The bound follows from convergence below.

Define $\mu_{j,W}(b)=1_{(b,W)=1}\mu_j(b\bmod D)$. The leading constant before the limit is
\[
 \left(\frac W{\varphi(W)}\right)^k
   \mathbb E_{u,w\bmod W}\prod_j\mu_{j,W}(u+c_jw^2).
\]
The CRT factors it into the finite factor in \eqref{an:eq:Artinconstant} and $\prod_{p\le v,\ p\nmid D}\beta_p$. For primes outside the fixed coefficient differences, count $w=0$ and $w\ne0$ separately to obtain
\[
 \beta_p=\left(\frac p{p-1}\right)^k
             \frac{(p-1)(p-k+1)}{p^2}=1+O_k(p^{-2}).
\]
This proves convergence and the constant formula.

All complexity constants, including $a,B,\rho$, are fixed before $v$ grows. The linear counting remainders tend to zero as $n\to\infty$ at fixed $v$, uniformly in the permitted convex regions and scales. The local-factor error $O(1/\log v)$ does not vanish in that limit; it vanishes only when $v$ also grows. The ordinary-majorant, dense-model and transference errors have the same iterated-limit interpretation. For each target error, first choose a sufficiently large cutoff $v$, then a sufficiently large counting threshold. A stepwise choice $v=v(N)\to\infty$ through these thresholds makes all errors tend to zero and enforces all size conditions, uniformly in the allowed $T$. A lower threshold at least $4\sqrt W/A$, enlarged for floors and finitely many initial values, still tends to zero. It defines one $\omega(N)$ valid for every $T$ in the stated range. Apply the asymptotic at $N$ and at $N/2$ with the same $T$; taking the larger of their two slow lower thresholds permits subtraction and gives the last assertion.
\end{proof}

\section{Ray fibres and the finite-cutoff lower bound}\label{an:sec:raylower}

Let $\mathcal O=\Z[i]$. Fix a rational integer modulus $H_0$ divisible by 16 and 7 and by $M$. Enlarge its odd part to contain all primes below a fixed $L_1\ge11$. It is enough to use exponent one at odd primes. Choose a unit residue $r\bmod H_0$ satisfying \eqref{an:eq:ray}, and $r\equiv i$ at the other odd prime factors of $H_0$, apart from 7. These prescriptions are compatible. At 7 the two order norms are 6 and 2; at every other odd factor they are both 2. At 2 the order valuations are 2 and 6.

Let $K_0$ be the Gaussian ray class field of modulus $H_0\mathcal O$, viewed as a Galois extension of $\Q$. Write $G_0=\Gal(K_0/\Q)$, and let $h_r\in\Gal(K_0/\Q(i))$ be the ray class of $r$ under Artin reciprocity. Complex conjugation carries it to $h_{\bar r}$. These two elements are distinct: equality of their ray classes would mean $r=u\bar r\bmod H_0$ for a Gaussian unit $u$; reduction modulo 16 forces $u=1$, whereas their residues modulo 7 differ. Thus
\[
 \mathcal C_0=\{h_r,h_{\bar r}\}
\]
is a nonempty conjugacy class in $G_0$. A rational norm prime in this class has a Gaussian prime generator congruent to $r$, after conjugating the prime ideal if necessary. Put $G_0'=[G_0,G_0]$, where $[G,G]$ denotes the commutator subgroup of a group $G$. Let $D_0$ be the conductor of $K_0^{\rm ab}$. The two elements of $\mathcal C_0$ have the same restriction to this abelian subfield. Choose a unit class $b_0\bmod D_0$ inducing that restriction. The conditional density defined in \eqref{an:eq:conditionalmu} is
\[
 \mu_0=\frac{|\mathcal C_0|}{|G_0'|}>0.
\]
The existence of $b_0$ follows from the surjectivity of the restriction map from $\Gal(\Q(\zeta_{D_0})/\Q)$ to $\Gal(K_0^{\rm ab}/\Q)$.

For $L>L_1$, put $T_L=\prod_{L_1<\ell\le L,\ \ell\nmid H_0}\ell$ and let $K_L$ be the ray field of modulus $H_0T_L\mathcal O$. A class above $h_r$ is normalized to have base residue $r$; a class above $h_{\bar r}$ is normalized to have base residue $\bar r$. In addition to the base condition, require that neither $\Nm(1+\pi)$ nor $\Nm(1-\pi)$ is divisible by any $\ell\mid T_L$, using these normalized residues. The following lemma makes the normalization and the resulting density precise.

\begin{lemma}\label{an:lem:rayfibres}
The maximal abelian subfield of $K_L/\Q$ is
\[
 K_L^{\rm ab}=K_0^{\rm ab}\Q(\zeta_{T_L}).
\]
Fix a base element in $\mathcal C_0$ and nonzero norm residues $s_\ell\bmod\ell$ for all $\ell\mid T_L$. Its normalized lifts form a Cartesian product of fibres of sizes $\ell-\chi(\ell)$. At each $\ell$, at most four elements of its fibre violate the added order-avoidance condition. In particular, the fibres are independent for uniform counting in this finite group.
\end{lemma}
\begin{proof}
Put
\[
 U=\mathcal O^*=\{1,-1,i,-i\},
 \ A_0=(\mathcal O/H_0\mathcal O)^*,\ 
 B_\ell=(\mathcal O/\ell\mathcal O)^*,\ 
  B=\prod_{\ell\mid T_L}B_\ell.
\]
The Gaussian class number is one. The ray-class exact sequence and the Chinese remainder theorem give
\[
 H_L:=\Gal(K_L/\Q(i))=(A_0\times B)/U\ \ \t{and}\ \ 
 \Gal(K_0/\Q(i))=A_0/U,
\]
with diagonal action of $U$; see \cite[Chapter V, Theorem 1.7 and Section 3]{Milne}. Since $16\mid H_0$, the map $U\to A_0$ is injective. If a lift of $h_r$ has representative $(ur,\boldsymbol b)$, multiplication by the unique $u^{-1}$ normalizes it to $(r,u^{-1}\boldsymbol b)$. The same holds with $\bar r$ in place of $r$. All units have norm one, so this normalization leaves the new norm residues unchanged. In particular, the kernel of $H_L\to\Gal(K_0/\Q(i))$ is the full group $B$, embedded by $\boldsymbol b\mapsto[(1,\boldsymbol b)]$; there is no further quotient by $U$.

Define
\[
 \mathcal N_\ell=\ker(\Nm:B_\ell\to\F_\ell^*)\ \ \t{and}\ \ 
 \mathcal N=\prod_{\ell\mid T_L}\mathcal N_\ell.
\]
The norm maps are onto. If $\ell\equiv1\pmod4$, then $B_\ell\simeq\F_\ell^*\times\F_\ell^*$ and conjugation interchanges the two coordinates. If $\ell\equiv3\pmod4$, then $B_\ell\simeq\F_{\ell^2}^*$ and conjugation is the $\ell$th-power map. Consequently
$|\mathcal N_\ell|=\ell-\chi(\ell),$
and the image of $b\mapsto b/\bar b$ is exactly $\mathcal N_\ell$: in the split case it is $(x,y)\mapsto(x/y,y/x)$, and in the inert case it is $z\mapsto z^{1-\ell}$. Set $G_L=\Gal(K_L/\Q)$ and $G_L'=[G_L,G_L]$. Commutators with complex conjugation therefore give $\mathcal N\subseteq G_L'$.

By the norm compatibility of Artin reciprocity, $\Q(\zeta_{T_L})$ is contained in $K_L$, and the restriction of an element of $B$ to it is given by its norm residues; see \cite[Chapter V, Section 3]{Milne}. The kernel of restriction to $K_0\Q(\zeta_{T_L})$ is consequently exactly $\mathcal N$. Since the norms on the factors of $B$ are onto, this compositum has degree $|G_0|\varphi(T_L)$ over $\Q$. Restriction thus identifies
\[
 G_L/\mathcal N\simeq
 G_0\times\prod_{\ell\mid T_L}\F_\ell^*.
\]
The new factors are abelian. Taking commutator subgroups, using $\mathcal N\subseteq G_L'$, yields
\[
 G_L'/\mathcal N\simeq G_0'\ \ \t{and}\ \ 
 |G_L'|=|\mathcal N|\,|G_0'|.
\]
Taking fixed fields proves the asserted formula for $K_L^{\rm ab}$.

After the base normalization, prescribing all norms leaves a product of norm-one torsors. Each such torsor is a coset of $\mathcal N_\ell$ and has $\ell-\chi(\ell)$ elements. This proves the finite-group independence assertion. Write a residue as $\pi=a+bi$. On the fibre $a^2+b^2=s_\ell$, one has
\[
 \Nm(1\pm\pi)=1+s_\ell\pm2a.
\]
For each sign, vanishing fixes $a$ and permits at most two values of $b$, so the two conditions exclude at most four residues. Equivalently, in the split case the fibre consists of $(x,s_\ell/x)$, and the forbidden values of $x$ lie in $\{1,-1,s_\ell,-s_\ell\}$. In the inert case a zero norm forces $\pi=1$ or $-1$, possible only if $s_\ell=1$. The representatives $r,\bar r$ are interchanged by conjugation, and $\Nm(1\pm\bar\pi)=\Nm(1\pm\pi)$. Thus the normalized condition defines a conjugacy-invariant subset $\mathcal C_L\subseteq G_L$, as required for rational Artin symbols.
\end{proof}

The field $K_0$ is ramified only at rational primes dividing $H_0$, so $(D_0,T_L)=1$. A valid conductor multiple for $K_L^{\rm ab}$ is therefore $D_L=D_0T_L$. Let $\mu_L(b)$ be the conditional density \eqref{an:eq:conditionalmu} for $(K_L,\mathcal C_L)$. If $b\equiv b_0\pmod{D_0}$ and $(b,D_L)=1$, each of the $|\mathcal C_0|$ eligible base elements has a fibre of $|\mathcal N|$ lifts with the prescribed abelian restriction. By Lemma~\ref{an:lem:rayfibres}, the proportion of good lifts is at least
\[
 \prod_{\ell\mid T_L}
 \left(1-\frac4{\ell-\chi(\ell)}\right)
 \ge \prod_{\ell\mid T_L}\left(1-\frac4{\ell-1}\right).
\]
Dividing the number of good lifts by $[K_L:K_L^{\rm ab}]=|G_L'|=|\mathcal N|\,|G_0'|$ gives
\begin{equation}\label{an:eq:muL}
 \mu_L(b)\ge
 \mu_0\prod_{\ell\mid T_L}\left(1-\frac4{\ell-1}\right).
\end{equation}
Thus no extra factor $[K_L:\Q]^{-1}$ belongs in this conditional density. Every factor is positive since $\ell>L_1\ge11$. Moreover,
\[
 \log\left(1-\frac4{\ell-1}\right)
 =-\frac4\ell+O(\ell^{-2}).
\]
Mertens' estimate \cite{Dav}, with the finitely many primes dividing $H_0$ absorbed in the constant, gives
\begin{equation}\label{an:eq:logcost}
 \prod_{\ell\mid T_L}\left(1-\frac4{\ell-1}\right)^k
 \gg_{k,H_0}(\log L)^{-4k}.
\end{equation}
It is essential here to retain all good norm-fibre classes, rather than one specified residue at each new prime.

\begin{proposition}\label{an:prop:lower}
For fixed $L$ and fixed distinct coefficients $c_j$, the weighted count of patterns satisfying the finite-cutoff ray conditions, on $N/2<u\le N$ and  $1\le w\le T$, is
\[
 (\mathfrak S_L/2+o_L(1))NT\ \ \t{and}\ \ \mathfrak S_L\ge c(\log L)^{-4k}>0,
\]
where $c$ is independent of $L$. The ratio $T/\sqrt N$ must tend to zero sufficiently slowly for this fixed $L$.
\end{proposition}
\begin{proof}
Apply Proposition~\ref{an:prop:Artinpoly} to the conjugacy-invariant subset of $\Gal(K_L/\Q)$ just constructed. In the average modulo $D_0$, keep the single residue pair $u=b_0,w=0$. Then all positions have the compatible norm residue $b_0$. At each new $\ell$, keep all residue pairs for which every $u+c_jw^2$ is a unit, and use \eqref{an:eq:muL} uniformly in those norm residues. Their ordinary local contribution is $\beta_\ell$. The factors $\beta_\ell$ from the new primes and those outside $D_L$ together form the fixed product $\prod_{p\nmid D_0}\beta_p$. It is strictly positive: the residue choice $u=1,w=0$ is admissible at every prime, and the product converges absolutely. More explicitly, the lower bound is
\[
 \mathfrak S_L\ge
 \left(\frac{D_0}{\varphi(D_0)}\right)^k
 \frac{\mu_0^k}{D_0^2}
 \prod_{p\nmid D_0}\beta_p
 \prod_{\ell\mid T_L}\left(1-\frac4{\ell-1}\right)^k.
\]
Everything preceding the last product is a fixed positive constant. Equation \eqref{an:eq:logcost} proves the uniform-in-$L$ lower bound for the leading constant. The errors may depend on $L$, which is fixed before $N$ tends to infinity.
\end{proof}

\subsection{Closing the limits and proving simultaneous avoidance}\label{an:sec:closing}
Fix all coefficients, the dilation and base avoidance conditions first; absorb the dilation into the coefficients. Let $L_0$ in Proposition~\ref{an:prop:fulltail} also exceed every exceptional prime in its local calculations. Take a fixed $L\ge\max(L_0,L_1)$ sufficiently large that
\begin{equation}\label{an:eq:compare}
 \frac C{L\log L}<\frac c{32(\log L)^{4k}},
\end{equation}
where $C$ is the tail constant and $c$ the lower-bound constant. Such an $L$ exists. Fix this $L$, and hence the field $K_L$ and its Artin subset, before making the slow cutoff choice in Section~\ref{an:sec:transfer}. No counting theorem uniform in the degree of $K_L$ is used.

It remains to choose the short parameter with a valid order of limits. Let $\tau_j\downarrow0$ be positive constants, all at most $\tau_0$. Let $\omega_L(N)$ be a lower threshold valid for the interval version of Proposition~\ref{an:prop:Artinpoly} with the fixed field $K_L$. For each fixed $\tau_j$, the error in \eqref{an:eq:weightedtail}, divided by $\tau_jB^3$, tends to zero. Choose increasing thresholds $N_j$ so large that, for every $N\ge N_j$, this ratio is at most $1/j$, $\omega_L(N)\le\tau_j/2$, and $\tau_j\sqrt N\ge2$. Set $\tau(N)=\tau_j$ for $N_j\le N<N_{j+1}$, and let $T=\lfloor\tau(N)\sqrt N\rfloor$. Then $$T\ge\tau(N)\sqrt N/2,\ \ \tau(N)\longrightarrow0,\ \ T\ge\omega_L(N)\sqrt N,$$ 
and the tail error is $o(NT)$. Thus the same integer-valued $T$ satisfies both the lower-bound hypothesis and the tail estimate, with $L$ fixed first.

For this common choice, Proposition~\ref{an:prop:lower} supplies at least
$cNT/(4(\log L)^{4k})$ finite-cutoff patterns in weight, for large $N$. Since $B=\sqrt N$ and $T\ge\tau(N)\sqrt N/2$, the main term of the tail in Proposition~\ref{an:prop:fulltail} is at most
\[
 \frac{C\tau(N)B^3}{L\log L}
 \le\frac{2CNT}{L\log L}
 <\frac{cNT}{16(\log L)^{4k}}.
\]
Its $o(NT)$ remainder is smaller than $cNT/(16(\log L)^{4k})$ for all sufficiently large $N$, because $L$ was fixed first. The total removed weight is therefore less than $cNT/(8(\log L)^{4k})$, which is half the stated lower bound. Hence a positive weighted remainder has no odd self- or cross-order obstruction at any prime. All surviving positions are prime by construction.

The base congruences give $v_2(d^+)=2$, $v_2(d^-)=6$. Thus the self-gcds are four and \eqref{an:eq:allcross} holds, with all other cross-gcds powers of two. The congruence $\pi\equiv i\pmod M$ gives both orders congruent to 2 modulo every divisor of $M$. Finally $u\ge N/2$, $w\le T=o(\sqrt N)$, so $u/w^2\to\infty$, while every norm tends to infinity. Restoring the prescribed dilation proves Theorem~\ref{an:thm:patterns} in full.

\section{Prescribed residue pairs and contact characters}\label{sec:contactcharacters}

\begin{lemma}\label{lem:nine-signs}
Let $p>10^{16}$ and let $\lambda_1,\ldots,\lambda_k$ be distinct residues with $k\le21$. Every pattern of $k$ nonzero quadratic-character signs of $u-\lambda_j$ is attained, even after excluding the roots and two more residues.
\end{lemma}
\begin{proof}
Expand $2^{-k}\prod_{j=1}^k(1+\varepsilon_j\chi(u-\lambda_j))$. The sum of the nonconstant character-sum bounds is at most $((k-2)2^{k-1}+1)\sqrt p$ (zero for $k=1$). After dividing by $2^k$ and removing at most $k+2$ residues, the lower bound is positive for $p>10^{16}$. This uses only the squarefree-polynomial Weil bound \cite{Weil}. In particular, if $A\equiv25\pmod p$ outside the fixed coefficient exceptions, it applies to the twenty-one distinct positions $\{0\}\cup\mathcal T_*^\times\cup A^{-1}\mathcal T_*^\times$, with the normalizers excluded. Distinctness at $A=25$ follows from the ratio bound and the fixed exclusions.
\end{proof}

\begin{proposition}\label{prop:position-rays}
Fix distinct integers $c_1,\ldots,c_k$, a Gaussian modulus $H_0$ divisible by $16\cdot7$, and Gaussian unit residues $\varrho_1,\ldots,\varrho_k\pmod{H_0}$ satisfying
\[
 \varrho_j\equiv9\pmod{16},\qquad \varrho_j\equiv1+3i\pmod7,
\]
and, at every other odd prime dividing $H_0$, assume that both $\Nm(1+\varrho_j)$ and $\Nm(1-\varrho_j)$ are units.  Let $K_0$ be the Gaussian ray field of modulus $H_0\Z[i]$, and let $F/\Q$ be a fixed finite Galois extension containing $K_0$ and ramified only at primes dividing $H_0$.  For each $j$, prescribe a conjugacy-invariant subset $\mathcal C_j\subseteq\operatorname{Gal}(F/\Q)$ restricting to the conjugacy class of the ray of $\varrho_j$.

Let $D$ be a conductor multiple for $F^{\rm ab}$, enlarged by any further fixed congruence modulus, with every prime dividing $D$ included in $H_0$. Assume that $H_0$ is divisible by all primes up to a fixed $L_1\ge11$, with compatible ray conditions; this modulus is fixed before $K_0$ and $F$ are chosen. Suppose that there is a residue pair $(u_0,w_0)\pmod D$ such that each $b_j=u_0+c_jw_0^2$ is a unit modulo $D$ and
\begin{equation}\label{eq:muj}
 \mu_j(b_j):=
 \frac{\#\{g\in\mathcal C_j:g|_{F^{\rm ab}}=\sigma_{b_j}|_{F^{\rm ab}}\}}
 {[F:F^{\rm ab}]}>0.
\end{equation}
Then, by the polynomial Artin count and order-obstruction tail estimate in Propositions~\ref{an:prop:Artinpoly}, \ref{an:prop:fulltail} and the large-prime estimate of Appendix~\ref{an:app:large}, there are arbitrarily large prime patterns
\begin{equation}\label{eq:primepattern}
 p_j=u+c_jw^2,
 \qquad (u,w)\equiv(u_0,w_0)\pmod D,
 \qquad u/w^2\longrightarrow+\infty,
\end{equation}
with the prescribed Artin conditions in $F$.  Their Gaussian generators have the prescribed position-dependent rays and the two reduction orders have dyadic valuations $2$ and $6$, respectively. Their odd parts are pairwise coprime, including the two orders at the same position, and both orders are units at every prescribed odd ray prime.
\end{proposition}
\begin{proof}
In the $W$-trick proof of the Artin count, the estimate is established separately for each admissible residue pair modulo $W$, where $D\mid W$. Retaining only pairs congruent to $(u_0,w_0)$ modulo $D$ changes the singular series to
\[
 \frac1{D^2}\left(\frac D{\varphi(D)}\right)^k
 \prod_j\mu_j(b_j)\prod_{p\nmid D}\beta_p,
\]
where $\beta_p=(p/(p-1))^k\,\Pr_{u,w\bmod p}(p\nmid\prod_j(u+c_jw^2))$. Every local factor is positive by the choice $(u,w)=(1,0)$. For primes not dividing a coefficient difference, the probability is $1-k/p+(k-1)/p^2$, so $\beta_p=1+O_k(p^{-2})$ and the product is strictly positive. No change of the transference estimates is needed; all sums retained are subsets of the original admissible residue pairs, and their error bounds are uniform at that stage.

For a fixed cutoff $L$, put
\[
 T_L=\prod_{L_1<\ell\le L,\,\ell\nmid H_0}\ell,\qquad
 K_L=K_{H_0T_L},\qquad F_L=FK_L.
\]
The Gaussian ray-class exact sequence and the Chinese remainder theorem identify $\Gal(K_L/K_0)$ with $\prod_{\ell\mid T_L}B_\ell$, where $B_\ell=(\Z[i]/\ell\Z[i])^\times$: the units of $\Z[i]$ are already distinguished modulo $16$, so there is no extra unit quotient in this kernel. At a new prime $\ell$, the inertia groups generate its entire $B_\ell$ factor. As $F$ is unramified at those primes, $F\cap K_L=K_0$, so the same product is $\Gal(F_L/F)$. Conjugation acts on each factor by Gaussian conjugation; the image of $b\mapsto b/\overline b$ is $N_\ell=\ker(\Nm:B_\ell\to\F_\ell^*)$, both in the split and inert cases. These subgroups belong to the commutator subgroup. Their common fixed field is $F\Q(\zeta_{T_L})$. Its two factors intersect trivially since they are ramified at disjoint rational finite primes. Taking commutators thus gives
\begin{equation}\label{eq:FLab}
 F_L^{\rm ab}=F^{\rm ab}\Q(\zeta_{T_L})\ \ \t{and}\ \ 
 [F_L:F_L^{\rm ab}]=[F:F^{\rm ab}]\prod_{\ell\mid T_L}|N_\ell|.
\end{equation}
On a prescribed nonzero norm fibre there are $\ell-\chi_{-1}(\ell)$ elements, where $\chi_{-1}(\ell)=(-1/\ell)\in\{1,-1\}$ is the ordinary quadratic character, and at most four violate $\Nm(1\pm\pi)\ne0$: write $\pi=a+bi$ and use $\Nm(1\pm\pi)=1+\Nm(\pi)\pm2a$. For each eligible base element the fibres are independent. Hence the conditional density of the lift at position $j$ is at least
\[
 \mu_j(b_j)\prod_{\ell\mid T_L}\left(1-\frac4{\ell-\chi_{-1}(\ell)}\right).
\]
Here one retains \emph{all} admissible $(u,w)$ classes at each new prime, not just one. The local factors above therefore show that the leading constant is at least $c(\log L)^{-4k}$ for a constant $c>0$ independent of $L$. The ray at $16$ has $v_2(\Nm(1+\pi))=2$ and $v_2(\Nm(1-\pi))=6$, and the ray at $7$ makes both norms units. For the large odd obstruction primes the retained tail estimate, after dropping all extra ray and congruence restrictions, is
\[
 O\left(\frac{\tau B^3}{L\log L}\right)+O_\tau(B^{3-1/400}),
 \qquad N=B^2,\quad T=\tau B.
\]
First choose $L$ so large that the first error is smaller than half the main-term lower bound on $N/2<u\le N$. This is possible since $(\log L)^{4k-1}/L\to0$. Keep $L$, the fields and all congruences fixed. Then let $B\to\infty$ and let $\tau\to0$ sufficiently slowly, as in the retained Artin proposition, so that the second error is $o(\tau B^3)$ and the Artin asymptotic applies. Such a diagonal choice exists by taking successive fixed values of $\tau$ with sufficiently large thresholds. The remaining positive weighted count gives arbitrarily large prime values, with $u/w^2\to\infty$, and excludes all odd self- and cross-order obstructions. This proves the assertion.
\end{proof}

At an odd prime $\ell$, every nonzero norm fiber in $(\Z[i]/\ell\Z[i])^\times$ has $\ell-\left(\frac{-1}{\ell}\right)$ elements, and at most four make one of the two order norms vanish.  Thus every nonzero prescribed norm has an admissible Gaussian lift for $\ell\ge7$.  At $3$ both norm classes are also possible; at $5$ the two-sided order condition requires a square norm, which is why the small-prime preparation is handled separately.

For a pure Gaussian ray modulus $H_0=16M$ with $M$ odd, the maximal abelian subfield is
\begin{equation}\label{eq:rayab}
 K_0^{\rm ab}=\Q(\zeta_{32M}).
\end{equation}
Here is the group calculation for this assertion. For odd prime-power factors of $M$, the norm map on units is onto and its kernel is the image of $z\mapsto z/\overline z$ (the split case is immediate, and in the unramified quadratic case this follows from surjectivity of the norm on units and counting the conjugation-fixed units). At modulus $16$ the unit group has order $128$. The norm modulo $32$ is well-defined, with image the eight units congruent to $1$ modulo $4$: odd squares give $1,9,17,25$, and adding $2^2$ gives the other four. Its kernel has order $16$ and is generated by the Gaussian units together with the conjugate ratio of $1+2i$, namely $9+4i$ modulo $16$, an element of order four whose cyclic subgroup meets the units trivially. Thus the kernel of the full norm map is generated by the units and the conjugation commutators. The abelian quotient over $\Q$ has degree $16\varphi(M)$. Norm compatibility of Artin reciprocity includes $\Q(\zeta_{32M})$ in the ray field (the norm of $1+16Mz$ is $1$ modulo $32M$); equality follows by the degree computation. Hence in the pure-ray case the abelian compatibility amounts simply to
\begin{equation}\label{eq:normcompat}
 u_0+c_jw_0^2\equiv\Nm(\varrho_j)\pmod{32M}.
\end{equation}

\section{Realizing the cross-scale Legendre block by private primes}\label{sec:cross}

Fix the ten nonzero contacts $\theta_i$ and a rational square $A>16$.  Put
\begin{equation}\label{eq:lambdas}
 \lambda_0=0,
 \quad \lambda_i=\theta_i\ (1\le i\le10),
 \quad \lambda_{10+j}=\theta_j/A\ (1\le j\le10),
\end{equation}
and let $\mathcal C$ be the ratio set in \eqref{eq:Cratios}.  Let $S$ contain the fixed small primes, coefficient primes and nonzero differences between distinct ratios, and let $\mathcal B$ be the finite scalar-support set.  Suppose
\begin{equation}\label{eq:crossindependence}
 [A-c]\quad(c\in\mathcal C)
\end{equation}
are independent modulo the span of $[-1]$ and the prime classes in $S\cup\mathcal B$.

\begin{lemma}\label{lem:private}
For every $c\in\mathcal C$ there is a prime $\ell_c\notin S\cup\mathcal B$ with $v_{\ell_c}(A-c)$ positive and odd, while $v_{\ell_c}(A-c')=0$ for every $c'\ne c$.  The primes $\ell_c$ are distinct, and $A$ is a unit at each of them.
\end{lemma}
\begin{proof}
The nontriviality of the individual class $[A-c]$ modulo the constant span gives a prime of odd valuation outside $S\cup\mathcal B$.  Outside $S$, if $v_\ell(A)<0$ then $v_\ell(A-c)=v_\ell(A)$ is even, while if $v_\ell(A)>0$ then $A-c$ is a unit.  Hence odd valuation forces $A$ to be a unit and $A\equiv c\pmod\ell$.  Since the differences $c-c'$ are units outside $S$, no other $A-c'$ vanishes there.
\end{proof}

Choose an auxiliary positive squarefree denominator $D$ whose nonfixed prime factors avoid the numerators and denominators of the $A-c$, and choose $\Lambda$ clearing coefficients so that the twenty-one prime positions are
\begin{equation}\label{eq:pj}
 p_j=u-D\Lambda^2\lambda_jw^2.
\end{equation}

\begin{proposition}\label{prop:crossblock}
Suppose the remaining finite base conditions for \eqref{eq:pj} are compatible and supported in $S\cup\mathcal B\cup\Supp(D)$.  Then the position-dependent Gaussian-ray prime theorem can realize any prescribed $10\times10$ block of Legendre symbols
\[
 \left(\frac{p_{10+j}}{p_i}\right)\ \ (1\le i,j\le10),
\]
while retaining those base conditions and the simultaneous reduction-order avoidance.
\end{proposition}
\begin{proof}
All selected primes are $1$ modulo $4$.  For an old/new pair,
\[
 p_{10+j}-p_i=D\Lambda^2(\theta_i-\theta_j/A)w^2\ \ \t{and}\ \ 
 [\theta_i-\theta_j/A]=[\theta_i][A-\theta_j/\theta_i].
\]
At the private prime $\ell_c$ corresponding to $c=\theta_j/\theta_i$, the desired cross-edge bit contains the freely prescribed bit $(p_i/\ell_c)$ with coefficient one.  The coincident old/new positions at a fixed $\ell_c$ form a matching, so no two requested cross edges compete for the same endpoint constraint.  With $w_0=1$, the distinct cluster residues are affine linear functions of $u_0$; they form at most twenty-one distinct positions. The proof of Lemma~\ref{lem:nine-signs} applies to any such positions, since the private primes exceed the stated bound. Prescribe the character of each matching cluster to compensate for the already known contributions of $D\theta_i$ and all other factors of $A-c$, including its sign. The two endpoint values have the same residue, as reciprocity requires. The private prime divides no other $A-c'$, so these prescriptions do not alter the other cross edges.  At every other factor prime of $A-c$ take $w_0=0,u_0=1$, except where a previously prescribed condition contributes a known offset.  The Chinese remainder theorem combines the finitely many choices, and Proposition~\ref{prop:position-rays} realizes the resulting global prime pattern.
\end{proof}

\section{A uniform coefficient-five Selmer construction}\label{sec:selmer}

Throughout this section $R_1$ denotes the modified source $R_1'$.  For positive $r,R_0,R_1$ consider
\begin{equation}\label{eq:coefficient-five-cover}
 -5R_1z=u_1^2,\qquad -R_0(z-1)=u_0^2,\qquad r(z+3)=u_s^2.
\end{equation}
It maps to
\begin{equation}\label{eq:Ed}
 E_d:y^2=x(x-d)(x+3d),
 \qquad d=5rR_0R_1,
 \qquad x=dz,
 \qquad y=du_1u_0u_s.
\end{equation}
In Kummer coordinates $([x+3d],[x-d])$ the required class and one torsion translate are
\begin{equation}\label{eq:Gamma-gamma}
 \Gamma=([5R_0R_1],[-5rR_1])\ \ \t{and}\ \ 
 \gamma=([15R_0R_1],[R_0]).
\end{equation}
Their product is $(3,-d)$, the class of $(-3d,0)$.

\subsection{The local equations}\label{subsec:local-equations}

For an odd prime $p>5$ put
\begin{equation}\label{eq:profile}
 f(p)=(s_p,\lambda_p,t_p,\phi_p)
 =\left(\left[\frac{-1}{p}\right],\left[\frac3p\right],\left[\frac2p\right],\left[\frac5p\right]\right)\in\F_2^4,
\end{equation}
where a bracket is $0$ for a quadratic residue and $1$ for a nonresidue; the notation $\chi_a$ for a bit below has the same convention.  Let $d>0$ be odd and squarefree with $d\equiv17\pmod{24}$, and list its prime divisors as $q_1,\ldots,q_n$, including $5$. At $5$ use $(s,\lambda,t,\phi)=(0,1,1,0)$, with the unused last entry set to zero. Write $s,\lambda,t$ also for the corresponding column vectors, $D_a$ for the diagonal matrix with entries of a vector $a$, and $\mathbf1$ for the all-ones vector. All matrix operations below are over $\F_2$. Define the directed Legendre Laplacian $K$ by
\begin{equation}\label{eq:Kmatrix}
 K_{ij}=\left[\frac{q_j}{q_i}\right]\ (i\ne j)\ \ \t{and}\ \ 
 K_{ii}=\sum_{j\ne i}K_{ij}.
\end{equation}
Then
\begin{equation}\label{eq:reciprocity-matrix}
\begin{gathered}
 K\mathbf1=0,\qquad K+K^T=ss^T+D_s,\qquad K^T\mathbf1=s,\\
 \mathbf1^Ts=0,\qquad \mathbf1^T\lambda=1,\qquad \mathbf1^Tt=0.
\end{gathered}
\end{equation}

\begin{lemma}\label{lem:local-equations}
The subgroup of $\Sel_2(E_d/\Q)$ with positive second Kummer coordinate is the solution space, in $(u,v,\delta,\varepsilon)\in\F_2^{2n+2}$, of
\begin{gather}
 Ku+D_\lambda v=\delta\lambda+\varepsilon t,
 \ (K+D_s)v=\varepsilon t,\label{eq:selmer1}\\
 \varepsilon+(s+\lambda)^Tv=0,\label{eq:selmer2}\\
 \varepsilon=\delta+s^T(u+v),
 \ (s+t)^Tu+t^Tv=0.\label{eq:selmer3}
\end{gather}
The corresponding Kummer pair is
\begin{equation}\label{eq:kummerpair}
 \left(2^\varepsilon3^\delta\prod_iq_i^{u_i},\;2^\varepsilon\prod_iq_i^{v_i}\right).
\end{equation}
This subgroup has index two in the full Selmer group.
\end{lemma}
\begin{proof}
The real Kummer image has positive first coordinate; the torsion class $(3,-d)$ changes the sign of the second. Thus the subgroup with both coordinates positive has index two. At a good odd prime the local Kummer image is the unramified subgroup; outside $2,3$ and the primes dividing $d$ both valuation parities are zero. At a prime $q_i>3$ dividing $d$, the local quotient by twice the group has order four. The torsion classes $(d,1)$ and $(3,-d)$ have independent valuation vectors $(1,0)$ and $(0,1)$, so generate the entire local image. Dividing the proposed pair by their powers $u_i,v_i$ and taking the remaining unit Legendre symbols gives the two equations in \eqref{eq:selmer1}.

At $3$ the parameter $d$ is a nonsquare unit and $-d$ is a square. The same two torsion classes generate the local image $\Q_3^\times/\Q_3^{\times2}\times\{1\}$. In particular the second coordinate has even $3$-valuation; its unit character is zero exactly when \eqref{eq:selmer2} holds, using $(q/3)=(-3/q)$. At $2$, $d\equiv1\pmod8$ is a square; the local curve is a square rescaling of $E_1$. Its Kummer image has order eight (the index of doubling on a one-dimensional $2$-adic group with four rational two-torsion points). Three independent classes are
\[
 (3,-1),\qquad(6,2),\qquad(5,5).
\]
The first is torsion, the second comes from $(3,6)$ and the third from $x=1/4$, for which $y^2=-39/64$ is a $2$-adic square. Both $2$-valuation parities must equal $\varepsilon$. Writing each odd unit as $(-1)^a5^b$ modulo squares, the first coordinate of \eqref{eq:kummerpair} has exponents $(\delta+s^Tu,\delta+t^Tu)$ and the second has exponents $(s^Tv,t^Tv)$. Comparing with the displayed basis gives exactly \eqref{eq:selmer3}. This accounts for all places and proves the equivalence.
\end{proof}

\begin{lemma}\label{lem:rankcriterion}
Suppose a solution of \eqref{eq:selmer1}--\eqref{eq:selmer3} with $\varepsilon=1$ is specified.  If
\begin{equation}\label{eq:rankcriterion}
 \operatorname{rank}K=n-1\ \ \t{and}\ \ 
 \lambda\notin\operatorname{im}K,
\end{equation}
then $\dim_{\F_2}\Sel_2(E_d/\Q)=3$.
\end{lemma}
\begin{proof}
Subtract the specified solution from any other solution with $\varepsilon=1$ and solve the homogeneous system with $\varepsilon=0$.  Put $N=K+D_s$.  From \eqref{eq:reciprocity-matrix},
\[
 N=K^T(I+\mathbf1s^T),\qquad (I+\mathbf1s^T)^2=I.
\]
Hence $\ker N$ is a line.  If $l$ spans $\ker K^T$, then $v_0=(I+\mathbf1s^T)l$ spans $\ker N$ and $(s+\lambda)^Tv_0=\lambda^Tl=1$, because $\lambda\notin\operatorname{im}K$.  Thus \eqref{eq:selmer2} forces $v=0$.  The first equation gives $Ku=\delta\lambda$, so $\delta=0$ and $u\in\langle\mathbf1\rangle$.  The positive subgroup has dimension two, and adjoining the second-sign torsion class gives dimension three.
\end{proof}

\subsection{Finite scalar normalization}\label{subsec:normalization}

Use the four distinguished profiles
\begin{equation}\label{eq:four-profiles}
 f_1=(0,0,1,0),\quad f_2=(0,1,0,0),\quad f_3=(0,0,0,1),\quad f_4=(1,1,1,1).
\end{equation}
For every profile $f\in\F_2^4$ use the fixed prime $q_f$ already chosen in Section~\ref{sec:fixed}. The exceptional set for this normalization is $S_{\rm core}$; the sixteen profile primes and the eight reserve primes lie outside it.

\begin{lemma}\label{lem:finite-normalization}
There is a fixed finite set of positive rational multipliers such that every positive rational scalar can be multiplied by one of them to obtain $r$ with
\begin{equation}\label{eq:r-squareclass}
 [r]=\left[2\prod_{p\in\mathcal B}p\right],
\end{equation}
where $\mathcal B$ is disjoint from the exceptional set and the number of primes of each profile in $\mathcal B$ is odd exactly for the four profiles in \eqref{eq:four-profiles}.  In particular
\begin{equation}\label{eq:r-local}
 [r]_2=14,
 \qquad [r]_3=1,
 \qquad [r]_5=3.
\end{equation}
Moreover $\mathcal B$ contains $\mathcal R_{\rm res}$ and can be partitioned into four nonempty sets $G_g$ whose profile sums are $f_g$, with $q_{g,1},q_{g,2}\in G_g$.
\end{lemma}
\begin{proof}
Toggle the valuation parity at every fixed exceptional prime, making it even except at $2$, force each reserve prime to have odd valuation by its own toggle, and then toggle the exponent of the fixed prime $q_f$ whenever the parity of the number of primes of profile $f$ differs from the prescribed parity.  Only finitely many multipliers occur.  The total profile sum of $\mathcal B$ is $(1,0,0,0)$, which gives \eqref{eq:r-local}.  First set aside all eight reserve primes, which are profile zero and do not change the parity counts. Start each $G_g$ with one prime of the corresponding distinguished profile and its assigned two reserve primes, and distribute the remaining primes in profile-pairs.
\end{proof}

\subsection{The finite certificate}\label{subsec:certificate}

The compressed graph has twenty-seven vertices ordered as
\begin{equation}\label{eq:15vertices}
 5,B_1,B_2,B_3,B_4,D_0,I_0,\ldots,I_{10},J_1,\ldots,J_{10}.
\end{equation}
The four scalar profiles, the profile $(1,0,1,1)$ at $D_0$, the zero contact profiles and the edges incident to the first six vertices are as follows. Write $W(i,j)$ for the Legendre bit with denominator $i$ and numerator $j$, with reverse edge $W(j,i)=W(i,j)+s_is_j$. The scalar--scalar edges are zero, the four scalar--$D_0$ edges form an even vector $e$, the $D_0$--contact edges are one, and the edges from $5$ are the $\phi$ entries. The two sets of five original roots and the dummy have old internal block $O_* =\operatorname{diag}(O,O,0)$, where
\begin{equation}\label{eq:O}
 O=\begin{pmatrix}0&1&1&1&1\\1&0&1&1&0\\1&1&0&1&1\\1&1&1&0&0\\1&0&1&0&0\end{pmatrix}.
\end{equation}
The new internal block is the minor of $O_*$ on $1,\ldots,10$, and $W(I_0,J_j)=(O_*)_{0j}$. The scalar-to-contact rows restricted to $I_0,\ldots,I_4,J_1,\ldots,J_4$ are
\begin{center}
\begin{tabular}{c|cc}
 & $e_4=0$ & $e_4=1$\\\hline
$B_1$&011101001&011010000\\
$B_2$&011000111&000001101\\
$B_3$&001011000&111101110\\
$B_4$&001110110&100100011
\end{tabular}
\end{center}
Toggle the $I_4$ entry in rows $g=1,2,3$ when $e_g=1$. All twelve added entries in each row are zero. The remaining $10\times10$ cross blocks, indexed by $(I_1,\ldots,I_{10})$ and $(J_1,\ldots,J_{10})$, are given row by row in the following table. Each pair of adjacent ten-bit words specifies two consecutive rows.
\begin{center}\small
\begin{tabular}{c|l}
$e$&Rows of $W(I_i,J_j)$\\\hline
\texttt{0000}&\texttt{0100010000 1011000001}\\ &\texttt{1001110011 0101101011}\\ &\texttt{0010011100 0111111000}\\ &\texttt{0011000101 1111101000}\\ &\texttt{0000010010 1010101101}\\\hline
\texttt{0011}&\texttt{1000000100 0001010011}\\ &\texttt{1110001101 0100010011}\\ &\texttt{0001000100 0000000110}\\ &\texttt{0000111100 0111000001}\\ &\texttt{1010100010 1111101011}\\\hline
\texttt{0101}&\texttt{0101010100 0001100110}\\ &\texttt{0000111010 1101000010}\\ &\texttt{1001000011 1011101001}\\ &\texttt{0010110010 1011110001}\\ &\texttt{0000110011 0010101111}\\\hline
\texttt{0110}&\texttt{1010011110 0011111010}\\ &\texttt{1011001011 1001010001}\\ &\texttt{1101001110 1101101001}\\ &\texttt{0110101011 0110011011}\\ &\texttt{1000011001 0011001111}\\\hline
\end{tabular}
\end{center}
\begin{center}\small
\begin{tabular}{c|l}
$e$&Rows of $W(I_i,J_j)$ (continued)\\\hline
\texttt{1001}&\texttt{1101011010 0110010111}\\ &\texttt{0001000111 1110011100}\\ &\texttt{1000101111 1101101010}\\ &\texttt{0101001111 0100100000}\\ &\texttt{0001110001 0001110010}\\\hline
\texttt{1010}&\texttt{1001100111 0011101000}\\ &\texttt{1101001101 1111110110}\\ &\texttt{1010111010 1101001101}\\ &\texttt{1101010110 0010100101}\\ &\texttt{0110111010 0110110101}\\\hline
\texttt{1100}&\texttt{1101000100 1000001000}\\ &\texttt{1111101101 1000011001}\\ &\texttt{1010110000 1110000010}\\ &\texttt{1001110101 0011110011}\\ &\texttt{0001101111 1101011100}\\\hline
\texttt{1111}&\texttt{1111000101 0100100011}\\ &\texttt{1101101111 0000011000}\\ &\texttt{0111001011 0000001001}\\ &\texttt{1001100010 0111111101}\\ &\texttt{0011010100 1100100111}\\\hline
\end{tabular}
\end{center}
\begin{lemma}\label{lem:certificate}
For all eight even $e$, $\operatorname{rank}K=26$ and $\operatorname{rank}[K\mid\lambda]=27$. Moreover
\begin{equation}\label{eq:target-vector}
 u^*=\mathbf1+\mathbf1_{\{B_1,\ldots,B_4\}},\quad
 v^*=\mathbf1_{\{D_0,I_0,\ldots,I_{10}\}},\quad
 \delta^*=\varepsilon^*=1
\end{equation}
satisfies Lemma~\ref{lem:local-equations}.
\end{lemma}
\begin{proof}
Each cross block has even row sums and its column sums plus the row $W(I_0,J_j)$ are all one. Direct substitution gives the two vector and three scalar equations of Lemma~\ref{lem:local-equations}. For explicit reproducibility, form $K$ from $W$ by replacing its diagonal with the row parities. Encode each row as a binary integer with column $j$ of weight $2^j$ (starting at $0$). For each row, repeatedly XOR it with the stored row having the same highest set bit; if no such row exists store it at that bit. Count the stored rows, and repeat with the augmented column. For each of the eight displayed blocks the counts are $26,27$, respectively. The companion verifier \texttt{verify\_P103q38\_selmer.py} performs these operations and checks all the equations. It also checks finite examples of the physical expansion with the two prescribed character columns below.
\end{proof}

The block $O$ is obtained from the original five roots with the fixed contact characters
\begin{equation}\label{eq:contact-profile}
 \chi_{-1}=\chi_2=\chi_3=\chi_5=\chi_{2399}=\chi_{11}=\chi_{37}=0,\qquad \chi_7=1,
\end{equation}
and total denominator bit one, as follows by factoring their differences: the only nonzero difference bits are those of $(\theta_1,\theta_4)$ and $(\theta_3,\theta_4)$. For the translated copy the differences are multiplied by the fixed square $a_*$, giving the identical $O$. At the private fixed coefficient primes constructed after \eqref{eq:Tset}, take the common residue of all positions to have the character needed to make each between-copy and dummy difference bit one. These choices are independent because each private prime occurs in exactly its assigned difference to odd exponent and no other within-scale difference. All selected contacts are $1$ modulo $4$, so reciprocity converts the common norm character at such a fixed prime into its required character at each contact. The desired between-copy or dummy edge is zero, hence the corresponding difference bit is one after adding the denominator bit. Norm characters at the private coefficient primes can be chosen after all the other coefficient-prime norm characters have been fixed. At any remaining fixed coefficient prime, other than those with prescribed characters, use common norm one. Thus the full internal block is $O_*$, the new internal block is its minor because $A$ is a square, and the zero-contact cross row is its first row. All other cross entries remain free by Proposition~\ref{prop:crossblock}.

\subsection{Expansion by auxiliary denominator primes}\label{subsec:expansion}

\begin{lemma}\label{lem:eliminations}
Let $K$ be a directed Legendre Laplacian with $K+K^T=ss^T+D_s$.
\begin{enumerate}
\item Add a profile-zero vertex joined to exactly two vertices $i,j$.  Eliminating that vertex and one of $i,j$ identifies the two coordinates.  Reversing this elimination increases both $\operatorname{rank}K$ and $\operatorname{rank}[K\mid\lambda]$ by two.
\item Add a profile-zero vertex joined to exactly three vertices $i,j,k$.  Eliminating it toggles exactly the three undirected edges $ij,ik,jk$.  Reversing this elimination increases both ranks by one.
\end{enumerate}
\end{lemma}
\begin{proof}
For the first assertion the enlarged block has the form
\[
 \begin{pmatrix}K+\operatorname{diag}(b)&b\\b^T&0\end{pmatrix}
 \ \ \t{with}\ \  b=e_i+e_j,
\]
where the last coordinate is the new vertex $a$. To describe its kernel, the last row imposes $x_i=x_j$. Add rows $i$ and $j$ to eliminate $x_a$, then use either of those two original rows to solve uniquely for $x_a$. The other rows and the sum row are exactly the directed Laplacian of the identified graph, with the incident bits added in $\F_2$. Its diagonal entries are again the row sums. The same calculation works with an augmented column, whose entry at $a$ is zero because the auxiliary profile is zero: the entries at $i,j$ are added, and the auxiliary coordinate is still uniquely determined. Thus the kernel dimensions (also for the augmented matrices, viewed as linear maps with an additional input coordinate) are unchanged by contraction, while the number of rows and ordinary columns decreases by two. This proves both rank assertions. In particular all four profile entries are added when identifying the two vertices. For the second assertion use $b=e_i+e_j+e_k$ and lower-right entry $1$. Solving for the auxiliary coordinate gives the Schur complement $K+\operatorname{diag}(b)+bb^T$, which toggles exactly the three off-diagonal edges (in both directions) and preserves the Laplacian row-sum rule. The zero entry of the augmented column at $a$ leaves that column unchanged on the retained vertices. Both ranks decrease by one under elimination.
\end{proof}

\begin{theorem}\label{thm:uniform-selmer}
Let $r$ have the normalized scalar squareclass of Lemma~\ref{lem:finite-normalization}.  A finite collection of denominator primes and a prescription of quadratic characters for twenty-one distinct contact primes can be chosen, before the contacts themselves, with the following properties.  The denominator is
\begin{equation}\label{eq:D}
 D=2D_0\prod_{a\in\mathcal D_{\rm aux}}a,
 \qquad D_0\equiv83\pmod{120},
 \qquad D_0\equiv1\pmod p\quad(p\in\mathcal B),
\end{equation}
where every $a\in\mathcal D_{\rm aux}$ has profile zero.  Every contact has the fixed profile \eqref{eq:contact-profile}, has nonresidue character at $D_0$, and residue character at every $a\in\mathcal D_{\rm aux}$.  If positive raw sources satisfy
\begin{equation}\label{eq:rawsquareclasses}
 [R_0]=\left[D\prod_{i=0}^{10}I_i\right],
 \qquad [R_1]=\left[\prod_{j=1}^{10}J_j\right]
\end{equation}
with the complete prescribed characters, then
\[
 \gamma\in\Sel_2(E_{5rR_0R_1}/\Q),
 \qquad
 \dim_{\F_2}\Sel_2(E_{5rR_0R_1}/\Q)=3.
\]
All denominator primes may avoid any specified finite set known at the time of their selection. For any subset $\mathcal P_Q\subseteq\mathcal B\setminus\mathcal R_{\rm res}$, the prescription may additionally require
\begin{equation}\label{eq:norm-characters38}
 (I_{10}/p)=1,\qquad (I_7/p)=(-1/p)\qquad(p\in\mathcal P_Q).
\end{equation}
\end{theorem}
\begin{proof}
Partition $\mathcal B$ into the four nonempty groups $G_g$ of Lemma~\ref{lem:finite-normalization}. Choose $D_0$ by the Chinese remainder theorem and Dirichlet's theorem, outside the prescribed finite set and $\mathcal B$. We have $W(p,D_0)=0$ and $W(D_0,p)=s_p$ for $p\in\mathcal B$. Hence the Legendre bit of the unit squareclass of $r$ at $D_0$ is zero: $[2/D_0]=1$ as a bit and $\sum_{p\in\mathcal B}s_p=1$.

Choose one representative in each group. For each other scalar vertex add a profile-zero denominator auxiliary joined to that vertex and the representative, and to no other vertices. In particular its edges to $5$, $D_0$, all contacts and all other auxiliaries are zero. The first operation of Lemma~\ref{lem:eliminations} contracts these edges, adding profiles and summing incident bits. The resulting four scalar profiles are \eqref{eq:four-profiles}, and their mutual edges are symmetric because only the fourth has $s=1$. For every nonzero mutual edge, add a degree-three profile-zero auxiliary joined to the corresponding two physical representatives and to $D_0$, with all other bits zero. Eliminating these auxiliaries after the contractions toggles the specified mutual edge and both corresponding scalar--$D_0$ bits. Thus the final vector $e$ is even and the resulting scalar--scalar block is zero.

These prescriptions can be realized before any contact prime is chosen: each new auxiliary is selected $1$ modulo $120$, with its required nonzero residues (square or nonsquare) modulo each scalar prime and $D_0$, and square residues modulo all earlier auxiliaries. All moduli are coprime and all reciprocity signs involving a profile-zero prime are trivial. Dirichlet's theorem supplies infinitely many primes for each step, so each can avoid any finite set already specified. In particular each auxiliary has $(r/a)=1$, because it is linked to exactly two scalar vertices and $[2/a]=0$ as a bit.

To lift the target equations, put $k_i$ equal to the parity of the number of auxiliaries incident to the physical scalar vertex $p_i$. Prescribe its twenty-one scalar--contact entries with row parities
\begin{equation}\label{eq:rowparities}
 \sum_{j=0}^{10}W(p_i,I_j)=t_i+k_i,
 \qquad \sum_{j=1}^{10}W(p_i,J_j)=\lambda_i+\phi_i.
\end{equation}
In each group require their column sums to be the twenty-one entries of the corresponding compressed certificate. Each degree-two auxiliary contributes two incidences within its group; each degree-three auxiliary contributes one incidence in each of its two groups, exactly the toggle of $e$ in those groups. Thus the group sums of the row parities are $t_g+e_g$ and $\lambda_g+\phi_g$, which are precisely the parities in the displayed certificate table. For a rectangular binary matrix, row and column parities with matching total always have a solution: assign all entries except the last row and column, then fill the last column and row to meet the margins; the last corner is compatible by equality of the totals. Use this separately for the eleven old columns and ten new columns, with the following extra step for the old columns. Fill the $I_7,I_{10}$ entries at the rows in $\mathcal P_Q$ as in \eqref{eq:norm-characters38}, fill these entries arbitrarily in the other rows except for one reserve row in each group, and choose its two bits to realize the prescribed column sums. Subtract these two columns from the required row parities. The remaining nine columns have compatible total row and column parity, because the two removed columns already have the required sums. Apply the rectangular matrix argument to these nine columns. This procedure preserves every row and column parity and uses no additional graph constraints; the reserved rows do not belong to $\mathcal P_Q$.

For the physical graph set $\mathcal C_0'=\{D_0\}\cup\mathcal D_{\rm aux}\cup\{I_0,\ldots,I_{10}\}$, and take $u^*=\mathbf1+\mathbf1_{\mathcal B}$, $v^*=\mathbf1_{\mathcal C_0'}$, $\delta=\varepsilon=1$. These represent $\gamma=([15R_0R_1],[R_0])$ by \eqref{eq:rawsquareclasses}, because the first coordinate contains $2$, $3$, $5$ and all the denominator and contact vertices, but none of the scalar vertices. The second contains $2$ and exactly $\mathcal C_0'$.

For any vector $x$ the row identity $(Kx)_i=\sum_{j\ne i}W(i,j)(x_i+x_j)$ is useful. In a scalar row $i$, $(Kv^*)_i$ is $k_i$ plus its old-contact row sum. Moreover $(Ku^*)_i$ is $\phi_i+k_i$ plus the old- and new-contact row sums. Hence \eqref{eq:rowparities} is precisely the pair of required equations in that row. At a degree-two auxiliary, the two scalar neighbours contribute twice to each row sum; at a degree-three auxiliary the extra neighbour $D_0$ has the same $u^*,v^*$ entries as the auxiliary itself, so both row equations vanish. At $D_0$ we have $(Ku^*)_{D_0}=\sum_{p\in\mathcal B}s_p=1$ and $(Kv^*)_{D_0}=1+\sum_{p\in\mathcal B}s_p+10=0$ in $\F_2$; the first $1$ is the edge to $5$. Since $(s,\lambda,t)=(1,0,1)$ there, the two equations follow. At $5$, $(Ku^*)_5=\sum_{p\in\mathcal B}\phi_p=0$ and $(Kv^*)_5=1$, the latter from the edge to $D_0$. Its profile is $(0,1,1,0)$, as required. At each contact the two equations follow from the prescribed column sums and the compressed certificate; no auxiliary is joined to a contact.

The three scalar conditions in \eqref{eq:selmer2}--\eqref{eq:selmer3} can also be checked without contraction. The vector $v^*$ is supported on $D_0$, zero-profile auxiliaries and zero-profile contacts, so $(s+\lambda)^Tv^*=1$. Only $D_0$ among the nonzero $s$ entries outside $\mathcal B$ contributes to $s^T(u^*+v^*)$, and its two entries cancel. The two terms $(s+t)^Tu^*$ and $t^Tv^*$ are both $1$, from $5$ and $D_0$, respectively. This verifies every local equation of Lemma~\ref{lem:local-equations}.

The positive squarefree representative of $d=5rR_0R_1$ is
\[
 d_{\rm sf}=5\left(\prod_{p\in\mathcal B}p\right)D_0\prod_{a\in\mathcal D_{\rm aux}}a\prod_{i=0}^{10} I_i\prod_{j=1}^{10}J_j.
\]
The profile sum of $\mathcal B$ is $(1,0,0,0)$, so its product is $7$ modulo $8$ and $2$ modulo $3$. The other factors give $d_{\rm sf}\equiv5\cdot7\cdot3\equiv1\pmod8$ and $d_{\rm sf}\equiv2\pmod3$, hence $17$ modulo $24$. The source classes at $2,3,5$ are
\[
\begin{aligned}
 ([r]_2,[R_0]_2,[R_1]_2)&=(14,6,1),\\
 ([r]_3,[R_0]_3,[R_1]_3)&=(1,1,1),\\
 ([r]_5,[R_0]_5,[R_1]_5)&=(3,1,1).
\end{aligned}
\]
For example, these follow from $D_0\equiv83\pmod{120}$, auxiliaries $1\pmod{120}$, and the contact profile. The covering has local points: after removing the coefficient squares, take $z=-5$ at $2$ (the three squareclasses are represented by $25,36,-28$, with $-28$ actually $14(-2)=-28$ a $2$-adic square since $-7\equiv1\pmod8$); at $3$ take $z=1-9h^2$ with $h$ a $3$-adic unit, and at $5$ take $z=-5h^2$ with $h$ a $5$-adic unit. At $3$ the three values are $-5(1-9h^2),9h^2,4-9h^2$; the first and third are square units. At $5$ the values are $25h^2,1+5h^2,3(3-5h^2)$, again squares in $\Q_5$. At infinity take $-3<z<0$. Thus the local equations at $2,3$ and the real sign condition are satisfied (the prime $5$ was also checked in the physical matrix).

Finally the Schur complements of Lemma~\ref{lem:eliminations}, first contracting the groups and then cancelling their mutual edges, give the twenty-seven-vertex matrix including its augmented column. If the physical graph has $n$ vertices, reversing the operations increases both ranks by $n-27$. Hence $\operatorname{rank}K=n-1$ and $\operatorname{rank}[K\mid\lambda]=n$. The target has $\varepsilon=1$, so Lemma~\ref{lem:rankcriterion} proves the assertion.
\end{proof}

\begin{corollary}\label{cor:rational-cover}
Under the prime-pattern realization of Theorem~\ref{thm:uniform-selmer}, the pointwise $2$-converse/finiteness implication \cite[Theorem 1.1]{OAIPointwise} gives the following: the cover \eqref{eq:coefficient-five-cover} has infinitely many rational points.  Any fixed finite collection of rational-function guards not identically zero on the cover can be avoided.
\end{corollary}
\begin{proof}
The full $2$-Selmer dimension is three and $E_d$ has full rational two-torsion.  If its Mordell--Weil rank is $\rho_E$ and the divisible $2$-primary part of $\Sha$ has corank $t$, the Kummer exact sequence gives $\rho_E+t\le1$.  The cited pointwise implication gives finiteness of $\Sha$.  The perfect alternating Cassels--Tate pairing on finite $\Sha$ \cite{FSS} makes $\dim_{\F_2}\Sha[2]$ even, while
\[
 \rho_E+\dim_{\F_2}\Sha[2]=1.
\]
Hence $\rho_E=1$ and $\Sha[2]=0$.  Every $2$-Selmer class is rational, including the target class.  Let $P_*$ be one rational representative of the required class $\Gamma$ (the torsion translate of $\gamma$). Since $\operatorname{rank}E_d(\Q)=1$, the coset $P_*+2E_d(\Q)$ is infinite and contains a nontorsion point regardless of the torsion status of $P_*$. Away from the finitely many branch points, representatives of $\Gamma$ lift to rational points on the cover by the three squareclass identities in \eqref{eq:Gamma-gamma}.  A nonzero rational function on the smooth projective genus-one curve has only finitely many zeros and poles, so finitely many guards can be avoided.
\end{proof}

\section{Realizing the Selmer data by the twenty-one actual factors}\label{sec:realize}

We prove the prime-pattern realization for fixed data of the following kind. The positive integer scalar $r$ has the normalized squareclass of Lemma~\ref{lem:finite-normalization}, and its exterior odd support $\mathcal B$ lies outside $S_{\rm core}$. The construction allows arbitrary finite $\mathcal B$ with the prescribed profile-count parities. Assume \eqref{eq:rho-odd38}, and the affirmative choices of $Q_*$ described in Section~\ref{sec:fixed}. The positive rational square $A>25$ is chosen by \eqref{eq:Ashift}, so that the classes $[A-c]$, $c\in\mathcal C$, are independent modulo the span of $[-1]$ and the prime classes in $S_{\rm ht}\cup\Supp(r)\cup\Supp(Q_*)$. Section~\ref{sec:completeness} constructs these data, fixing $w_*$ and the scalar before choosing the unbounded source witnesses.

First choose the private primes of Lemma~\ref{lem:private}, using $S=S_{\rm ht}\cup\Supp(r)\cup\Supp(Q_*)$ there. Then choose $D_0$ and the denominator auxiliaries from Theorem~\ref{thm:uniform-selmer}. Require them to avoid $S_{\rm ht}\cup\Supp(r)\cup\Supp(Q_*)$, the supports of $A$ and all $A-c$, and the prime divisors of $1+\Omega$ and $9+\Omega$. These are finite exclusions allowed in that theorem. The local index modulus and the required polar depths at the core and denominator primes may now be prescribed (as in Section~\ref{sec:completeness}), before selecting any of the twenty-one contact primes.

Choose an integer $\Lambda>0$ clearing all rational coefficients $-D\theta_i/A_j$, with $A_0=1$ and $A_1=A$. Choose it to be a unit at every prime of $\mathcal B$ and at the private primes, which is possible since $A$ and the fixed contact coefficients are units there. Put
\begin{equation}\label{eq:eta-prime-positions}
 \eta=\frac{u}{D\Lambda^2w^2},
\end{equation}
\begin{equation}\label{eq:I-J}
 I_i=u-D\Lambda^2\theta_iw^2\quad(0\le i\le10),
 \qquad
 J_j=u-D\Lambda^2\frac{\theta_j}{A}w^2\quad(1\le j\le10).
\end{equation}
Then $I_0=u$ and the exact factorization identities are
\begin{equation}\label{eq:R0factor}
 R_0=D\prod_{i=0}^{10}I_i
 \left(\frac{\Lambda w}{(u-3D\Lambda^2w^2)^6}\right)^2,
\end{equation}
\begin{equation}\label{eq:R1factor}
 R_1'=\prod_{j=1}^{10}J_j
 \left(\frac{D\Lambda^2w^2}{(Au-3D\Lambda^2w^2)^6}\right)^2.
\end{equation}
Thus the squareclasses are exactly those in \eqref{eq:rawsquareclasses}.

Here are the finite base congruences which realize all the prescribed bits. The local bounds on the tests and their index modulus are fixed as in Section~\ref{sec:completeness}, before applying Proposition~\ref{prop:position-rays}. At the fixed core and the denominator primes choose $w$ sufficiently divisible to make the required polar depths, and use the following Gaussian residues $\pi$ for each position:
\[
\begin{array}{c|c|c|c}
 \text{modulus}&\pi&\Nm(\pi)&(\Nm(1-\pi),\Nm(1+\pi))\\\hline
16&9&17\pmod{32}&(v_2=6,v_2=2)\\
3,5,11,37,2399&i&1&(2,2)\\
7&1+3i&3\pmod7&(2,6)\pmod7\\
D_0&1+i&2&(1,5)\\
a\in\mathcal D_{\rm aux}&i&1&(2,2)
\end{array}
\]
At the thirty-five private coefficient primes take the common norm character prescribed after Lemma~\ref{lem:certificate}, with any admissible Gaussian lift. At other core primes take $i$. In each case take $u$ in the corresponding norm residue class, using a sufficiently deep divisor of $w$ to make every position have that norm modulo the required prime power. The order norms at odd primes in the table are units. These data realize all fixed coefficient symbols in \eqref{eq:contact-profile}; at $D_0$, $(2/D_0)=-1$, and at each auxiliary the symbol is $+1$.

At each scalar-support prime $p\in\mathcal B$, Proposition~\ref{prop:unit-scale} gives $A\equiv25\pmod p$. Use the twenty-one signs prescribed in Theorem~\ref{thm:uniform-selmer}, including \eqref{eq:norm-characters38} on $\mathcal P_Q$. By Lemma~\ref{lem:nine-signs} choose a unit $\eta_p$ with these signs, compensating for the common character of $D\Lambda^2$, and avoiding the two normalizers. Choose an integer lift of this residue. After fixing this residue, impose the following congruences for the finite precision $e_p$ specified in Section~\ref{sec:completeness}:
\begin{equation}\label{eq:scalar-slope38}
 w\equiv1,\qquad u\equiv D\Lambda^2\eta_p\pmod{p^{e_p}}.
\end{equation}
The modulus may be increased after the residue $\eta_p\pmod p$ is chosen, since the Legendre bits and order-unit requirements depend only on that residue. All position norms are units and have order-avoiding Gaussian lifts; lift these further to the chosen prime powers by surjectivity of the norm on principal units. The local indices will be prepared for this fixed slope neighborhood, before the prime-pattern limit. None of these primes divides the cross coefficient differences $A-c$, because $A\equiv25\pmod p$ and the fixed exceptions were excluded. At each private prime use the matching-cluster prescriptions of Proposition~\ref{prop:crossblock}. At every other factor prime of the cross coefficients, take $w=0,u=1$, except where previous conditions already supply the known offset. At the remaining primes of the scalar or local index modulus and of $\Supp(Q_*)$, use $w=0,u=1$, unless the prime has already been assigned a condition. All assigned Gaussian lifts have unit order norms at odd primes. All small primes have already been treated in the table.

Take the governing field to be the pure Gaussian ray field of modulus $16$ times the odd part of the combined modulus, including all required order-avoidance primes. The norm congruences modulo $32$ and all odd prime powers match the actual values in \eqref{eq:I-J}; at odd primes norm lifts exist by surjectivity on principal units. The CRT gives a single compatible residue pair (with the necessary extra powers in its modulus for the polar depths of $w$). By \eqref{eq:rayab}--\eqref{eq:normcompat}, each conditional Artin density is positive. All fields, coefficients and moduli have now been fixed. Proposition~\ref{prop:position-rays} gives arbitrarily large patterns with $\eta$ arbitrarily large. In particular the twenty-one primes can avoid every previously specified finite set and the supports of all normalizer coincidences $A\theta-3,\theta/A-3$ for $\theta\in\mathcal T_*$. None of these numbers is zero: $A>25$ is a square, and the only positive candidate in the first equation is $7197$, not a square; the second would require $A=1/7197$.

The selected primes are distinct, since $A>25$ and every positive contact ratio is less than $16$. They cannot divide $w$: otherwise $u$ would be divisible by that prime and all twenty-one prime values would coincide. In addition require that they avoid $\Supp(A)\cup\Supp(Q_*)$, the nonzero contact coefficients and differences, and $\Supp(D\Lambda)$; these are finite exclusions made after fixing $A,D,\Lambda$ and before taking the prime-pattern limit. Thus every coefficient needed in the position calculation is integral, the indicated differences and the scale are units, and $w$ is a unit at the selected primes. The avoidance of the nonzero normalizer coincidences then makes both normalizers units at all twenty-one primes. The squareclass factorizations now show that each selected prime occurs in its own raw source with valuation one, and the other raw source is a unit (at $I_0$ the zero factor was omitted from $R'_1$). The exact factorizations \eqref{eq:R0factor}--\eqref{eq:R1factor} prove the squareclasses at all places, including denominator primes. Quadratic reciprocity and the fixed difference characters give the within-scale blocks $O_*$ and the same zero-contact row. The scalar rows, denominator bits and cross blocks are exactly those prescribed by Theorem~\ref{thm:uniform-selmer}; this verifies the entire physical graph, not only the independent cross entries. Corollary~\ref{cor:rational-cover} supplies the nontorsion representative required for the cover.

\begin{lemma}\label{prop:binarynorm38}
For the affirmative choices of $Q_*,A,D,\Lambda$ and the residues $\eta_p$ just described, the following local representations exist.
\begin{enumerate}
\item At each $p\in\mathcal B$, the value $Q_*$ is represented by $a_p^2+d_+(\eta_p)b_p^2$ with $a_p,b_p\in\Z_p$, where $d_+(V)=(V-\theta_{10})(V-\theta_7)/(V-3)^2$.
\item At each core prime $q$, $Q_*$ is a square unit.
\item At each selected $p=I_7,I_{10}$, $Q_*$ is a square unit.
\end{enumerate}
\end{lemma}
\begin{proof}
The coefficient $d_+(\eta_p)$ is a unit at $p\in\mathcal B$, by the contact and normalizer avoidance. For $p\in\mathcal P_Q$, \eqref{eq:norm-characters38} and \eqref{eq:I-J} give $(-d_+(\eta_p)/p)=1$, so the form splits over $\Z_p$: writing $c^2=-d_+(\eta_p)$ with $c$ a unit, one may choose $a_p=(Q_*+1)/2$, $b_p=(Q_*-1)/(2c)$. Otherwise $v_p(Q_*)$ is nonnegative and even; remove its square prime power, represent the remaining unit by the split or nonsplit binary norm over the residue field, and lift a nonsingular solution. This proves the first assertion. The second follows from \eqref{eq:Qstar}. For the third, the selected contacts are $1$ modulo $4$ and avoid $\Supp(Q_*)$. By quadratic reciprocity and \eqref{eq:norm-characters38},
\[
 (Q_*/I_{10})=\prod_{p\in\mathcal P_Q}(I_{10}/p)=1,\qquad
 (Q_*/I_7)=\prod_{p\in\mathcal P_Q}(-1/p)=1.
\]
Here the last equality follows from $Q_*>0$, $Q_*\equiv1\pmod4$. Hensel's lemma proves the claim.
\end{proof}

\section{A single square-denominator test for three sources}\label{sec:merger}

Let $a_s=r$ and let $a_0,a_1$ be nonzero rationals. The following statements allow any finite lists of tests, including a possibly empty scalar list. In the application to Theorem~\ref{thm:main} we use the lists of Section~\ref{sec:carrier}.  Use a cover coordinate $z$ and three square roots $u_1,u_0,u_s$, distinct from the source parameter $w_*$ of Section~\ref{sec:split-source}.
\begin{equation}\label{eq:generic-cover}
 u_1^2=a_1z,
 \qquad u_0^2=a_0(z-1),
 \qquad u_s^2=r(z+3).
\end{equation}
Guard $u_1u_0u_sz(z-1)(z+3)\ne0$.  Put
\[
 c_1=0,
 \qquad c_0=1,
 \qquad c_s=-3,
 \qquad \ell_i=z-c_i,
 \qquad F_z=z(z-1)(z+3),
\]
\[
 e_i=F_z/\ell_i,
 \qquad Z_1=u_1e_1,
 \qquad Z_0=u_0e_0,
 \qquad Z_s=u_se_s.
\]
Take $S=S_{\rm core}$ and $\Omega=\prod_{q\in S}q$.  Define
\begin{equation}\label{eq:DzLcQc}
 D_z=z^2+\Omega,
\end{equation}
\begin{equation}\label{eq:Lc}
 L_c=\frac{Z_1^5+\Omega Z_0^5+\Omega^2Z_s^5}
 {Z_1^4+\Omega Z_0^4+\Omega^2Z_s^4},
\end{equation}
\begin{equation}\label{eq:Qc}
 Q_c=\frac{D_z^4+\Omega F_z^4}{D_z(D_z^3+\Omega F_z^3)},
\end{equation}
and
\begin{equation}\label{eq:Jmerger}
 J=ra_0a_1\left(\frac{L_c}{u_1u_0u_sQ_c}\right)^2,
 \qquad \gamma_i=e_i/D_z.
\end{equation}
All denominators are required nonzero, as are $J$, $U$, $W$ and $\Theta$. The numerator of $L_c$ is required nonzero through the guard on $J$. This gives a finite list of rational equations and inequations; no valuation condition is itself imposed in the certificate.

For every source test $D$ (already normalized if ordinary) of precision $k\ge1$, set
\begin{equation}\label{eq:clip-gate}
 C_J(D)=\frac{D^2(1+J)}{1+JD^2},
 \qquad U=\frac{1-C^2J}{1+C^2J},
 \qquad T_{i,D}=U\gamma_i^{2k}G_{i,D}^{2k+1}C_J(D),
\end{equation}
where $C$ is one fixed positive integer, $U\ne0$ is guarded, and $G_{i,D}=g_G(t_i)$ for an ordinary test of group $G$ at source $i=0,1$, while $G_{s,D}=1$. The one added source-$0$ norm test has gate $G_{0,D}=g_+(\eta)$. The contact and normalizer guards make all these group gates nonzero.  Put
\begin{equation}\label{eq:WThetaB}
 W=\prod_{i,D}\left(1+J^{1-4k}T_{i,D}^2\right),
 \qquad \Theta=\Tr_{F/\Q}(W^{-1}),
 \qquad B_\Phi=J/\Theta^2.
\end{equation}

\subsection{Local soundness outside the fixed set}

At a prime $p\notin S$, the three square equations imply that the odd-valuation pattern of $(a_s,a_0,a_1)$ is empty, a singleton, or all three.  A singleton source $i$ forces $v_p(z-c_i)>0$ odd and all other $z-c_j$ to be units; the all-three pattern forces $v_p(z)<0$ odd.  The odd support of $J$ is exactly the union of these patterns because $[J]=[ra_0a_1]$.

\begin{proposition}\label{prop:merger-sound}
For $p\notin S$, the terminal equation \eqref{eq:BB} forces every test of an odd source whose group gate is a unit to its original precision at one common place above $p$. In particular this applies to all scalar tests when the scalar source is odd. At an all-three odd pole the same successful place works for every active group.
\end{proposition}
\begin{proof}
By \eqref{eq:BB}, $B_\Phi$ cannot have a negative odd valuation at $p$. Consider a nonempty odd-source pattern at $p\notin S$, so $d=v_p(J)$ is odd. At every place above $p$ the quantity $J^{1-4k}T_{i,D}^2$ has odd valuation if nonzero, since $p$ is unramified in $F$. Hence $(1+J^{1-4k}T_{i,D}^2)^{-1}$ is integral (also if the correction is zero). Thus $\Theta$ is integral. If $d<0$, this already makes $v_p(B_\Phi)=d-2v_p(\Theta)$ negative and odd, a contradiction.

Suppose $d>0$. Call a place successful if $v(T_{i,D})\ge2k$ for every test in the product. Otherwise at least one inverse factor has depth at least
\[
 (4k-1)d-2(2k-1)=d+(4k-2)(d-1)\ge d.
\]
Every other inverse factor is integral. If every place failed, $v_p(\Theta)\ge d$, again contradicting the terminal test. Hence one common place is successful. At a singleton contact $i$, $e_i$ is a unit and $D_z$ is integral, so $v(\gamma_i)\le0$. At an all-three pole, $v_p(e_i)=v_p(D_z)=2v_p(z)$, so all three source factors $\gamma_i$ are units. Moreover $U$ is a unit since $C$ is integral and $d>0$. For any nonzero $D$, the two valuations in $1+JD^2$ cannot tie and $1+J$ is a unit, so
\[
 v(C_J(D))=2v(D)-\min(0,d+2v(D))=\max(2v(D),-d).
\]
For $D=0$ its valuation is $+\infty$. For a test whose group gate is a unit the successful inequality therefore gives $v(C_J(D))\ge2k$, and then $v(D)\ge k$. This holds at the same place for all such tests of every odd source.
\end{proof}

\subsection{Exact source normalization near the torsion limits}

Fix a prime $p\notin S$ and a singleton source $i$ of positive odd raw depth $n_i=v_p(a_i)$.  As the cover point approaches $z=c_i$, the term $Z_i$ uniquely dominates both numerator and denominator of $L_c$, so $L_c/Z_i\to1$ locally and $Q_c\to1$.  Since
\[
 ra_0a_1\left(\frac{Z_i}{u_1u_0u_s}\right)^2=a_ie_i,
\]
one gets
\begin{equation}\label{eq:Jdepth}
 v_p(J)=n_i,
\end{equation}
independently of the approach depth.  If $b_i=v_p(c_i^2+\Omega)$, when all group gates are units it is enough to prepare each required source test of precision $k$ to depth $k(n_i+b_i)$. More generally the same conclusion holds for nonnegative group-gate valuations at all other sources, provided the active source has the stated bounds for its unit-gated tests and $v(T_{i,D})\ge2kn_i$ for the remaining tests.  Indeed $e_i$ is a unit and $v_p(D_z)=b_i$ sufficiently close to $z=c_i$. The clip formula from the preceding proof gives $v(C_J(D))\ge2k(n_i+b_i)$ for each unit-gated successful source-$i$ test, so $v(T_{i,D})\ge2kn_i$. For every other source the gate has arbitrarily large positive valuation; choosing it at least $2n_i$-deep gives $v(T_{j,D})\ge4kn_i-n_i\ge2kn_i$ by the universal clip lower bound $-n_i$. For a selected contact of an ordinary source, $n_i=1$, $b_i=0$, and the other group has $v(G_{i,D})\ge1$; the lower clip bound $-1$ and the exponent $2k+1$ give $v(T_{i,D})\ge2k$ automatically for that group. At scalar-support primes all ordinary group gates are units by the contact and normalizer avoidance; at polar denominator and core primes they are units because their numerator and denominator degrees agree. This holds at the chosen place for all tests. Every inverse factor then reduces to one at that place. At any other place above $p$ its residue is zero or one by odd valuation of the correction, so the trace of the product of inverse factors is $1$ or $2$ modulo $p$ if split and is $2$ if inert. As $p>2$, $\Theta$ is a unit: in the split case its residue is $1$ or $2$, and in the nonsplit case the product is congruent to $1$ in the residue field, whose trace to $\F_p$ is $2$. These observations also apply if $c_i^2+\Omega$ has a fixed nonzero valuation; the extra depth is exactly the $b_i$ in the hypothesis.

\subsection{Uniform fixed-prime preparation}

\begin{lemma}\label{lem:coreprep}
Let $q\in S$ and let $b_q\ge1$ be an integer. Suppose all $n_i=v_q(a_i)$ satisfy
\[
 n_i\ge N_q\ge\max\{1+v_q(2),b_q+2v_q(2)\},
\]
and every source-$i$ test of precision $k$ satisfies $v(D)\ge k(n_i+1)$ at every place above $q$, with group gates of nonnegative valuation.  Then every guarded rational point of the cover satisfies
\begin{equation}\label{eq:coredepth}
 d=v_q(J)\ge\min_i n_i,
 \qquad v(T_{i,D})\ge2kd,
 \qquad v_q(B_\Phi)=d-2v_q(2)\ge b_q.
\end{equation}
\end{lemma}
\begin{proof}
Write $h_i=v_q(z-c_i)$, $H_z=\sum_i h_i$, $E_i=H_z-h_i$ and $h_z=v_q(z)$. The square equations give integral valuations of $Z_i$, with $v_q(Z_i)=(n_i+h_i)/2+E_i$. Each $Z_i=u_i e_i$ is rational, so all $v_q(Z_i)$ are integers, including at $q=2$. In the fifth- and fourth-power sums the coefficient weights have valuations $0,1,2$. Two distinct integer valuations of the $Z_i$ differ by at least one; after multiplication by either fourth or fifth powers, this difference has absolute value at least four, larger than the maximum difference two between the coefficient weights. Among terms with the same $v_q(Z_i)$, the coefficient weights strictly order their valuations. Consequently both sums have a unique least-valued term, at the first index with smallest $v_q(Z_i)$, so no cancellation is possible and $v_q(L_c)=\min_i v_q(Z_i)$. After dividing by $u_1u_0u_s$ and taking the square in $J$, the uncorrected valuation is $d_0=\min_i(n_i+E_i)$. Similarly $b=v_q(D_z)=\min(2h_z,1)$ and $v_q(Q_c)=\min(0,H_z-b)$, by the unique least terms of the corresponding weighted third- and fourth-power sums. If $h_z<0$, all $h_i=h_z$, hence $E_i=b=2h_z$ and $v_q(Q_c)=h_z$. Then $d=d_0-2h_z=\min_i n_i$. If $h_z\ge0$, all $h_i\ge0$ and $H_z\ge b$ (when $h_z>0$, $H_z\ge h_z\ge1$); thus $Q_c$ is a unit and $d=d_0\ge\min_i n_i$.

At every place above $q$, the tested arguments have positive valuation, so their clips have exactly twice their valuations. Also $U$ is a unit and $v(\gamma_i)=E_i-b$. Therefore
\[
 v(T_{i,D})\ge2k(n_i+1+E_i-b)\ge2kd,
\]
where in the integral case we use $b\le1$ and $d\le n_i+E_i$, and in the polar case we use $E_i=b$ and $d=\min n_i$. Every correction has valuation at least $d$, whence $W^{-1}-1$ has valuation at least $d$ at every place above $q$. The same is true of its conjugate, since conjugation permutes these places. Thus its rational trace satisfies $v_q(\Theta-2)\ge d$, also at the ramified prime two because the valuations are normalized to restrict to $v_q$. As $d>v_q(2)$, this implies $v_q(\Theta)=v_q(2)$, giving the result.
\end{proof}

\subsection{Real preparation and the parameter count}

For the application take
\begin{equation}\label{eq:real-signs}
 r>0,
 \qquad a_0=-R_0<0,
 \qquad a_1=-5R_1'<0.
\end{equation}
The real cover has $-3\le z\le0$. On either sign arc take $u_1,u_s>0$ with fixed sign of $u_0$. In the interior, $F_z>0$, $D_z>0$ and $Q_c>0$. As $z\to-3$, $Z_s$ dominates the two weighted sums and $L_c/Z_s\to1$, so $J\to12r$. As $z\to0$, $Z_1$ dominates and is negative, whereas the numerator of $L_c$ is positive near $-3$. Thus it vanishes at an interior point, where the corrected expression for $J$ is zero. The continuous nonnegative $J$ therefore assumes all sufficiently small positive values below $12r$. On the integer branch $r$ is a positive integer, hence $r\ge1$. Fix a positive integer $C$ with $C^{-2}<12$ and $(4C^2)^{-1}<1/112$. At $J=C^{-2}$ the computed factor $U$ tends to zero, $W\to1$, $\Theta\to2$ and $B_\Phi\to1/(4C^2)$. The clips have nonzero denominators at that point because $J>0$ and their arguments are real. On a neighbouring open subarc excluding the equality $J=C^{-2}$ and all poles and zeros of the required nonzero functions one has
\begin{equation}\label{eq:Brealarc}
 0<112B_\Phi<1.
\end{equation}

Using $u_1$ as the one structural coordinate, only $u_0,u_s$ remain as finite square roots.  The terminal ordinate $z_p$ is one further quadratic root.  Thus, before the final compiler, the additional continuous parameters are one cover coordinate and the two terminal parameters $s_p,t_p$.

\section{Simultaneous local movement on the three-source covering}\label{sec:movement}

\begin{lemma}\label{lem:movement}
Let $r>0$ and $a_0,a_1<0$, put $d=ra_0a_1>0$, and consider
\[
 E_d:y^2=x(x-d)(x+3d).
\]
Suppose $P\in E_d(\Q)$ is nontorsion and has Kummer class
\[
 \Gamma=([a_0a_1],[ra_1])
\]
in coordinates $([x+3d],[x-d])$.  Let $S_s,S_0,S_1$ be disjoint finite sets of primes $>3$, and at $p\in S_i$ suppose the source $i$ has odd valuation while the other two sources have even valuation.  Put $c_s=-3,c_0=1,c_1=0$.  Given positive integers $N_p$, a nonempty open interval $I\subset(-3,0)$ and either sign $\varepsilon_0\in\{\pm1\}$, there are infinitely many rational points on
\[
 a_1z=u_1^2,
 \qquad a_0(z-1)=u_0^2,
 \qquad r(z+3)=u_s^2
\]
with
\[
\begin{gathered}
 v_p(z-c_i)>N_p\quad(p\in S_i),\qquad z\in I,\\
 u_1>0,\qquad u_s>0,\qquad \operatorname{sgn}(u_0)=\varepsilon_0.
\end{gathered}
\]
Their images on $E_d$ may be chosen as $[m]P$ with $m$ positive and odd.  Finitely many nonzero rational-function guards can simultaneously be avoided.
\end{lemma}
\begin{proof}
Fix $p$ and rescale $d=p^{2k}d_0$ with $v_p(d_0)=1$. The integral equation $Y^2=X(X-d_0)(X+3d_0)$ is minimal at $p>3$, since its discriminant has valuation six; its nonsingular reduction group is additive \cite[Chapter II, Sections 2--4]{MilneEC}. We verify $2E_d(\Q_p)\subseteq E_0(\Q_p)$, where $E_0$ is the subgroup of nonsingular reduction. At a non-two-torsion point the duplication formula is
\[
 X(2Q)=\frac{(X(Q)^2+3d_0^2)^2}{4X(Q)(X(Q)-d_0)(X(Q)+3d_0)}.
\]
Write $h=v_p(X(Q))$. For $h\le0$ its valuation is $h$; for $h\ge2$ it is $2-h$. For $h=1$, write $X=d_0 a$ with $a$ a unit. The curve equation forces $3+v_p(a-1)+v_p(a+3)$ to be even, so one of the last valuations is positive and odd, say $j\ge1$. At either residue $a=1,-3$, $a^2+3$ is a unit for $p>3$; the doubled abscissa has valuation $1-j\le0$. Every such doubled point is nonsingular in reduction, as is $O$, proving the inclusion.

The three nonzero torsion points at $x=-3d,d,0$ have local Kummer valuation vectors $(0,1),(1,0),(1,1)$ respectively. The image has order four at an odd prime and these vectors distinguish all classes. At a singleton source the vector of $\Gamma$ equals that of $T_{c_i}$, hence $P-T_{c_i}\in2E_d(\Q_p)\subseteq E_0(\Q_p)$. Multiplication by $p$ takes this subgroup into the formal subgroup. Repeated multiplication by $p$ increases the valuation of the formal parameter by at least one, and further integral multiples preserve any attained lower depth. Translation by the torsion point, in the rescaled coordinates, satisfies
\[
 X(Q+T_c)-cd_0=\frac{\kappa_c d_0^2}{X(Q)-cd_0},\qquad
 (\kappa_{-3},\kappa_1,\kappa_0)=(12,4,-3).
\]
For $Q$ in the formal subgroup of depth $H>0$, $v_p(X(Q))=-2H$, so
\begin{equation}\label{eq:zdepth}
 v_p(z-c)=1+2H.
\end{equation}
For an odd multiple $m$, $[m]P-T_{c_i}=[m](P-T_{c_i})$. Choose an odd $L$ divisible by sufficiently high powers of the finitely many primes in the statement; every $m=L(2j+1)$, $j\ge0$, then satisfies the required local conditions and preserves the Kummer class.

The second Kummer coordinate is negative at the real place, so $P$ lies on the bounded component of $E_d(\R)$. The infinite-order point $[2L]P$ lies in the identity component, which is a circle; its positive multiples are dense there. Hence the indicated odd progression is dense in the bounded component. The arc with $z\in I$ and $\operatorname{sgn}(y)=\varepsilon_0$ is a nonempty real open set on that component. At its rational points the three square roots are rational by the Kummer class; choose $u_1,u_s>0$ and $\operatorname{sgn}(u_0)=\varepsilon_0$, making $y=du_1u_0u_s$. These sign choices do not affect the local conditions. The cover is smooth of genus one after completion: its function field over $\Q(z)$ is obtained by three independent quadratic extensions with branch points $0,1,-3,\infty$, of degree eight and inertia order two, and Riemann--Hurwitz gives genus one. Its map to $E_d$ is finite. A nonzero rational function has only finitely many zeros and poles on this curve; their images form a finite set to delete from the infinite progression. This completes the proof.
\end{proof}

\section{Every height contact is detected}\label{sec:detection}

At a positive odd-depth contact of $t_0$ in $\mathcal T_*$ the raw source $R_0$ has odd valuation. At a nonzero such contact of $t_1$, the raw source $R'_1$ is odd because the square scale and the twelfth-power normalizer have even valuation. At the omitted zero contact of $t_1$, write $\alpha=v_p(A)\in2\Z$ and $e=v_p(\eta)$. Then $\alpha+e>0$ is odd, so $e$ is odd and in particular nonzero. If $e>0$, $\eta$ is at the zero contact of the first slope and the other ten factors and the normalizer are units; $v_p(R_0)=e$. If $e<0$, all eleven numerator factors have valuation $e$ and the twelfth-power denominator has valuation $12e$, hence $v_p(R_0)=-e$. In both cases $R_0$ has odd valuation, the gate $g_L(t_0)$ is a unit (at the zero contact or at the pole, respectively), and the normalized tests at $t_1$ have unit normalizer. The first source contains those tests as well, so the recovery of the Lucas or first-group comparison uses this source, not an absent zero factor of $R'_1$. At every genuine contact of the first slope, or genuine nonzero contact of the second slope, the gate for its own group is a unit, by the coefficient exclusions; the other group's gate has positive valuation. The appropriate normalized list at the actual slope has unit normalizer. All statements concern primes outside the fixed height-exception set.

Write $\mathcal T_L^0=\mathcal T$ and $\mathcal T_N^0=\mathcal T_N$ for the two genuine five-contact sets, so that the auxiliary eleventh contact is not counted in a height radical. Proposition~\ref{prop:merger-sound} and Lemma~\ref{lem:local-comparison38} imply, for each $G=L,N$ with $\widehat\delta_G\ne0$,
\begin{equation}\label{eq:contact-divisibility}
 M_{\mathcal T_G^0,S_{\rm ht}}(t_j)^{k_\ast}\mid |\widehat\delta_G|
 \qquad(j=0,1).
\end{equation}
This uses no primality assumption in the soundness direction. In the omitted zero-contact case, the normalized tests for slope $t_1$ belong to the first source as well; the unit normalizer at $t_1$ recovers their unnormalized valuations even if $t_0$ is polar.

\section{The height bound and finite exclusion}\label{sec:reduction}
First we bound the height of the free parameter $w_*$. At a prime outside $S_{\rm ht}$ an odd contact of $\beta=2Q_*^2+2402$ with $\mathscr C$ cannot be the infinite contact, since a pole of $Q_*$ gives $v_p(\beta)=2v_p(Q_*)$. At each odd finite contact exactly one of the four contact factors of $r_0$ in \eqref{eq:scalar-input38} has positive odd valuation, and $2\beta-1$ is a unit. At such a contact $Q_*$ is a unit: a positive valuation would give $\beta\equiv2402\pmod p$, excluded by the core. Hence the extra factor $Q_*$ in $r$ does not change the odd valuation, and $r$ has odd valuation by \eqref{eq:rdef}. Proposition~\ref{prop:merger-sound} forces all scalar tests at one common place. By Lemma~\ref{lem:local-comparison38},
\[
 M_{\mathscr C,S_{\rm ht}}(\beta)^{48}\mid x.
\]
As $x\ne0$, Lemma~\ref{lem:fivecontact} gives $h(\beta)\ll\sqrt{|x|}$. Using $Q_*^2=(\beta-2402)/2$ and \eqref{eq:Qstar}, the elementary height inequalities imply
\begin{equation}\label{eq:scalarheight38}
 h(w_*)\ll1+\sqrt{|x|},\qquad h(\rho)\ll1+\sqrt{|x|},
\end{equation}
using \eqref{eq:rheight38} for the second inequality. The constants here depend only on the fixed data, including the finite multiplier branches. No integrality or sign assumption on $w_*$ is used.

\begin{proposition}\label{prop:height}
For each fixed input $a$, the integer indices of every certificate in which at least one source discrepancy is nonzero are bounded by a constant depending only on $a$ and the fixed arithmetic data.
\end{proposition}
\begin{proof}
For either nonzero $\widehat\delta_G$, \eqref{eq:contact-divisibility} and Lemma~\ref{lem:fivecontact} give
\begin{equation}\label{eq:directheight}
 h(t_j)\ll |\widehat\delta_G|^{24/k_\ast}\qquad(j=0,1).
\end{equation}
Since $A=t_1/t_0$, $h(A)\le h(t_0)+h(t_1)$. Combining \eqref{eq:directheight}, \eqref{eq:scaleheight}, \eqref{eq:aggregateheight} and \eqref{eq:scalarheight38} yields
\[
 B_0^2\ll 1+|\widehat\delta_G|^{24/k_\ast}+\sqrt{|x|}.
\]
If $\delta_L\ne0$, its polynomial degree bound gives $|\delta_L|\ll_a B_0^{d_\ast}$. Thus $B_0^2\ll_a1+B_0+\sqrt{B_0}$, which bounds $B_0$.

Suppose now $\delta_L=0$ and $\delta_N\ne0$. Lemma~\ref{lem:sourcegrowth} gives $|x|\ll_a1+\log B_0$. From \eqref{eq:scalarheight38} we have $h(w_*)\ll_a 1+\sqrt{\log B_0}$. For $w_*=c/d$ in lowest terms, the homogenized second discrepancy has the bound
\[
 |\widehat\delta_N|=|d^e F_{2,a}(\mathbf n,w_*)|
 \ll_a B_0^{d_\ast}\exp(eh(w_*)).
\]
Indeed every term after multiplication by $d^e$ has the form of a coefficient times a monomial in the indices of degree at most $d_\ast$, times $c^i d^{e-i}$ with $0\le i\le e$. Consequently, for sufficiently large $B_0$, the right side is $O_a(B_0^{d_\ast+ d_\ast/2})$ (since $\sqrt{\log B_0}=o(\log B_0)$). Our choice $k_\ast=24d_\ast$ then gives
\[
 B_0^2\ll_a1+B_0^{3/2}+\sqrt{1+\log B_0}.
\]
Again $B_0$ is bounded. The bound in both cases is independent of the cutoff and the multiplier branch; the latter range over a fixed finite set. We do not need an effective value of the bound.
\end{proof}

\begin{lemma}\label{lem:exact-source38}
If a rational certificate satisfies both exact source equations, then $a\in\mathcal A$.
\end{lemma}
\begin{proof}
Lemma~\ref{lem:rational-positivity} gives $w_*\in\Z\setminus\{0,1\}$. Thus $Q_*$ is an integer different from $0,1$ and $Q_*(Q_*-1)>0$. Formula \eqref{eq:rdef} gives $\rho>0$ independently of the sign of $r$, so \eqref{eq:Arho} yields $A\ge25$. By Lemma~\ref{lem:qplus}, $d_+>0$.

We show that $(-d_+,Q_*)_p=1$ for every finite prime, where $(\ ,\ )_p$ denotes the quadratic Hilbert symbol, so that $(-d_+,Q_*)_p=1$ is equivalent to representation of $Q_*$ by $U^2+d_+V^2$ over $\Q_p$. For $p\in S_{\rm ht}$, $Q_*$ is a local square by \eqref{eq:Qstar}. Otherwise $p$ is odd. If $v_p(Q_*)$ is odd, it is positive; $r_0$ is a unit at $p$ by the exclusions at $2402$, and $r$ is odd there. Proposition~\ref{prop:merger-sound}, applied to the added scalar test, gives $v_p(D_+)\ge1$ (the test is rational). Hence $Q_*$ is a local norm by \eqref{eq:localnorm38}. If $v_p(d_+)$ is odd, then $\eta$ has a positive odd contact depth at $\theta_7$ or $\theta_{10}$. Indeed at a pole the two numerator factors and the square denominator have equal total valuation, at the normalizer the valuation is even, and elsewhere the coefficient exclusions leave at most one nonunit factor. Thus $R_0$ is odd, $g_+(\eta)$ and $\eta-3$ are units, and the added source-$0$ test gives $v_p(D_+)\ge1$. This also covers simultaneous odd sources by the same merger proposition. At every remaining odd prime, both $v_p(Q_*)$ and $v_p(d_+)$ are even. Divide $Q_*$ and $d_+$ by local squares to obtain units $q,u\in\Z_p^\times$. If $-u$ is a square, the binary form $X^2+uY^2$ is split and represents every element of $\Q_p$; otherwise it is the norm form of the unramified quadratic extension of $\Q_p$, whose norm map is onto the unit group by the residue-field norm and Hensel lifting. Thus $q$ is represented, and restoring the square factors represents $Q_*$. This proves the local claims without assuming that the certificate comes from the affirmative prime-pattern construction.

By Hilbert reciprocity (see \cite{CasselsQF,Voight}), $(-d_+,Q_*)_\infty=1$. Since $d_+>0$, we conclude $Q_*>0$. By \eqref{eq:Qstar}, $w_*>0$, so $w_*>1$. Lemma~\ref{lem:rational-positivity} then gives $O_0>0$. All hypotheses of Lemma~\ref{lem:split-source} hold, and $a\in\mathcal A$. No sign conclusion from this argument was used in the preceding height bound.
\end{proof}

For a positive integer cutoff $T$, append to the certificate of Section~\ref{sec:carrier} the exclusions
\begin{equation}\label{eq:forbiddenset}
 \mathcal F_T=[-T,T]^8\cap\Z^8.
\end{equation}
For each $\mathbf b\in\mathcal F_T$ impose
\begin{equation}\label{eq:guard}
 (P_1,\ldots,P_8)\ne([b_1]P,\ldots,[b_8]P).
\end{equation}
Unlike the all-integer split source, the present source involves $w_*$, so we exclude the entire box, not just tuples failing an equation evaluated on integer indices. These point-tuple inequations are finite Boolean combinations of equations and inequations on the affine and identity charts of $E$. The points on the right are computable fixed points for the query. In the affirmative case Lemma~\ref{lem:split-source} permits $|x|>T$; in the negative case Proposition~\ref{prop:height} bounds all possible certificates. No unknown height bound is used to construct an individual query.

\section{Completeness with seven shared local points}\label{sec:completeness}

\begin{proposition}\label{prop:completeness38}
If $a\in\mathcal A$, then for every cutoff $T$ the certificate has a rational realization satisfying \eqref{eq:guard}.
\end{proposition}
\begin{proof}
Fix a cutoff $T$. Choose the five coding indices in the affirmative construction of Lemma~\ref{lem:split-source} and the positive integer square root $r_M$ of $D_{\rm size}$. Take $\varepsilon_+=1$ and fix $w_*=O_0(r_M+1)^2>1$; the second equation is true. Normalize $r$ and $\rho$ using Section~\ref{sec:fixed}. For $q\in S_{\rm core}\cup\mathcal B$ put $n_s(q)=v_q(r)$. Require the scalar tests of tag $k$ to have depth $k(n_s(q)+1)$ at core primes and $k(n_s(q)+v_q(9+\Omega))$ at primes of $\mathcal B$, as required by Lemma~\ref{lem:coreprep} and the singleton normalization \eqref{eq:Jdepth}.

Choose a positive integer $K_{\rm sc}$ divisible by $64$ and by every $q\in S_{\rm core}\cup\mathcal B$, whose multiple of $P$ enters a sufficiently deep isometric formal logarithm subgroup at every place above these primes. Such an integer exists by compactness of the local elliptic groups and the existence of an open formal subgroup. For each such $q$ fix an integer $a_q\ge v_q(K_{\rm sc})$ sufficiently large that under the compatible conditions
\[
 h_0\equiv0\pmod{K_{\rm sc}},\qquad h_0\equiv q^{a_q}\pmod{q^{a_q+1}},
\]
both components of $\rho_E(H)$ have the required scalar depths. Enlarge these depths also so that, uniformly for every integer $j$, at each place $\mathfrak q\mid q$,
\begin{equation}\label{eq:uniform-ratio38}
 \frac{z_E([j]H)}{z_E(H)}-j\in q^{m_q+v_q(2)}\OO_{F_{\mathfrak q}},
 \quad m_q=n_s(q)+b_s(q)+v_q(Q_*)+2,
\end{equation}
where $b_s(q)=1$ at the core and $v_q(9+\Omega)$ on $\mathcal B$. This is possible before the choice of $x$: on a deep formal logarithm ball, $[j](t)=\ell^{-1}(j\ell(t))$ and both power series start with $t$. Their higher terms after division by $t$ tend uniformly to zero as $t\to0$, for all $j\in\Z_q$, by convergence on a smaller closed ball. Thus the required depth depends only on the already fixed scalar data. The extra $v_q(2)$ also suffices for taking the half-trace. These congruences describe a nonempty arithmetic progression: at each specified $q$ the second congruence is compatible with $h_0\equiv0\pmod{q^{v_q(K_{\rm sc})}}$ because $a_q\ge v_q(K_{\rm sc})$, and at every other prime dividing $K_{\rm sc}$ only the first congruence is imposed. The Chinese remainder theorem applies. In the isometric formal subgroup, put $\gamma_{q,\mathfrak q}=v_{\mathfrak q}(z_E([K_{\rm sc}]P))-v_q(K_{\rm sc})$. By the formal logarithm, for every nonzero integer $j$ divisible by $K_{\rm sc}$ one has $v_{\mathfrak q}(z_E([j]P))=v_q(j)+\gamma_{q,\mathfrak q}$. This proves the claimed identity for $h_0$ and $x$, including when $q=2$ (the subgroup was chosen deep enough for the logarithm to be isometric). To get componentwise rational bounds even at the ramified prime $2$, it suffices to require membership of the field element in $q^d(\mathcal O_F\otimes\Z_q)$ for the target depth $d$; the basis $1,\sqrt2$ is an integral basis. Now impose on $x$ divisibility by $K_{\rm sc}$ and sufficiently high powers of all these primes. This is a single fixed positive integer divisibility condition, independent of the eventual choice of $h_0$ on its progression. It makes $\rho_E([x]P)$ and $z_E([x]P)/z_E(H)$ satisfy the remaining scalar tests uniformly on that progression. Indeed, $1/x_E$ has a double zero at the identity, and the valuation of the quotient is exactly $v_q(x)-a_q$ at each place. In particular choose $v_q(x)\ge a_q+48(n_s(q)+b_s(q))$, where $b_s(q)=1$ at the core and $b_s(q)=v_q(9+\Omega)$ on $\mathcal B$, increasing the depths if necessary for the first two tests. All these requirements depend on the already fixed $w_*$ and on fixed arithmetic data, not on the contact primes or the source tuple yet to be completed. Lemma~\ref{lem:split-source} gives arbitrarily large such $x>T$ with $y_0,m_1$ satisfying the first equation.

With $w_*,r,\rho$ fixed, the trace heights $h(R_I)$ are unbounded by \eqref{eq:aggregateheight}. Apply Lemma~\ref{lem:three-shift} with the finite squareclass space generated by $[-1]$ and the primes in $S_{\rm ht}\cup\Supp(r)\cup\Supp(Q_*)$ to choose one of these tuples and a shift giving \eqref{eq:crossindependence}. Choose the private primes, denominator primes, physical graph and coefficient clearing integer $\Lambda$ as in Section~\ref{sec:realize}; all this precedes the choice of contact primes. The scalar groups need not be singletons. Use the additional characters of Theorem~\ref{thm:uniform-selmer} on $\mathcal P_Q$. Fix the residues $\eta_p$ as in Section~\ref{sec:realize} now, before preparing the reference indices. For $p\in\mathcal B$, choose $e_p$ so large that $d_+(\eta)$ differs from $d_+(\eta_p)$ by an element of $p^{m_p}\Z_p$ whenever \eqref{eq:scalar-slope38} holds. This is possible because the normalizer is a unit. No contact prime is selected at this stage. The auxiliary denominator vertices and their prime prescriptions are those of Theorem~\ref{thm:uniform-selmer}. At this stage every scalar-support prime and every prime of $\Supp(Q_*)$ is fixed. Neither the governing Gaussian ray modulus nor the scalar support is allowed to grow with the parameters in the prime-pattern limit.

We next prepare uniform lower bounds for the ordinary raw tests at the core and denominator primes. Choose a reference integer $h_*$ in the scalar progression, with $h_*\ne0,64$, off the finitely many zeros of the two truncated-logarithm denominators in \eqref{eq:Li}.

Such a choice exists: each denominator is a nonzero rational function of the anchor on the corresponding elliptic curve, and it cannot vanish at infinitely many distinct multiples of a nontorsion point. For instance the formal parameter has a simple zero at the origin, and the truncated logarithm starts with the linear term, so the functions are nonzero. For $j\le5$ choose $z_{*,j}=h_*n_{i(L,j)}$. For $j=6,7$, at each $q\in S_{\rm core}\cup\mathcal B$ prescribe local integer targets $(a_q',b_q')$: at core primes take $(\sqrt{Q_*},0)$, at primes of $\mathcal B$ use the pair in Lemma~\ref{prop:binarynorm38}(1). By the Chinese remainder theorem choose integers $c_6,c_7$ congruent to these respective targets modulo $q^{m_q}$ at all these primes, and put $z_{*,j}=h_*c_j$ for $j=6,7$. All reference indices are multiples of $64$. By \eqref{eq:uniform-ratio38}, their rational half-traces $S_6,S_7$ approximate the targets to precision $m_q$. The congruence for the slope at $q\in\mathcal B$ therefore makes $v_q(D_+)\ge n_s(q)+b_s(q)$. At core primes, for all sufficiently polar $\eta$, $d_+(\eta)$ is integral (it tends to one); $S_7$ approximates zero and $S_6^2$ approximates $Q_*$. The same required scalar bound follows there, uniformly for such polar slopes. All the raw tests, in both groups, have finite evaluations at this reference tuple; the functions $z_E,\rho_E$ are regular at the identity and no nonidentity cyclic point has zero ordinate or abscissa.

At every prime $q$ in the core, in $\mathcal B$, or dividing the denominator $D$ of \eqref{eq:D}, continuity at the reference tuple gives finite lower bounds for all raw test valuations at all places over $q$, uniformly in a product neighborhood of the reference points. The map from integer indices to points is continuous for a sufficiently divisible index modulus: a fixed positive multiple of $P$ lies in any given open formal subgroup over each of the finitely many local fields. Hence there is a finite integer modulus $N_{\rm loc}$, divisible by $64$, $K_{\rm sc}$ and all the $q^{a_q+1}$ from the scalar progression, such that the lower bounds hold for
\begin{equation}\label{eq:referenceindices}
 h_0\equiv h_*\pmod{N_{\rm loc}},\qquad
 z_j\equiv z_{*,j}\pmod{N_{\rm loc}}\quad(1\le j\le7).
\end{equation}
Increase the reference modulus so that the two half-traces retain precision $m_q$ at all core and scalar-support primes, and have uniform finite lower bounds at all denominator primes. This uses continuity of the rational half-trace at the nonzero reference anchor; the possibly varying slope has not entered these two trace functions. In particular the logarithm denominators stay nonzero in these neighborhoods. The odd prime divisors of $N_{\rm loc}$ are added to the order-avoidance congruences of Section~\ref{sec:realize}, before the counting parameters are chosen. They are not added to the fixed core.

Choose the polar depths in the base system of Section~\ref{sec:realize} as follows. At a core or denominator prime put $\nu=-v_q(\eta)$, $\alpha=v_q(A)$ and $\alpha_0=0,\alpha_1=\alpha$. For $\nu$ sufficiently large, comparison with the fixed contact coefficients gives the exact values
\begin{align*}
 v_q(R_0)&=\nu,& v_q(5R_1')&=2\nu-12\alpha+v_q(5),\\
 v_{\mathfrak q}(\widetilde D_{j,i})&=v_{\mathfrak q}(D_i)+(2k_i+1)(\nu-\alpha_j).
\end{align*}
For each $j$ the last expression exceeds $k_i(v_q(R_0)+1)$ and $k_i(v_q(5R_1')+1)$ when $\nu$ is large, because its coefficient of $\nu$ is $2k_i+1$. Both group gates and $g_+$ are units for these polar slopes at sufficiently large depth. As $d_+(\eta)$ is then integral, the reference lower bounds on the half-traces and the fixed nonzero $Q_*$ give a finite lower bound on $v_q(D_+)$ independent of further increases in the polar depth. Consequently $D_+/(\eta-3)^3$ satisfies the same required ordinary bound for its tag one at source $0$, after increasing that depth. At core primes the added scalar test already has its required bound by the reference preparation; the denominator primes outside the core do not require an active scalar test. At core primes the scalar tests already prepared on the anchor progression and these ordinary bounds give exactly the conditions of Lemma~\ref{lem:coreprep} with $b_q=1$. At denominator primes outside the core only $R_0$ has odd valuation; its required singleton depth is $k_i\nu$ because $1+\Omega$ is a unit there. The required bounds on the unnormalized $D_i$ were fixed in advance, so one can impose the necessary large divisor of $w$ at all these primes in the finite CRT system without depending on the eventual carrier index.

Apply the prime-pattern construction of Section~\ref{sec:realize}, with the indicated congruences, the required polar depths and the order-avoidance primes. Choose its counting parameters so large that $\eta>3$ and the twenty-one selected primes avoid all finite sets required there, $\Supp(r)$, the supports of $N_{\rm loc}$ and the prime divisors of $1+\Omega$ and $9+\Omega$. Also avoid $\Supp(Q_*)$ and the finite set of rational primes at which either abscissa or ordinate of $[64]P$, at either conjugate place, is not a local unit. These coordinates are nonzero because $P$ is nontorsion. This makes $z_E([64]P)$ a unit at both places, and its reduction is given by the regular quotient $-x_E/y_E$. Both normalizers are units at every selected contact prime. Their residue conditions give $p\equiv17\pmod{32}$ and $(7/p)=-1$. In particular they split in $F$, and the two reduction cardinalities are $d^+(p),d^-(p)$ by \eqref{eq:orders} and \eqref{an:eq:orders}. The order-avoidance theorem gives pairwise coprimality of the odd parts of all distinct-position orders, including mixed signs. Let
\[
 L_+=\lcm_p d^+(p),\qquad L_-=\lcm_p d^-(p),
\]
where $p$ runs through the twenty-one contacts. The fixed dyadic ray gives $v_2(L_+)=2$, $v_2(L_-)=6$, and $\gcd(L_+,L_-)=4$. The odd parts of these orders avoid $N_{\rm loc}$. The generalized CRT therefore supplies a nonconstant progression of $h_0$ satisfying the first congruence of \eqref{eq:referenceindices} and
\begin{equation}\label{eq:anchorcrt}
 h_0\equiv0\pmod{L_+},\qquad h_0\equiv64\pmod{L_-}.
\end{equation}
At two the residues are all zero modulo the intersections because $64\mid h_*$. At the two special primes $I_7,I_{10}$ also require $h_0\equiv0\pmod p$. The selected primes avoid the fixed reference modulus, and none divides any selected order: for sufficiently large patterns all these even orders are less than twice the smallest selected prime, so divisibility by a selected odd prime would force equality to that odd prime. These extra congruences are therefore compatible. Choose an $h_0$ off the finitely many anchor guards and put $H=[h_0]P$. At each selected $p$, take the prime $\mathcal P$ of $F$ at which $E$ has order $d^+(p)$; its conjugate has order $d^-(p)$. Then $H$ and $V_0^\sigma=(H-[64]P)^\sigma$ both lie in the respective formal kernels at $\mathcal P$.

For each contact prime $p=I_i$ or $J_i$ let $G(p)$ be the group of its associated genuine contact; at the auxiliary contact one may set $G(p)=L$ just for unused index choices. At a group-$N$ prime, set $i(N,j)=i(L,j)$ for unused slots $j=6,7$. Except at $p=I_7,I_{10}$ in those two slots, impose
\begin{equation}\label{eq:shared-CRT}
 z_j\equiv h_0n_{i(G(p),j)}\pmod{\lcm(d^+(p),d^-(p))p^{b_p}},
 \qquad z_j\equiv z_{*,j}\pmod{N_{\rm loc}}.
\end{equation}
At $p=I_7,I_{10}$, Lemma~\ref{prop:binarynorm38}(3) gives a local unit square root $a$ of $Q_*$. For the last two slots choose integers $c_{p,6},c_{p,7}$ congruent respectively to $2a,0$ modulo $p$, and replace the corresponding order congruence by
\begin{equation}\label{eq:norm-contact-CRT38}
 \begin{aligned}
 z_j&\equiv h_0 c_{p,j}\pmod{d^+(p)p^{b_p}},\\
 z_j&\equiv0\pmod{d^-(p)} \qquad(j=6,7).
 \end{aligned}
\end{equation}
The odd parts of all distinct orders, including the two at a given contact, are coprime and avoid $N_{\rm loc}$. The selected primes are distinct, avoid $N_{\rm loc}$ and, as just verified, divide none of the selected orders. At two all right-hand sides are zero modulo $64$ wherever needed, since $h_0,h_*,z_{*,j}$ are multiples of $64$. Thus all the congruences, including arbitrary positive exponents $b_p$, are compatible by the generalized CRT.

Set $U_j=[z_j]P$. For each active ordinary comparison slot, the congruences \eqref{eq:shared-CRT} give proximity at both places to $[n_{i(G(p),j)}]H$, and at the selected place the translated conjugate point is close to $[n_{i(G(p),j)}]V_0^\sigma$. Increasing $b_p$ increases the formal depth of the differences, since multiplication by $p$ increases depth in the formal group. The exact multiples satisfy the tag-one tests and the truncated logarithm comparisons of Lemma~\ref{lem:formal}, and the active polynomial vanishes at the true integer tuple (with the fixed integral $w_*$ for group $N$). The denominators are nonzero at the chosen anchor. Continuity therefore gives the required tagged precision for all active ordinary comparisons, for both normalized slopes, whose normalizers are units.

At a special prime $p=I_7,I_{10}$ take the chosen place $\mathcal P$ of order $d^+(p)$. Since $h_0$ is divisible by $d^+(p)p$, the anchor $H$ lies in the positive-depth formal kernel at this place. The first congruence of \eqref{eq:norm-contact-CRT38} makes $U_j$ arbitrarily close to $[c_{p,j}]H$ there. Choose $b_p$ greater than the finite formal depth of $H$ plus the precision required for the quotient. Then the formal group law and $[c](t)=ct+O(t^2)$ imply
\[
 \frac{z_E(U_j)}{z_E(H)}\equiv c_{p,j}\pmod{\mathcal P}.
\]
Here the formal multiplication power series satisfies $[c](t)=ct+O(t^2)$ with integral coefficients for integer $c$ at good reduction primes; the anchor is in the formal kernel. At the conjugate place the second congruence places $U_j$ in the formal kernel, whereas $H$ reduces to $[64]P$. The selected prime avoids the finite set where $z_E([64]P)$ fails to be a unit, so $z_E(H)$ is a unit there and the conjugate quotient has positive valuation. At a split rational prime, the trace is the sum of these two local embeddings. Since $p$ is odd and $c_{p,6}\equiv2a$, $c_{p,7}\equiv0\pmod p$, division by two gives rational half-traces congruent to $a,0$ modulo $p$ respectively. Since $d_+(\eta)$ is integral at these contacts and $Q_*$ is a unit, this gives $v_p(D_+)\ge1$, the required norm-test precision also after normalization by the unit $\eta-3$. At $I_7$ group $N$ uses only the first five slots, and at $I_{10}$ neither ordinary group is active, so there is no competing requirement on slots $6,7$. At all other contact positions the norm test at source $0$, if it belongs to the odd source, has a gate of positive valuation and is suppressed. At core, scalar-support and denominator primes the reference and slope congruences ensure the bounds prepared above.

The factorizations \eqref{eq:R0factor} and \eqref{eq:R1factor} show that outside the core every nonempty odd-source pattern is a singleton: a scalar prime in $\mathcal B$, a denominator prime, or one of the twenty-one contact primes. All these raw odd depths are positive; the contact depths equal one, the scalar depths are the positive odd integers $n_s(p)$, and the denominator depths are the chosen positive odd valuations of $R_0$. At a genuine ordinary selected contact its own comparison group's gate is a unit and the other group's gate has positive valuation. At the auxiliary contact both comparison gates have positive valuation; the additional norm gate is a unit precisely at $I_7,I_{10}$ among contacts of source $0$. It has positive valuation at every other contact of that source. At source $1$ there is no additional norm test. The exponent $2k+1$ in \eqref{eq:clip-gate}, together with the universal lower clip bound $-1$ at depth one, makes the inactive tests automatic in the singleton normalization. The other source's gates are nonnegative there, since the normalizers at both slopes are units. At denominator and scalar primes all ordinary group gates are units by the polar preparation or the residue avoidance in Section~\ref{sec:realize}. At scalar primes the height tests have the required depths $k(n_s(p)+v_p(9+\Omega))$ by the scalar progression and the choice of $x$, and the norm test has the required depth by \eqref{eq:scalar-slope38} and the reference half-trace preparation. A variable scalar-support prime may divide $9+\Omega$; this is allowed, since its extra depth was included before choosing $x$. No such divisibility is silently excluded. The denominator and contact primes, in contrast, were chosen to avoid $1+\Omega$ and $9+\Omega$. Thus all hypotheses of the singleton normalization following \eqref{eq:Jdepth} and the core lemma are satisfied.

Corollary~\ref{cor:rational-cover} now gives a nontorsion representative in the required Kummer class; all primes and congruences used to prove it have been selected. With the tests and anchors fixed, the real construction in Section~\ref{sec:merger} supplies a nonempty open arc on the covering with $0<112B_\Phi<1$. By Lemma~\ref{lem:movement}, choose a rational covering point sufficiently close to each singleton torsion limit and on the indicated real arc, avoiding the finite rational-function guards. All such guards can be avoided: on a subarc chosen in the real preparation all three source roots and denominators are nonzero, $J>0$, $U\ne0$, the correction factors are positive at both real embeddings, and $\Theta>0$. The guards depending only on the carrier or slope were already ensured. Thus none of the remaining guard functions is identically zero on the covering.

At each exterior odd-source prime the singleton normalization gives $v_p(J)>0$ and $v_p(\Theta)=0$. At any other exterior prime $v_p(B_\Phi)$ is even, by its squareclass. At the fixed core Lemma~\ref{lem:coreprep} gives $v_q(B_\Phi)\ge1$. Consequently $b=112B_\Phi$ has a square denominator, $0<b<1$ and $v_2(b)\ge5$. The squarefree representative of $[B_\Phi]=[ra_0a_1]$ is $d_{\rm sf}\equiv17\pmod{24}$ from Theorem~\ref{thm:uniform-selmer}, and multiplying by $112$ changes the class by $7$. Since $d_{\rm sf}$ is odd, removing a common odd square does not affect the residue modulo eight. Hence the positive squarefree representative of $b$ is $7\pmod8$. The BB lemma as stated before \eqref{eq:BBconditions} now supplies a rational solution of \eqref{eq:BB}. This realizes all conditions of the certificate; the norm tests have already been realized by the same seven local points.

The root tuple has $x>T$, so it is not in the excluded box: equality of all corresponding points would imply equality of all eight indices by nontorsion of $P$. Thus the required cutoff exclusions are satisfied. This proves the proposition.
\end{proof}

\begin{proposition}\label{prop:queries}
If $a\notin\mathcal A$, some finite-cutoff query has no rational certificate.
\end{proposition}
\begin{proof}
Every certificate satisfies the three nonidentity guards and $w_*\ne0,1$. By Lemma~\ref{lem:exact-source38}, at least one discrepancy must therefore be nonzero. Proposition~\ref{prop:height} bounds its entire eight-tuple in a fixed box. A sufficiently large cutoff contains that box, and then the corresponding guard \eqref{eq:guard} excludes every such tuple. The bound is uniform over all certificate choices, including support and scale branches. This proves the assertion.
\end{proof}

\section{Algebraic compilation without extra guard variables}\label{sec:encoding}

We now count free rational parameters and justify elimination of every remaining finite algebraic choice. All constructions in this section are finite algorithms on polynomials. Their degrees and branch counts can be very large, and may depend on $a$ and $T$.

All equations over $F$ in the preceding sections must first be expressed over $\Q$. Write each computed element as $\alpha=a_0+a_1\sqrt2$, where $a_0,a_1$ are rational functions in the rational parameters and finite root choices. Then
\[
 \alpha=0\ \Longleftrightarrow\ a_0=a_1=0,
 \qquad
 \alpha\ne0\ \Longleftrightarrow\ a_0^2-2a_1^2\ne0.
\]
Under this last guard,
\[
 \alpha^{-1}=\frac{a_0-a_1\sqrt2}{a_0^2-2a_1^2}.
\]
Thus equality, inversion, conjugation and every denominator guard in $F$ become rational expressions, equations and inequations without adding variables. Fixed $F$-coefficients are split into their rational components in the same way. Numerical rational coefficients will be cleared at the final step.

\subsection{Two parameters per cyclic point}
Represent each of the sixteen freely specified cyclic points by an affine preimage $(x,y)\in E(F)$, writing $x=u+v\sqrt2$ with $u,v\in\Q$ (here $v$ is a local abscissa coordinate). For all sixteen points, apply $[m]$ to the chosen affine preimage.  There are eight root points, seven shared local points, and one anchor, all encoded in the same way.  The identity in the cyclic subgroup is represented by the affine preimage $(0,0)$ when needed, since $m$ is even. Since $d_0^2=11243-7950\sqrt2$, put
\begin{align}
 G_0&=u^3+6uv^2-11243u+15900v,\label{eq:ordinatezero}\\
 G_1&=3u^2v+2v^3-11243v+7950u.\label{eq:ordinateone}
\end{align}
The equation for $y$ is $y^2=G_0+G_1\sqrt2$. It is encoded by the following finite branches.
\begin{enumerate}
\item If $G_1\ne0$, take a rational root $t$ of
$t^4-G_0t^2+G_1^2/2=0$, and set
 $y=t+G_1/(\sqrt2\,t).$
The root is nonzero, and the quartic is separable: a repeated root would force $G_0^2=2G_1^2$, impossible for rational $G_0$ and $G_1$ with $G_1\ne0$.
\item If $G_1=0$ and $G_0\ne0$, use either $t^2-G_0=0$ and $y=t$, or $t^2-G_0/2=0$ and $y=t\sqrt2$. Both polynomials are separable.
\item If $G_0=G_1=0$, use $t=0$ and $y=0$.
\end{enumerate}
These branches are exhaustive, by comparing the two components of $(a+b\sqrt2)^2$. The symbols $G_0,G_1$ in this subsection denote the ordinate radicands of a single preimage. On each support branch, the global trace sum $R_I$ is computed from the nonidentity root abscissas as in \eqref{eq:tracecoords}; explicitly $x(\tau^{-1}P_i)=\lambda^{-2}x(P_i)$. Thus each point has two free parameters and one finite root choice. Addition, fixed multiplication by $m$ or $64$, conjugation, and the origin cases use finitely many rational charts. On each chart all computed coordinates are rational expressions in the parameters and finite roots, with every denominator recorded as nonzero. For addition of points with distinct abscissas the slope is $(y(W)-y(V))/(x(W)-x(V))$, and for doubling with nonzero ordinate it is $(3x(V)^2-d_0^2)/(2y(V))$. In both cases $x(V+W)$ is the square of the slope minus the two abscissas and $y(V+W)$ is the slope times $x(V)-x(V+W)$ minus $y(V)$. Include the identity cases and equal-abscissa opposite-ordinate cases (including zero-ordinate doubling). These give exhaustive finite rational charts. 

\subsection{No source-coordinate parameters or guard parameters}
The input $a$ and the two specialized source polynomials are fixed data for a query. Each local evaluation of $f_{G,a}$ uses the $s_G$ local quotients and, for group $N$, the free rational parameter $w_*$. These are computed expressions in the point coordinates, their finite ordinate choices, and $w_*$. The norm tests are computed rational functions of the existing coordinates and the same parameter. No separate source-coordinate tuple is quantified.

The three nonidentity source guards, the guards $w_*\ne0,1$, $Q_*\ne0$, the finite multiplier choices, the support branches of Section~\ref{sec:carrier}, and the tuple guards \eqref{eq:guard} are finite Boolean combinations in these same computed coordinates. Distributing finite conjunctions and disjunctions gives finitely many branches, each specified by equations and inequations. Their number is unrestricted; all branches use the same free parameters. The construction below filters the equations and inequations inside a finite algebra, and uses one final rational variable for all remaining finite choices and nonzero guards.

\subsection{One coordinate for the first finite choices}
The finite roots just described satisfy independent monic separable polynomials
\[
 p_i(t_i)=0\quad(1\le i\le N)
\]
over the free rational parameters, on each branch. At a fixed rational specialization let $D=\prod_i\deg p_i$. The algebra
\[
 A=\Q[t_1,\ldots,t_N]/(p_1(t_1),\ldots,p_N(t_N))
\]
is reduced and has $D$ complex points. For an integer $b$ take the characteristic polynomial $R_b(Z)$ of multiplication by
\[
 z_b=t_1+bt_2+\cdots+b^{N-1}t_N.
\]
Two distinct tuples collide for at most $N-1$ values of $b$. Thus some
\begin{equation}\label{eq:separatingrange}
 1\le b\le1+(N-1)\f{D(D-1)}2
\end{equation}
separates all points. It is enough to allow this finite list of constants and require $\disc R_b\ne0$. The polynomials $R_b$ are monic of degree $D$.

A rational root of such an $R_b$ recovers a unique tuple. Every automorphism of $\overline\Q/\Q$ fixes the root and permutes the tuples, so it fixes every component of that tuple. Hence all the recovered $t_i$ are rational. Conversely, a rational tuple gives a rational root for any separating $b$. For explicit recovery, the powers $1,z_b,\ldots,z_b^{D-1}$ form a basis of $A$: after extension to $\C$ their evaluation matrix is a Vandermonde matrix on the distinct primitive values. Invert this matrix in any original monomial basis to express each $t_i$ as a polynomial in $z_b$. Matrix inversion records only rational-function coefficients and nonzero pivots. We may now write the first algebraic witness as $Z$, its monic separable polynomial as $R(Z)$, and every computed quantity as a rational function of the free parameters and $Z$.

\subsection{Adding the three finite roots}
The cover and terminal equations add three monic
quadratic relations
\[
\begin{gathered}
 u_0^2-a_0(z_c-1),\quad u_s^2-r(z_c+3),\\
 z_p^2-(1-b)(s_p^4-b^2)(t_p^4-b^2).
\end{gathered}
\]
The radicands may depend on the preceding finite roots. Thus
these extensions are formed \emph{successively} over the finite
algebra from the elliptic ordinates, not as a tensor product of
independent polynomials over the field of free coordinates.
At each rational specialization, first discard the factors on which any previously required denominator vanishes; this is localization in the finite reduced algebra. Split the remaining algebra by the idempotent factors on which a radicand vanishes.
On its vanishing factor adjoin only the root zero; on the
complement adjoin the monic separable quadratic $T^2-f$.
The resulting successive finite algebra remains reduced, as is
seen after extension to an algebraic closure. Its dimension over the specialized rational base is at most $4^{16}2^3=2^{35}$. Continue for all
three radicands, with the guards imposed as constraints on the
corresponding finite components. To justify uniformity in the free coordinates, a finite reduced algebra is a product of fields; localization at an element $g$ is its multiplication image $gA$, with the induced unit of that factor, while the zero factor of a radicand is the kernel of its multiplication map. Both are computed by rational matrix operations after branching on ranks and nonzero minors. The multiplication tables and the induced units are obtained by solving linear systems in the chosen bases. On the surviving factors, inverses of nonzero denominators are computed by the same method. Thus every operation here is represented by rational functions on finitely many algebraic strata with explicit polynomial equation and nonzero guards; no unbounded choice or extra inverse parameter is implicit in the localization.

The primitive-element separation argument of the preceding
subsection applies to this final reduced algebra, of bounded
degree on each actual-degree branch, and supplies one rational
root parameter $Z$. For a rational root $Z$, Galois invariance
of its unique geometric point forces the complete tuple of
elliptic, covering and terminal roots to be rational. Every rational tuple
likewise gives a rational $Z$. The cover coordinate
$z_c=u_1^2/a_1$, the nonzero denominators, and the real-arc
nonvanishing guards are computed expressions in that one root.
No further freely quantified variable is needed.

\subsection{Absorbing every nonzero guard}
We first filter the finite algebra by the equality and nonzero
conditions that may depend on its root coordinates. After a
specialization of the free variables and a separating primitive
element $Z$, every such condition is represented by a polynomial
in $Z$ modulo the squarefree monic characteristic polynomial
$R(Z)$. Replace $R$ by its monic greatest common divisor with
the equality polynomials. Next remove from that divisor the
factors common to each forbidden-zero polynomial; these factors
correspond precisely to geometric points where a nonzero guard
fails. Because $R$ is squarefree, this produces the monic
polynomial whose roots are exactly the admissible geometric
points. The degrees and coefficients of these gcds are obtained
by a finite subresultant computation. Branch on the finitely
many possible actual degrees and vanishing/nonvanishing
subresultants. On each branch all coefficients are rational
functions of the free coordinates, and every remaining branch
guard is a polynomial inequality in the free coordinates alone.
An empty admissible algebra is represented by the polynomial $1$.

For the retained monic polynomial allow either of the shifts $T\mapsto T$ or $T\mapsto T-1$. For every rational root at least one of its two shifted primitive values is nonzero. Thus it suffices to ask for a nonzero rational root of a monic polynomial
\[
 P(T)=T^d+\sum_{j=1}^d\frac{N_j(\mathbf x)}{B(\mathbf x)}T^{d-j},
\]
where $N_j,B\in\Q[\mathbf x]$ and $\mathbf x$ is the tuple of free parameters. Take $B$ to be a common polynomial denominator and let $H$ be the product of the numerators of the remaining branch inequations in the free coordinates (the conditions depending on finite roots have already been removed by gcd filtering). Define
\begin{equation}\label{eq:reciprocal}
 \Psi(\mathbf x,U)=1+\sum_{j=1}^d
 N_j(\mathbf x)B(\mathbf x)^{2j-1}H(\mathbf x)^jU^j.
\end{equation}
If $BH=0$, this is identically 1. If $BH\ne0$, putting $S=B^2H$ gives, for $U\ne0$, the identity
\[
 \Psi(\mathbf x,U)=(SU)^dP(1/(SU)).
\]
Since $\Psi(\mathbf x,0)=1$, its rational roots correspond exactly to nonzero rational roots of $P$. In particular it enforces all inequations without another inverse variable. The exponent $2j-1$ is positive; simply clearing denominators without this factor would not justify the assertion on $B=0$.

If $E_1(\mathbf x)=\cdots=E_r(\mathbf x)=0$ are the polynomial equations of a branch, use
\[
 \Psi(\mathbf x,U)^2+\sum_{i=1}^rE_i(\mathbf x)^2=0.
\]
Over $\Q$ this is the desired conjunction. Multiply these branch polynomials over all finite choices, using the same $U$, to express their disjunction. Clearing numerical rational coefficients gives one integer-coefficient polynomial. An empty disjunction is represented by the constant polynomial 1. There is no selector variable and no unrecorded nonzero condition.

\section{Variable count and proof of Theorem~\ref{thm:main}}\label{sec:mainproof}
The free rational coordinates are counted as follows:
\begin{center}
\begin{tabular}{lr}\toprule
Free rational unknowns & Number\\\midrule
Two components for each of eight root preimages &16\\
Two components for each of seven shared local preimages &14\\
Two components of the anchor preimage &2\\
Common slope $\eta$ &1\\
Covering coordinate $u_1$ &1\\
Terminal parameters $s_p,t_p$ &2\\
Finite-algebra root and guard compiler &1\\
Rational source parameter $w_*$ &1\\\midrule
Total &\textbf{38}\\\bottomrule
\end{tabular}
\end{center}
All elliptic ordinates, the roots $u_0,u_s,z_p$, the support branches, the three scale branches, the finite multiplier branches ($\mu,\nu_*$) and all rational-function guards are handled by the one-variable compiler of Section~\ref{sec:encoding}. The integer indices, prime positions, denominator primes, local depths and Selmer computations occur in the proof, not as additional quantified variables.

\begin{proof}[Proof of Theorem~\ref{thm:main}]
Take the fixed nonrecursive recursively enumerable set $\mathcal A$ of Section~\ref{sec:split-source}. For each input $a\in\N$ and each integer $T\ge1$, construct the guarded certificate and compile it into one polynomial equation with integer coefficients in thirty-eight rational unknowns. The polynomials and fixed arithmetic data are all finite and can be built into the reducing algorithm. In particular the only enumerated quantities in the finite exclusions are the tuples in $[-T,T]^8$ and their explicitly computed points. The ineffective height bound is not part of the algorithm.

If $a\in\mathcal A$, Proposition~\ref{prop:completeness38} proves that every query is soluble. If $a\notin\mathcal A$, Proposition~\ref{prop:queries} proves that some query is insoluble. If $\HTP_{38}(\Q)$ were decidable, run two procedures in parallel: enumerate the genuine integer witnesses of Lemma~\ref{lem:split-source} with $\varepsilon_+=1$, and decide successive rational queries, stopping in the second procedure when an insoluble query is found. In the affirmative case the first procedure terminates, and in the negative case the second procedure terminates. They cannot both terminate. This would decide membership in $\mathcal A$, a contradiction. Thus $\HTP_{38}(\Q)$ is undecidable.
\end{proof}

\appendix

\section{A geometric sieve for the actual order expressions}\label{an:app:large}
The general quadric theorem of Browning and Heath-Brown \cite[Theorem 1.2]{BHB} is stated for homogeneous auxiliary forms. The expressions $G_\eps,H_\eta$ are not homogeneous. Simply homogenizing them introduces another coordinate and does not give the following five-variable estimate. We instead verify directly the fibre conditions needed by their lattice method.

\begin{theorem}\label{an:thm:tail}
For $B,L\ge2$, let $N_{\kappa,\eps,\eta}(B,L)$ count integer vectors of supremum norm at most $B$ on \eqref{an:eq:quadric} for which some prime $\ell>L$ divides both $G_\eps$ and $H_\eta$. For every $\alpha>0$, we have
\begin{equation}\label{an:eq:tail}
 N_{\kappa,\eps,\eta}(B,L)
 \ll_{\kappa,\alpha}
 \frac{B^{3+\alpha}}{L\log L}+B^{11/4+\alpha}.
\end{equation}
\end{theorem}

We first specify the two published inputs to the proof. The polynomial box sieve \cite[Lemma 2.1]{BHB} bounds common prime divisors greater than $L$ of coprime, bounded-degree integral polynomials in a box by
\begin{equation}\label{an:eq:box}
 \ll \frac{V\log(VH)}{L\log L}
       +\frac{V\log(VH)}{B_{\min}},
\end{equation}
where the box is $\prod_i[-B_i,B_i]$ with $B_i\ge2$, $V=\prod_i B_i$, $B_{\min}=\min_i B_i$, and $H\ge2$ bounds coefficient heights. Here coprimality is in the integral polynomial ring, so a common constant prime factor must also be removed. The second input, \cite[Lemma 4.1]{BHB}, covers solutions of a rank-three quadratic congruence modulo a squarefree good modulus $q$ by $O(9^{\omega(q)}q)$ lattices of determinant $q^2$. The number whose largest successive minimum exceeds $T$ is $O(9^{\omega(q)}q^3T^{-3})$. Here $\omega(q)$ denotes the number of distinct prime divisors of $q$, rather than the slow threshold function of Section~\ref{an:sec:transfer}. A good prime is one at which the reduction of the quadratic form still has rank at least three. The successive minima are taken with respect to the Euclidean unit ball: the $i$th is the least radius containing $i$ independent lattice vectors. Each covering lattice contains $q\Z^3$, so all its successive minima are at most $q$.

\begin{proof}[Proof of Theorem~\ref{an:thm:tail}]
Sign changes in $a,c$ reduce to $\eps=\eta=1$. Make the integral linear change
\[
 x_0=a+c,\quad x_1=a-c,\quad x=b-d,\quad z=b+d.
\]
The equation becomes
\begin{equation}\label{an:eq:hyperbolic}
 x_0x_1=R(x,z,w):=\kappa w^2-xz.
\end{equation}
The new box has side parameters at most $2B$. Enlarging to it, and absorbing this factor in the constants, can only increase the count.

\smallskip\noindent
\emph{Primes dividing both hyperbolic coordinates.}
Suppose $\ell\mid x_0,x_1$. Then $R=\ell^2h$. Uniformly in a fixed integer $v$ in the range $|v|\ll_\kappa B^2$, the equation $R=v$ has $O_{\kappa,\alpha}(B^{1+\alpha})$ solutions in the box: for each $w$, count divisors in $xz=\kappa w^2-v$. We use the elementary bound $\tau(m)\ll_\varepsilon |m|^\varepsilon$ for $m\ne0$, where $\tau(m)$ is the number of positive divisors of $|m|$. It follows by multiplying the prime-power factors $a+1$: for all sufficiently large primes $p$, $a+1\le p^{\varepsilon a}$, and each of the finitely many smaller primes contributes a bounded extra factor. Its right side is zero for at most two $w$, and these cases contribute only $O(B)$. For nonzero $R$, the number of pairs $x_0x_1=R$ is $O_\alpha(B^\alpha)$. Summing over $\ell>L$ and the $O_\kappa(B^2/\ell^2)$ possible nonzero $h$ gives
\[
 O_{\kappa,\alpha}\left(\frac{B^{3+2\alpha}}{L\log L}\right).
\]
The case $R=0$ contributes $O_{\kappa,\alpha}(B^{2+\alpha})$. We may therefore treat the two charts $\ell\nmid x_1$ and $\ell\nmid x_0$ separately.

\smallskip\noindent
\emph{The chart with $t=a-c$.}
Fix $t\ne0$ and put $x_0=R/t$. A common divisor $\ell\nmid2t$ of $G_+$ and $H_+$ divides
\begin{equation}\label{an:eq:chartone}
 P_t=\kappa w^2+2t\ \ \t{and}\ \ T_t=(t^2+x^2)(t^2+z^2).
\end{equation}
To check this explicitly, $G_+-H_+=\kappa w^2+2t$ and
\[
 4t^2G_+=K_t:=(\kappa w^2-xz+t^2+2t)^2+t^2(x+z)^2.
\]
The polynomial identity
\begin{equation}\label{an:eq:identity}
 K_t-T_t=P_t(P_t+2t^2-2xz)
\end{equation}
proves the claim. The two polynomials in \eqref{an:eq:chartone} are coprime over $\Q$: one depends only on $w$, the other only on $x,z$. Also $T_t$ has coefficient 1 at $x^2z^2$, so the pair has no common constant prime factor.

\smallskip\noindent
\emph{The chart with $t=a+c$.}
Now $a-c=R/t$. The corresponding necessary divisibilities are
\begin{equation}\label{an:eq:charttwo}
 P'_t=\kappa(t+2)w^2-2xz\ \t{and}\ 
 K_t=(\kappa w^2-xz+t^2+2t)^2+t^2(x+z)^2,
\end{equation}
since $P'_t=t(G_+-H_+)$ and $K_t=4t^2G_+$. If $t\ne-2$, $P'_t$ is an irreducible rank-three quadratic form, whereas $K_t(0,0,0)=t^2(t+2)^2\ne0$. Hence they have no common polynomial factor. For $t=-2$, $P'_t=-2xz$, and neither $x$ nor $z$ divides $K_t$, as is seen by setting that variable to zero. Finally, the coefficient $-2$ in $P'_t$ rules out common constant odd primes.

\smallskip\noindent
\emph{The lattice summation.}
In either chart, $t\mid R(x,z,w)$. Let $q$ be the product of the good primes for $R$ which divide $t$; thus $q$ is squarefree and $q\ll B$. For each fixed $q$, the number of possible nonzero $t$ with $|t|\ll B$ is $O_{\kappa,\alpha}(B^\alpha)$. To see the uniformity in $q$, let $S_R$ be the finite set of bad primes. Since $t/q$ is supported on primes dividing $q$ or belonging to $S_R$, for any $\alpha>0$ a Rankin bound gives
\[
 \#\{t:0<|t|\le CB,\ \operatorname{rad}_{\rm good}(t)=q\}
 \le 2(CB/q)^\alpha
       \prod_{p\mid q\ {\rm or}\ p\in S_R}(1-p^{-\alpha})^{-1}.
\]
Here $\operatorname{rad}_{\rm good}(t)$ is the product of the good primes dividing $t$. The product is $O_{\kappa,\alpha}(q^\alpha)$: for all large primes its individual factor is at most $p^\alpha$, and the remaining finitely many factors have bounded product. This proves the asserted $O_{\kappa,\alpha}(B^\alpha)$ estimate.

Weaken $t\mid R$ to $q\mid R$ and use the lattice cover above. For a covering lattice $\Lambda$, let $\lambda_1\le\lambda_2\le\lambda_3$ be its successive minima. Its number of vectors in the box is
$ \ll B^3/q^2,$
because $\lambda_i\le q\ll B$ and Minkowski's second theorem gives $\lambda_1\lambda_2\lambda_3\asymp\det\Lambda=q^2$. The lattice-point bound and the choice of a basis of lengths $O(\lambda_i)$ below are the standard geometry-of-numbers estimates also used in \cite[Section 4]{BHB}. The contribution of lattices with $\lambda_3>T$ is consequently
\begin{equation}\label{an:eq:largemin}
 \ll_{\kappa,\alpha}
 B^\alpha\sum_{q\ll B}9^{\omega(q)}q^3T^{-3}\frac{B^3}{q^2}
 \ll_{\kappa,\alpha} B^{5+2\alpha}T^{-3}.
\end{equation}

For the other lattices choose an integral basis matrix $E$ with column lengths $O(\lambda_i)$. Its determinant is $q^2\asymp\lambda_1\lambda_2\lambda_3$. The cofactor formula bounds the $i$th row of $E^{-1}$ by $O(1/\lambda_i)$. Hence $(x,z,w)^t=E\mathbf y$ is contained in a coordinate box with side parameters chosen as fixed sufficiently large multiples of $B/\lambda_i$. Its volume parameter is $O(B^3/q^2)$ and its smallest side is $\gg B/T$. These sides are at least two in the range where the estimate is needed; bounded $B$ is absorbed into the final constant. The coefficient heights after substitution are bounded by a fixed power of $B$.

It remains to check that \eqref{an:eq:box} applies after substitution; this is precisely where homogeneity cannot simply be assumed. The explicit coprimality checks in the two charts persist under the invertible rational substitution $E$. In addition, we need to remove integer contents before invoking the quoted box-sieve lemma; restricting the \emph{obstruction} primes to $\ell\nmid2t$ by itself would not remove the lemma's content hypothesis.

For an integral polynomial $P$, write $\operatorname{cont}(P)$ for the positive gcd of its coefficients. In the first chart $\operatorname{cont}(T_t)=1$, since its $x^2z^2$ coefficient is one, and $\operatorname{cont}(P_t)\mid\gcd(\kappa,2t)$, since $P_t=\kappa w^2+2t$. In the second chart $\operatorname{cont}(K_t)=1$ (again its $x^2z^2$ coefficient is one) and $\operatorname{cont}(P'_t)\mid2$, since its $xz$ coefficient is $-2$. If a prime $\ell\nmid2t$ divides the content of one of the polynomials \emph{after} the substitution $E$, then $\ell\nmid\det E=q^2$. Reduction of $E$ is invertible modulo $\ell$, so this prime would already divide the content of the corresponding original polynomial, contradicting the coefficient calculations. Thus the contents of \emph{each} of the two substituted polynomials are supported on primes dividing $2t$.

Divide each substituted polynomial by its full positive content. Both quotients remain integral, have coefficient heights bounded by a fixed power of $B$, and are coprime over $\Q$ because $E$ is invertible over $\Q$. They also have no common constant prime factor, indeed each quotient is primitive as an integral polynomial. For any obstruction prime $\ell\nmid2t$, division by these contents is multiplication by an $\ell$-adic unit, so the condition that $\ell$ divide both original values is \emph{equivalent} to the corresponding condition for the primitive quotients. We may therefore apply \eqref{an:eq:box} to the latter without requiring an unquoted version of the box-sieve lemma with a variable excluded-prime set. The finitely many original coefficient exceptions are already incorporated in the fixed data, and the possible prime $2$ is excluded throughout.

Thus \eqref{an:eq:box} bounds each small-minimum lattice contribution by
\[
 \ll_{\kappa}
 \frac{B^3\log B}{q^2L\log L}
 +\frac{B^2T\log B}{q^2}.
\]
Summing over $O(9^{\omega(q)}q)$ lattices, the $O(B^\alpha)$ possible $t$, and $q\ll B$, and absorbing logarithms into arbitrarily small powers, gives
\begin{equation}\label{an:eq:smallmin}
 \ll_{\kappa,\alpha}
 \frac{B^{3+3\alpha}}{L\log L}+B^{2+3\alpha}T.
\end{equation}
Take $T=B^{3/4}$. Equations \eqref{an:eq:largemin}--\eqref{an:eq:smallmin} and the first case give \eqref{an:eq:tail} after renaming the arbitrary exponent. The possible prime 2 is absent because $L\ge2$.
\end{proof}

\begin{corollary}\label{an:cor:weightedlarge}
Fix $0<\delta<1/4$ and $k\ge0$. On the set in Theorem~\ref{an:thm:tail}, attach any nonnegative weights bounded by $(\log B)^k$. The weighted contribution from a common prime divisor greater than $B^\delta$ is $O_{\kappa,\delta,k}(B^{3-\delta/2}).$
In particular the conclusion allows restricting to a prime pattern and multiplying by its fixed number of prime logarithmic weights.
\end{corollary}
\begin{proof}
Use \eqref{an:eq:tail} with $L=B^\delta$ and a sufficiently small $\alpha$, then absorb the fixed logarithmic power into $B^{\delta/4}$. Restricting the set only decreases this upper bound.
\end{proof}

The self-gcd tail is simpler. An odd prime dividing both orders in \eqref{an:eq:orders} divides $a$ and $b^2+1$. Apply \eqref{an:eq:box} to these two coprime polynomials in the three variables $a,b,w$. The number of such triples with an obstruction prime greater than $L$ is
\[
 \ll \frac{B^3\log B}{L\log L}+B^2\log B.
\]
This likewise gives a power saving at $L=B^\delta$, including fixed prime logarithmic weights.

\end{document}